\newif\ifprint
\printfalse
\documentclass[twoside,10pt]{HandbookOfModuli}

\usepackage{amsmath,amsthm}
\usepackage{amssymb}
\usepackage{latexsym}
\usepackage[utf8]{inputenc}
\usepackage{textcomp}
\usepackage{color}
\usepackage[pdftex]{graphicx}
\usepackage{titlesec}
\usepackage{xspace}

\ifprint
	\input{giovanni} 
	\usepackage[final]{microtype}
\fi
\usepackage{bm}
\renewcommand{\mathbf}[1]{\bm{#1}} 

\ifprint
	\definecolor{linkred}{rgb}{0,0,0} 
	\definecolor{linkblue}{rgb}{0,0,0} 
\else
	\definecolor{linkred}{rgb}{0.5,0.1,0.1}
	\definecolor{linkblue}{rgb}{0,0.1,0.3}
\fi

\numberwithin{equation}{section} 

\usepackage[
	hypertexnames=false,
	hyperindex,
	pagebackref,
	pdftex,
	pdftitle={Handbook of Moduli},
	pdfdisplaydoctitle,
	pdfpagemode=UseNone,
	breaklinks=true,
	extension=pdf,
	bookmarks=false,
	plainpages=false,
	colorlinks,
	linkcolor=linkblue,
	citecolor=linkred,
	urlcolor=linkred,
	pdfmenubar=true,
	pdftoolbar=true,
	pdfpagelabels,
	pdfpagelayout=TwoPage,
	pdfview=Fit,
	pdfstartview=Fit
]{hyperref}

\catcode`@=11 \baselineskip 4.5mm
\def\ps@handbook{\def\@oddhead{\hfill \leftmark \hfill\thepage }
\def\@evenhead{\thepage \hfill \rightmark \hfill}
\def\@oddfoot{}
\def\@evenfoot{}}
\def\@evenhead{}
\def\@oddfoot{}
\def\@evenfoot{\hfill\copyright\ China Higher Education Press}
\def\list#1#2{\ifnum \@listdepth >5\relax \@toodeep \else \global
\advance \@listdepth\@ne \fi \rightmargin \z@ \listparindent\z@
\itemindent\z@ \csname @list\romannumeral\the\@listdepth\endcsname
\def\@itemlabel{#1}\let\makelabel\@mklab \@nmbrlistfalse #2\relax
\@trivlist \parskip -\parsep \parindent\listparindent \advance
\linewidth -\rightmargin \advance\linewidth -\leftmargin \advance
\@totalleftmargin \leftmargin \parshape \@ne \@totalleftmargin
\linewidth \ignorespaces}
\renewcommand*\l@section{\@tocline{1}{0pt}{0em}{1em}{}}
\renewcommand*\l@subsection{\@tocline{2}{0pt}{1.5em}{2em}{}} 
\renewcommand\contentsnamefont{\bfseries\large} 
\catcode`@=12
\renewcommand{\theequation}{\thesection.\arabic{equation}}

\def\thebibliography#1{\section*{References}
\list{[\arabic{enumi}]}{\settowidth \labelwidth{[#1]} \leftmargin
\labelwidth \advance \leftmargin \labelsep \usecounter{enumi}}
\def\newblock{\hskip .11em plus .33em minus .07em} \sloppy
\clubpenalty 4000 \widowpenalty 4000 \sfcode`\.=1000 \relax}

\titleformat{\section}{\normalfont\large\bfseries}{\thesection.}{0.5em}{}[\kern0.em]
\titleformat{\subsection}{\normalfont\bfseries}{\thesubsection.}{0.3em}{}[\kern0.em]
\titleformat{\subsubsection}[runin]{\normalfont\bfseries}{\thesubsubsection.}{0.5em}{}[\kern0.5em]

\def\fofsubsubsection#1{\refstepcounter{equation}\subsubsection*{\theequation.\kern0.25em #1}}
\def\foisubsubsection#1{\refstepcounter{equation}\subsubsection*{\kern\parindent\theequation.\kern0.25em #1}}

\newtheorem{theorem}[equation]{Theorem}
\newtheorem{proposition}[equation]{Proposition}
\newtheorem{problem}[equation]{Problem}

\newtheorem{definition}[equation]{Definition}

\newtheorem{lemma}[equation]{Lemma}
\newtheorem{remark}[equation]{Remark}
\newtheorem{corollary}[equation]{Corollary}
\newtheorem{conjecture}[equation]{Conjecture}
\newtheorem{question}[equation]{Question}
\newtheorem{principle}[equation]{Principle}
\theoremstyle{definition}
\newtheorem{example}[equation]{Example}
\newtheorem{exercise}[equation]{Exercise}

\usepackage{latexsym,amssymb,graphicx,multicol,mathrsfs,color,xypic,amscd, amsxtra, verbatim}
\usepackage[cmtip,all]{xy}
\usepackage{psfrag}
\usepackage[version=0.96]{pgf}
\usepackage{tikz}
\usetikzlibrary{arrows}

\usepackage{array, multirow}

\DeclareMathOperator{\SL}{SL}

\DeclareMathOperator{\Chow}{Chow}
\newcommand{\gitq}{/\hspace{-0.25pc}/}
\newcommand{\s}{\mathrm{s}}
\renewcommand{\ss}{\mathrm{ss}}

\def\co{\colon\thinspace} 
\def\A{\mathcal{A}}

\def\Aut{\text{Aut}}
\def\rst{\text{r}}

\def\C{\mathcal{C}}

\DeclareMathOperator{\Exc}{Exc}

\DeclareMathOperator{\Eff}{Eff}
\DeclareMathOperator{\PEff}{\overline{Eff}}

\DeclareMathOperator{\Amp}{Amp}
\DeclareMathOperator{\Nef}{Nef}
\DeclareMathOperator{\N}{N}

\DeclareMathOperator{\Pic}{Pic}

\DeclareMathOperator{\eff}{eff}

\DeclareMathOperator{\Hilb}{Hilb}

\def\L{\mathcal{L}}
\def\M{\overline{M}}

\def\SM{\overline{\mathcal{M}}}

\def\O{\mathcal{O}}

\def\T{\mathcal{T}}

\def\P{\mathbb{P}}
\def\Q{\mathbb{Q}}
\def\X{\mathcal{X}}
\def\U{\mathcal{U}}

\def\S{\mathcal{S}}

\def\Spec{\text{\rm Spec\,}}
\def\Proj{\text{\rm Proj\,}}
\def\Supp{\text{\rm Supp\,}}
\def\Mor{\text{\rm Mor}}
\def\Pic{\text{\rm Pic\,}}

\def\pm{\{p_i\}_{i=1}^{m}}
\def\qm{\{q_i\}_{i=1}^{m}}
\def\pn{\{p_i\}_{i=1}^{n}}
\def\pnprime{\{p_i'\}_{i=1}^{n}}
\def\sigman{\{\sigma_{i}\}_{i=1}^{n}}

\def\sigmanprime{\{\sigma'_{i}\}_{i=1}^{n}}

\def\S{\mathcal{S}}

\def\Spec{\text{\rm Spec\,}}
\def\Proj{\text{\rm Proj\,}}
\def\Supp{\text{\rm Supp\,}}
\def\Mor{\text{\rm Mor}}
\def\Pic{\text{\rm Pic\,}}

\def\ra{\rightarrow}

\def\co{\colon\thinspace} 

\DeclareMathOperator{\Def}{Def}

\DeclareMathOperator{\BIG}{Big}
\DeclareMathOperator{\spec}{Spec}
\DeclareMathOperator{\proj}{Proj}

\DeclareMathOperator{\Bl}{Bl}
\DeclareMathOperator{\Sym}{Sym}
\DeclareMathOperator{\codim}{codim}

\DeclareMathOperator{\pr}{pr}

\def\Hn1{\mathcal{H}_{n,1}}

\renewcommand\S{\mathcal{S}}

\DeclareMathOperator{\HH}{\mathrm{H}}

\def\irr{\text{irr}}

\def\A{\mathcal{A}}
\def\Abar{\overline{\A}}
\def\C{\mathcal{C}}

\def\O{\mathcal{O}}

\newcommand\Tl[1]{\mathcal{T}_{#1}}
\newcommand\Mg[1]{\overline{\mathcal{M}}_{#1}}

\def\L{\mathcal{L}}

\def\X{\mathcal{X}}
\def\Y{\mathcal{Y}}

\def\T{\mathcal{T}}

\def\AA{\mathbb{A}}
\def\QQ{\mathbb{Q}}
\def\ZZ{\mathbb{Z}}

\def\PP{\mathbb{P}}
\def\ZZ{\mathbb{Z}}

\def\VV{\mathbb{V}}

\def\GG{\mathbb{G}}

\def\nb{\nobreakdash}

\begin{document}
\setcounter{page}{1}
%
%
\long\def\replace#1{#1}

%
%
\title[Moduli of Curves]{Alternate Compactifications of Moduli Spaces of Curves}
%
%
\author{Maksym Fedorchuk}
\address[Fedorchuk]{Department of Mathematics\\
Columbia University\\
2990 Broadway\\
New York, NY 10027}
\email{mfedorch@math.columbia.edu}
\author{David Ishii Smyth}
\address[Smyth]{Department of Mathematics\\
Harvard University\\
One Oxford Street\\
Cambridge, MA 01238}
\email{dsmyth@math.harvard.edu}

%
%

\begin{abstract}
	
We give an informal survey, emphasizing examples and open problems, of two interconnected research programs in moduli of curves: the systematic classification of modular compactifications of $M_{g,n}$, and the study of Mori chamber decompositions of $\M_{g,n}$.

\end{abstract}  

\maketitle
\tableofcontents
\thispagestyle{empty}
%
%

\section{Introduction}
The purpose of this article is to survey recent developments in two interconnected research programs within the study of moduli of curves: the systematic construction and classification of modular compactifications of $M_{g,n}$, and the study of the Mori theory of $\M_{g,n}$. In this section, we will informally describe the overall goal of each of these programs as well as outline the contents of this article.

One of the most fundamental and influential theorems in modern algebraic geometry is the construction of a 
modular compactification $M_{g} \subset \M_{g}$ for the moduli space of smooth complete curves of genus $g$ \cite{DM}. In this context, the word \emph{modular} indicates that the compactification $\M_{g}$ is itself a parameter space for a suitable class of geometric objects, namely stable curves of genus $g$.
\begin{definition}[Stable curve]\label{D:StableCurve} A complete curve $C$ is \emph{stable} if
\begin{itemize}
\item[(1)]$C$ has only nodes $(y^2=x^2)$ as singularities.
\item[(2)]$\omega_{C}$ is ample.
\end{itemize}
\end{definition}
The class of stable curves satisfies two essential properties, namely deformation openness (Definition \ref{D:DefOpen}) and the unique limit property (Definition \ref{D:ULP}). As we shall see in Theorem \ref{T:Construction}, any class of curves satisfying these two properties gives rise to a modular compactification of $M_{g}$.

While the class of stable curves gives a natural modular compactification of the space of smooth curves, it is not unique in this respect. The simplest example of an alternate class of curves satisfying these two properties is the class of pseudostable curves, introduced by Schubert \cite{Schubert}.

\begin{definition}[Pseudostable curve]\label{D:pseudostable} A curve $C$ is \emph{pseudostable} if
\begin{itemize}
\item[(1)] $C$ has only nodes ($y^2=x^2$) and cusps ($y^2=x^3$) as singularities.
\item[(2)] $\omega_{C}$ ample.
\item[(3)] If $E \subset C$ is any connected subcurve of arithmetic genus one,\\ then $|E \cap \overline{C \setminus E}|  \geq 2.$
\end{itemize}
\end{definition}
Notice that the definition of pseudostability involves a trade-off: cusps are exchanged for elliptic tails. It is easy to see how this trade-off comes about: As one ranges over all one-parameter smoothings of a cuspidal curve $C$, the associated limits in $\M_{g}$ are precisely curves of the form $\tilde{C} \cup E$, where $\tilde{C}$ is the normalization of $C$ at the cusp and $E$ is an elliptic curve (of arbitrary $j$-invariant) attached to $\tilde{C}$ at the point lying above the cusp (see Example \ref{E:stable-reduction-A2k}). Thus, any separated moduli problem must exclude either elliptic tails or cusps. Schubert's construction naturally raises the question:

\begin{problem}\label{P:classify-stability-conditions}
Classify all possible stability conditions for curves, i.e. classes of possibly singular 
curves which are deformation
open and satisfy the property that any one-parameter family of smooth curves contains a unique limit contained in that
class.
\end{problem} 

Of course, we can ask the same question with respect to classes of marked curves, i.e. modular compactifications of $M_{g,n}$, and the first goal of this article is to survey the various methods used to construct alternate modular compactifications of $M_{g,n}$. In Section \ref{S:compactifying}, we describe constructions using Mori theory, the combinatorics of curves on surfaces, and geometric invariant theory, and, in Section \ref{S:Classification}, we assemble a (nearly) comprehensive list of the alternate modular compactifications that have so far appeared in the literature.

Since each stability condition for $n$-pointed curves of genus $g$ gives rise to a proper birational model of $\M_{g,n}$, it is natural to wonder whether these models can be exploited to study the divisor theory of $\M_{g,n}$. This brings us to the second subject of this survey, namely the study of the birational geometry of $\M_{g,n}$ from the perspective of Mori theory. In order to explain the overall goal of this program, let us recall some definitions from birational geometry. For any normal, $\Q$-factorial, projective variety $X$, let $\Amp(X)$, $\Nef(X)$, and $\Eff(X)$ denote the cones of ample, nef, and effective divisors respectively (these are defined in Section \ref{S:birational-geometry}). Recall that a \emph{birational contraction} $\phi\co X \dashrightarrow Y$ is simply a birational map to a normal, projective, $\Q$\nb-factorial variety such that $\codim(\Exc(\phi^{-1})) \geq 2$. If $\phi$ is any birational contraction with exceptional divisors $E_1, \ldots, E_k$, then the \emph{Mori chamber of }$\phi$ is defined as:
$$
\Mor(\phi):=\phi^*\Amp(Y) + \sum_{a_i \geq 0}a_iE_i \subset \Eff(X).
$$
The birational contraction $\phi$ can be recovered from any divisor $D$ in $\Mor(\phi)$ as the contraction associated to the section ring of $D$, i.e. $X \dashrightarrow \Proj \bigoplus\limits_{m \geq 0}\HH^0(X,mD)$. In good cases, e.g. when $X$ is a toric or Fano variety, these Mori chambers are polytopes and there exist finitely many contractions $\phi_i$ which partition the entire effective cone, i.e. one has
$$
\PEff(X) = \overline{\Mor(\phi_1)} \cup \ldots \cup \overline{\Mor(\phi_k)}.
$$
Understanding the Mori chamber decomposition of the effective cone of a variety $X$ is therefore tantamount to understanding the classification of 
birational contractions of $X$. This collection of ideas, which will be amplified in Section \ref{S:birational-geometry}, suggests the following problems concerning $\M_{g,n}$.

\begin{problem}\label{P:birational-problems}
\begin{enumerate}
\item[]
\item Describe the nef and effective cone of $\M_{g,n}$.
\item Describe the Mori chamber decomposition of $\overline{\Eff}(\M_{g,n})$.
\end{enumerate}
\end{problem}

These questions merit study for two reasons: First, the attempt to prove results concerning the geometry of $\M_{g,n}$ has often enriched our understanding of algebraic curves in fundamental ways. Second, even partial results obtained along these lines may provide an interesting collection of examples of phenomena in higher-dimensional geometry. In Section \ref{S:nef-cone}, we will briefly survey the present state of knowledge concerning the nef and effective cones of 
$\M_{g,n}$, and many examples of partial Mori chamber decompositions will be given in Section \ref{S:log-MMP}. We should emphasize that Problem \ref{P:birational-problems} (1) has generated a large body of work over the past 30 years. While we hope that Section \ref{S:nef-cone} can serve as a reasonable guide to the literature, we will not cover these results in the detail they merit. The interested reader who wishes to push further in this direction is strongly encouraged to read the recent surveys of Farkas \cite{farkas-seattle, farkas-aspects} and Morrison \cite{morrison-mori}.  

A major focus of the present survey is the program initiated by Hassett and Keel to study certain log canonical models of $\M_{g}$ and $\M_{g,n}$.
For any rational number $\alpha$ such that $K_{\SM_{g,n}}+\alpha\delta$ is effective, they define
\begin{align*}
\M_{g,n}(\alpha)= \Proj \bigoplus_{m \geq 0} \HH^0\bigl(\SM_{g,n}, m (K_{\SM_{g,n}}+\alpha\delta)\bigr),
\end{align*}
where the sum is taken over $m$ sufficiently divisible, and ask whether the spaces $\M_{g,n}(\alpha)$ admit a modular description. When $g$ and $n$ are sufficiently large so that $\M_{g,n}$ is of general type, $\M_{g,n}(1)$ is simply $\M_{g,n}$, while $\M_{g,n}(0)$ is the canonical model of $\M_{g,n}$. Thus, the series of birational transformations accomplished as $\alpha$ decreases from 1 to 0 constitute a complete minimal model program for $\M_{g,n}$. In terms of Problem \ref{P:birational-problems} above, this program can be understood as asking:
\begin{enumerate}
\item Describe the Mori chamber decomposition within the two dimensional subspace of $\Eff(\M_{g,n})$ spanned by $K_{\SM_{g,n}}$ and $\delta$ (where it is known to exist by theorems of higher-dimensional geometry \cite{BCHM}).
\item Find modular interpretations for the resulting birational contractions.
\end{enumerate}
Though a complete description of the log minimal model program for $\M_{g}$ ($g\gg 0$) is still out of reach, the first two steps have been carried out
by Hassett and Hyeon \cite{HH1, HH2}. Furthermore, the program has been completely carried out for $\M_{2}$, $\M_{3}$, 
$\M_{0,n}$, and $\M_{1,n}$ \cite{Hgenus2, HL, FedSmyth, AS, SmythEI, SmythEII}, and we will describe the resulting chamber decompositions in Section \ref{S:log-MMP}. 
Finally, in Section \ref{S:heuristic} we will present some informal heuristics for making predictions about future stages of the program.

In closing, let us note that there are many fascinating topics in the study of birational geometry of $\M_{g,n}$ and related spaces which receive little or no attention whatsoever in the present survey. These include  
rationality, unirationality and rational connectedness of $\M_{g,n}$ \cite{barron-4, barron-6, katsylo-5, katsylo-3, mori-mukai-11, verra14, bruno-verra},
complete subvarieties of $M_{g,n}$ \cite{diaz, faber-complete, krichever}, birational geometry of Prym varieties \cite{donagi-prym, catanese-prym, verra-prym-6, verra-prym-5, katsylo-prym, dolgachevR23, izadi-etal, farkas-prym}, Kodaira dimension of $\M_{g,n}$ for $n\geq 1$ \cite{logan-kodaira, farkas-koszul} and of the universal Jacobian over $\M_g$ \cite{farkas-verra-jacobian}, moduli spaces birational 
to complex ball quotients \cite{kondo-3, kondo-4}, etc.

A truly comprehensive survey of the birational geometry of $\M_{g,n}$ should certainly contain a discussion of these studies, and these omissions reflect nothing more than (alas!) the inevitable limitations of space, time, and our own knowledge.

\subsection{Notation}\label{S:notation}

Unless otherwise noted, we work over a fixed algebraically closed field $k$ of characteristic zero.

A \emph{curve} is a complete, reduced, one-dimensional scheme of finite-type over an algebraically closed field. An \emph{$n$-pointed curve $(C, \pn)$} is a curve $C$, together with $n$ distinguished smooth 
points $p_1, \ldots, p_n \in C$. A \emph{family of $n$-pointed curves} $(f\co \C \rightarrow T, \sigman)$ is a flat, proper, finitely-presented morphism $f\co \C \rightarrow T$, together with $n$ sections $\sigma_1, \ldots, \sigma_n$, such that the geometric fibers are $n$-pointed curves, and with the total space $\C$ an algebraic
space over $k$. A \emph{class of curves} $\S$ is simply a set of isomorphism classes of $n$-pointed curves of arithmetic genus $g$ (for some fixed $g,n \geq 0$), which is invariant under base extensions, i.e. if $k \hookrightarrow l$ is an inclusion of algebraically closed fields, then a curve $C$ over $k$ is in a class $\S$ iff $C \times_k l$ is. If $\S$ is a class of curves, a \emph{family of $\S$-curves} is a family of $n$-pointed curves $(f\co \C \rightarrow T, \sigman)$ whose geometric fibers $(C_{\overline{t}},\{\sigma_i(\overline{t})\}_{i=1}^{n})$ are contained in $\S$.

$\Delta$ will always denote the spectrum of a discrete valuation ring $R$ with algebraically closed residue field $l$ and field of fractions $L$. When we speak of a finite base change $\Delta' \rightarrow \Delta$, we mean that $\Delta'$ is the spectrum of a discrete valuation ring $R' \supset R$ with field of fractions $L'$, where $L' \supset L$ is a finite separable extension. We use the notation
\begin{align*}
0&:=\Spec l \rightarrow \Delta,\\
\Delta^*&:=\Spec L \rightarrow \Delta,
\end{align*}
for the closed point and generic point respectively. The complex-analytically-minded reader will lose nothing by thinking of $\Delta$ as the unit disc in $\mathbb{C}$, and $\Delta^*$ as the punctured unit disc.

While a nuanced understanding of stacks should not be necessary for reading these notes, we will assume the reader is comfortable with the idea that moduli stacks represent functors. If the functor of families of $\S$-curves is representable by a stack, this stack will be denoted in script, e.g. $\SM_{g,n}[\S]$. If this stack has a coarse 
(or good) moduli space, it will be noted in regular font, e.g. $\M_{g,n}[\S]$. We call $\SM_{g,n}[\S]$ the \emph{moduli stack of $\S$-curves} and $\M_{g,n}[\S]$ the \emph{moduli space of $\S$-curves}.

Starting in Section \ref{S:birational-geometry-moduli}, we will assume the reader is familiar with intersection on $\Mg{g,n}$ and $\M_{g,n}$, as described in \cite{HM} (the original reference for the theory 
is \cite[Section 6]{HMKodaira}).  
The key facts which we use repeatedly are these: For any $g, n$ satisfying $2g-2+n > 0$, there are canonical identifications
$$\N^1(\M_{g,n})\otimes \QQ=\Pic(\M_{g,n}) \otimes \QQ=\Pic(\SM_{g,n})\otimes \QQ,$$
so we may consider any line bundle on $\SM_{g,n}$ as a numerical divisor class on $\M_{g,n}$ without ambiguity. Furthermore, by 
\cite{AC-picard} and \cite{AC}, $\N^1(\M_{g,n})\otimes \QQ$ is generated by the classes of $\lambda, \psi_1, \ldots, \psi_n$, and the boundary divisors $\delta_{0}$ and $\delta_{h,S}$, which are defined as follows. If $\pi\co \Mg{g,n+1} \ra \Mg{g,n}$ is the universal curve
with sections $\sigman$, then
$\lambda=c_1(\pi_*\omega_\pi)$ and $\psi_i=\sigma_i^*(\omega_{\pi})$. Let $\Delta=\Mg{g,n}\smallsetminus \mathcal{M}_{g,n}$ be the boundary with class $\delta$. Nodal irreducible curves form a divisor $\Delta_{0}$ whose class we denote 
$\delta_{0}$. For any $0\leq h\leq g$ 
and $S\subset \{1,\dots, n\}$, the boundary divisor 
$\Delta_{h, S}\subset \Mg{g,n}$ is the locus of reducible curves with a node separating 
a connected component of arithmetic genus $h$ marked by the points in $S$ from the rest of the curve, 
and we use $\delta_{h,S}$ to denote its class. 
Clearly, $\Delta_{h,S}=\Delta_{g-h, \{1,\dots, n\}\smallsetminus S}$ and $\Delta_{h,S}=\varnothing$ if $h=0$ and $|S|\leq 1$. When $g\geq 3$, these are the only relations
and $\Pic(\Mg{g,n})$ is a free abelian group generated by $\lambda, \psi_1, \ldots, \psi_n$ and the boundary divisors \cite[Theorem 2]{AC-picard}. We refer 
the reader to \cite[Theorem 2.2]{AC} for the list of relations in $g\leq 2$. 

Finally, the canonical divisor of $\Mg{g,n}$ is given by the formula
\begin{equation}\label{E:canonical}
K_{\Mg{g,n}}=13\lambda-2\delta+\psi.
\end{equation}
This is a consequence of the Grothendieck-Riemann-Roch formula applied to the cotangent bundle of $\Mg{g,n}$
\[
\Omega_{\Mg{g,n}}=\pi_*\bigl(\Omega_{\pi}(-\sum_{i=1}^n\sigma_i)\otimes \omega_{\pi}\bigr).
\]
The computation in the unpointed case is in \cite[Section 2]{HMKodaira}, and the computation in the pointed case follows similarly. Alternately, the formula for $\Mg{g,n}$ with $n\geq 1$ can be deduced directly from that of $\Mg{g}$ as in \cite[Theorem 2.6]{logan-kodaira}. Note that we will often use $K_{\Mg{g,n}}$ to denote a divisor class on $\M_{g,n}$ using the canonical identification indicated above.

\section{Compactifying the moduli space of curves}\label{S:compactifying}

\subsection{Modular and weakly modular birational models of $M_{g,n}$}\label{S:modular-models}

What does it mean for a proper birational model $X$ of $M_{g,n}$ to be modular? 
Roughly speaking, $X$ should coarsely represent a geometrically defined functor of 
$n$\nb-pointed curves. To make this precise, we introduce 
$\U_{g,n}$, the moduli stack of all $n$\nb-pointed curves (hereafter, we always assume $2g-2+n>0$).
More precisely, if we define a functor from schemes over $k$ to groupoids by
\begin{equation*}
\U_{g,n}(T):=
\left\{
\begin{matrix}
\text{ Flat, proper, finitely-presented morphisms $\C \rightarrow T$, with $n$ sections}\\
\text{$\{ \sigma_i\}_{i=1}^{n}$, and connected, reduced, one-dimensional geometric fibers}\\
\end{matrix}
\right\}
\end{equation*}
then $\U_{g,n}$ is an algebraic stack, locally of finite type over $\Spec k$ \cite[Theorem B.1]{Z-stability};
note that we allow $\C$ to be an algebraic space over $k$.
Now we may make the following definition.
\begin{definition}\label{D:Modular}
Let $X$ be a proper birational model of $M_{g,n}$. We say that $X$ is \emph{modular} if there exists a diagram
\[
\xymatrix{
\U_{g,n}&\\
\X \ar[u]^{i} \ar[r]_{\pi}&X\\
}
\]
satisfying:
\begin{enumerate}
\item $i$ is an open immersion,
\item  $\pi$ is a coarse moduli map.
\end{enumerate}
\end{definition}

Recall that if $\X$ is any algebraic stack, we say that $\pi\co \X \rightarrow X$ is a \emph{coarse moduli map} if $\pi$ satisfies the following properties:
\begin{enumerate}
\item $\pi$ is categorical for maps to algebraic spaces, i.e. if $\phi\co \X \rightarrow Y$ is any map from $\X$ to an algebraic space, then $\phi$ factors through $\pi$.
\item $\pi$ is bijective on $k$-points.
\end{enumerate}
For a more complete discussion of the properties enjoyed by the coarse moduli space of an algebraic stack, we refer the reader to \cite{keel-mori}.

The key point of Definition \ref{D:Modular} is that modular birational models of $M_{g,n}$ may be constructed purely by considering geometric properties of curves. 
More precisely, if $\S$ is any class of curves which is \emph{deformation open} and satisfies the \emph{unique limit property}, then there is an associated modular birational model $\M_{g,n}[\S]$. The precise definition of these two conditions is given below:

\begin{definition}[Deformation open]\label{D:DefOpen} A class of curves $\S$ is \emph{deformation open} if the following condition holds: Given any family of curves $(\C \rightarrow T, \sigman)$, the set $\{ t \in T \,:\, (C_{\overline{t}}, \{\sigma_i(\overline{t})\}_{i=1}^{n}) \in \S \}$ is open in $T$.
\end{definition}

\begin{definition}[Unique limit property]\label{D:ULP}
A class of curves $\S$ satisfies the \emph{unique limit property} if the following two conditions hold:
\begin{itemize}
\item[(1)] (Existence of $\S$-limits) Let $\Delta$ be the spectrum of a discrete valuation ring. If $(\C\ra \Delta^*,\sigman)$ is a family of $\S$-curves,
 there exists a finite base-change $\Delta' \rightarrow \Delta$, and a family of $\S$-curves $(\C' \rightarrow \Delta', \sigmanprime)$, such that
$$(\C', \sigmanprime)|_{(\Delta')^*} \simeq (\C,\sigman)\times_{\Delta} (\Delta')^*.$$
\item[(2)] (Uniqueness of $\S$-limits) Let $\Delta$ be the spectrum of a discrete valuation ring. Suppose that $(\C \rightarrow \Delta,\sigman)$  and $(\C' \rightarrow \Delta, \sigmanprime)$ are two families of $\S$-curves. Then any isomorphism over the generic fiber
$$(\C,\sigman)|_{\Delta^*} \simeq (\C',\sigmanprime)|_{\Delta^*}$$ extends to an isomorphism over $\Delta$:
$$(\C,\sigman) \simeq (\C',\sigmanprime).$$
\end{itemize}
\end{definition}

\begin{theorem}\label{T:Construction} Let $\S$ be a class of curves which is deformation open and satisfies the unique limit property. Then there exists a 
proper Deligne-Mumford stack of $\S$\nb-curves $\SM_{g,n}[\S]$ and an associated coarse moduli space $\M_{g,n}[\S]$, which is a modular birational model of $M_{g,n}$.
\end{theorem}
\begin{proof}
Set
$$
\SM_{g,n}[\S]:=\{t \in \U_{g,n}  \ | \ (C_{\overline{t}}, \{\sigma_i(\overline{t})\}_{i=1}^{n}) \in \S \} \subset \U_{g,n}.
$$
The fact that $\S$ is deformation open implies that $\SM_{g,n}[\S] \subset \U_{g,n}$ is open, so $\SM_{g,n}[\S]$ automatically inherits the structure of an algebraic stack over $k$. 

Uniqueness of $\S$\nb-limits implies that the automorphism group of each curve $\left[C,\pn\right]\in \SM_{g,n}[\S]$ is proper. If $\widetilde{C} \ra C$ is 
the normalization and $\qm$ are points lying over the singularities of $C$, then we have an injection $\Aut(C,\pn) \hookrightarrow \Aut(\widetilde{C}, \qm, \pn)$.
Since none of the components of $\widetilde{C}$ is an unpointed genus $1$ curve, the latter group is affine. 
We conclude that $\Aut(C,\pn)$ is finite, hence unramified (since we are working in characteristic $0$). It follows by \cite[Th\'{e}or\`{e}me 8.1]{LM} that $\SM_{g,n}[\S]$ is a Deligne-Mumford stack. 
The unique limit property implies that $\SM_{g,n}[\S]$ is proper, via the valuative criterion \cite[Proposition 7.12]{LM}. Finally, $\SM_{g,n}[\S]$ has a coarse moduli space $\M_{g,n}[\S]$ by \cite[Corollary 1.3]{keel-mori}, and it follows immediately that $\M_{g,n}[\S]$ is a modular birational model of $M_{g,n}$.
\end{proof}

While most of the birational models of $M_{g,n}$ that we construct will be modular in the sense of Definition \ref{D:Modular}, geometric invariant theory constructions (Section \ref{S:GIT}) give rise to models which are modular in a slightly weaker sense. Roughly speaking, these models are obtained from mildly non-separated functors of curves by imposing sufficiently many identifications to force a separated moduli space.

To formalize this, recall that if $\X$ is an algebraic stack over an algebraically closed field $k$, it is possible for one $k$\nb-point of $\X$ to be in the closure of another. In the context of moduli stacks of curves, one can visualize this via the valuative criterion of specialization, i.e. if $[C,p]$ and $[C',p']$ are two points of $\U_{g,n}$, then $[C',p'] \in \overline{[C,p]}$ if and only if there exists an isotrivial specialization of $(C,p)$ to $(C',p')$, i.e. a family $\C \rightarrow \Delta$ whose fibers over the punctured disc are isomorphic to $(C,p)$ and whose special fiber is isomorphic to  $(C',p')$.

\begin{exercise}\label{E:M11}
Let $\U_{1,1}$ be the stack of all $1$\nb-pointed curves of arithmetic genus one. Let $[C,p] \in \U_{1,1}$ be the unique isomorphism class of a rational cuspidal curve with a smooth marked point. Show that if $(E, p) \in \U_{1,1}$ is a smooth curve, then we have $[C,p] \in \overline{[E,p]}$. 
(Hint: Consider a trivial family $E\times \Delta\ra \Delta$, blow-up the point $p$ in $E\times 0$, 
and contract the strict transform of $E\times 0$. See also Lemma \ref{L:Contraction} and Exercise \ref{E:Cusp} below.)
\end{exercise}

Now if $\X$ is an algebraic stack with a $k$-point in the closure of another, it is clear (for topological reasons) that $\X$ cannot have a coarse moduli map. However, if we let $\sim$ denote the equivalence relation generated by the relations $x \in \overline{y}$, then we might hope for a map $\X \rightarrow X$ such that the points of $X$ correspond to equivalence classes of points of $\X$. We formalize this as follows: If $\X$ is any algebraic stack, we say that $\pi\co \X \rightarrow X$ is a \emph{good moduli map} if $\pi$ satisfies the following properties:
\begin{enumerate}
\item $\pi$ is categorical for maps to algebraic spaces.
\item $\pi$ is bijective on closed $k$-points, i.e. if $x \in X$ is any $k$ point, then $\phi^{-1}(x) \subset \X$ contains a unique closed $k$-point.
\end{enumerate}
Here, a closed $k$-point is, of course, simply a point $x \in \X$ satisfying $\overline{x}=x$.
The notion of a good moduli space was introduced and studied by Alper \cite{alper}. 
We should note that, for ease of exposition, we have adopted a slightly weaker definition than that appearing in \cite{alper}. Now we may define
\begin{definition}\label{D:weakly-modular}
Let $X$ be a proper birational model of $M_{g,n}$. We say that $X$ is \emph{weakly modular} if there exists a diagram
\[
\xymatrix{
\U_{g,n}&\\
\X \ar[u]^{i} \ar[r]_{\pi}&X\\
}
\]
satisfying:
\begin{enumerate}
\item $i$ is an open immersion,
\item  $\pi$ is a good moduli map.
\end{enumerate}
\end{definition}

\begin{exercise}[4 points on $\P^1$]\label{E:4points}
Consider the moduli space $M_{0,4}$ of 4 distinct ordered points on $\PP^1$, up to isomorphism. Of course, $M_{0,4} \simeq \PP^1\smallsetminus \{0,1,\infty\}$ 
where the isomorphism is given by the classical cross-ratio $$(\PP^{1}, \{p_i\}_{i=1}^{4}) \mapsto (p_1-p_3)(p_2-p_4)/(p_1-p_2)(p_3-p_4).$$
In this exercise, we will see that $\PP^1$ can be considered as a weakly modular birational model of $M_{0,4}$.

The most na\"ive way to enrich the moduli functor of $4$\nb-pointed $\PP^{1}$'s 
is to allow up to two points to collide, i.e. set
$$\S:=\{(\PP^{1}, \{p_i\}_{i=1}^{4}) \ | \text{ no three of $\{p_i\}_{i=1}^{4}$ coincident }\}.$$ 
The class of $\S$-curves is clearly deformation open so there is an associated moduli stack $\M_{0,4}[\S]$. The reader may check that the following assertions hold:
\begin{enumerate} 
\item There exists a diagram
\[
\xymatrix{
M_{0,4} \ar[rd] \ar[d]& \\
\M_{0,4}[\S] \ar[r]^{\phi}& \P^{1}
}
\]
extending the usual cross-ratio map.
\item The fibers $\phi^{-1}(0)$, $\phi^{-1}(1)$ and $\phi^{-1}(\infty)$ comprise
the set of all curves where two points coincide. Furthermore, each of these fibers contains
a unique closed point, e.g. the unique closed point of $\phi^{-1}(0)$ given by $p_1=p_3$ and $p_2=p_4$. 
\end{enumerate}
This shows that $\phi$ is a good moduli map, i.e. that $\PP^{1}$ is a weakly proper birational model of $M_{0,4}$.
\end{exercise}

It is natural to wonder whether there is an analogue of Theorem \ref{T:Construction} for weakly modular compactifications, i.e. whether one can formulate suitable conditions on a class of curves to guarantee that the associated stack has a good moduli space and that this moduli space is proper. For the purpose of this exposition, however, the only general method for constructing weakly modular birational models will be geometric invariant theory, as described in Section \ref{S:GIT}.

\subsection{Modular birational models via MMP}\label{S:modularMMP}
How can we construct a class of curves which is deformation open and satisfies the unique limit property? From the point of view of general theory, the most satisfactory construction proceeds using ideas from higher dimensional geometry. Indeed, one of the foundational theorems of higher dimensional geometry
 is that the canonical ring of a smooth projective variety of general type is finitely generated (\cite{BCHM}, \cite{siu}, \cite{lazic}). 
 A major impetus for work on this theorem was the fact that finite generation of canonical rings in dimension $n+1$ gives a canonical limiting process for one-parameter families of canonically polarized varieties of dimension $n$. In what follows, we explain how this canonical limiting process works and why it leads (in dimension one) to the standard definition of a stable curve.

Let us state the precise version of finite generation that we need:
\begin{proposition}\label{P:CanonicalRing}
Let $\pi\co \C \rightarrow \Delta$ be a flat, projective family of varieties satisfying:
\begin{enumerate}
\item $\C$ is smooth,
\item The special fiber $C \subset \C$ is a normal crossing divisor,
\item $\omega_{\C/\Delta}$ is ample on the generic fiber.
\end{enumerate}
Then
\begin{enumerate}
\item The relative canonical ring $\bigoplus\limits_{m \geq 0}\pi_*(\omega_{\C/\Delta}^m)$ is a birational invariant.
\item The relative canonical ring $\bigoplus\limits_{m \geq 0}\pi_*(\omega_{\C/\Delta}^m)$ is a finitely generated $\O_\Delta$-algebra.
\end{enumerate}
\end{proposition}
\begin{proof}
This theorem is proved in \cite{BCHM}. We will sketch a proof in the case where the relative dimension of $\pi$ is one. To show that $\bigoplus\limits_{m \geq 0}\pi_*\omega_{\C/\Delta}^m$ is a birational invariant, we must show that if $\pi\co \C \rightarrow \Delta$ and $\pi'\co\C' \rightarrow \Delta$ are two families satisfying hypotheses (1)-(3), and they are isomorphic over the generic fiber, then $\pi_*\omega_{\C/\Delta}^m=\pi'_*\omega_{\C'/\Delta}^m$ for each integer $m \geq 0$. Since $\C$ and $\C'$ are smooth surfaces, any birational map $\C \dashrightarrow \C'$ can be resolved into a sequence of blow-ups and blow-downs, and it suffices to show that the relative canonical ring does not change under a simple blow-up $\phi\co \C' \rightarrow \C$. In this case, one has $\omega_{\C'/\Delta}^m=\phi^*\omega_{\C/\Delta}^m(mE)$ and so by the projection 
formula $\phi_*\omega_{\C'/\Delta}^m=\omega_{\C/\Delta}^m$. 

To show that the ring is $\bigoplus\limits_{m \geq 0}\pi_*\omega_{\C/\Delta}^m$ is finitely generated, we will show that it is the homogenous coordinate ring of a certain birational model of $\C$, the so-called {\em canonical model} $\C^{can}$. To construct $\C^{can}$, we begin by contracting curves on which $\omega_{\C/\Delta}$ has negative degree. Observe that if $E \subset C$ is any irreducible curve in the special fiber on which $\omega_{\C/\Delta}$ has negative degree, then in fact $E \simeq \P^{1}$ and $\omega_{\C/\Delta}\cdot E=E^2=-1$ (this is a simple consequence of adjunction). By Castelnuovo's contractibility criterion, there exists a birational morphism $\phi\co \C \rightarrow \C'$, with $\Exc(\phi)=E$ and $\C'$  a smooth surface. After finitely many repetitions of this procedure, we obtain a birational morphism $\C \rightarrow \C^{min}$, such that the special fiber $C^{min}$ has no smooth rational curves $E$ satisfying $E^2=-1$; we call $\C^{min}$ the {\em minimal model} of $\C$.

Now $\omega_{\C^{min}/\Delta}$ is big and relatively nef, so the Kawamata basepoint freeness theorem \cite[Theorem 3.3]{kollar-mori} implies $\omega_{\C^{min}/\Delta}$ is
relatively base point free, i.e. there exists a projective, birational morphism over $\Delta$

$$
\C^{min} \rightarrow \C^{can}:=\Proj \bigoplus_{m \geq 0}\pi_*\omega_{\C^{min}/\Delta}^m
$$
contracting precisely those curves on which $\omega_{\C^{min}/\Delta}$ has degree zero. In particular, $\bigoplus\limits_{m \geq 0}\pi_*\omega_{\C^{min}/\Delta}^m$ is finitely generated. 

\end{proof}

Given Proposition \ref{P:CanonicalRing}, it is easy to describe the canonical limiting procedure.
Let $\C^* \rightarrow \Delta^*$ be a family of smooth projective varieties over the unit disc whose relative canonical bundle $\omega_{\C^*/\Delta^*}$ is ample (in the case of curves, this condition simply means $g \geq 2$). We obtain a canonical limit for this family as follows:
\begin{enumerate}
\item Let $\C \rightarrow \Delta$ be the flat limit with respect to an arbitrary projective embedding $\C^{*} \hookrightarrow \P^{n}_{\Delta^*} \subset \P^{n}_{\Delta}$. 
\item By the semistable reduction theorem \cite[Theorem 7.17]{kollar-mori}, there exists a finite base-change $\Delta' \rightarrow \Delta$ and a family $\C' \rightarrow \Delta'$ such that
\begin{itemize}
\item $\C'|_{(\Delta')^*} \simeq \C \times_{\Delta} (\Delta')^*$
\item The special fiber $C \subset \C'$ is a reduced normal crossing divisor. In the case of curves, this simply means that $C$ is a nodal curve.
\end{itemize}
\item By Proposition \ref{P:CanonicalRing}, the $\O_{\Delta'}$-algebra
$\bigoplus\limits_{m \geq 0}\pi_*\omega_{\C'/\Delta'}^m$ is finitely generated. Thus, we may consider the natural rational map 
$$\C' \dashrightarrow (\C')^{can}:=\Proj \bigoplus\limits_{m}\pi_*\omega_{\C'/\Delta'}^m,$$
which is an isomorphism over $\Delta'$.
\item The special fiber $C^{can} \subset (\C')^{can}$ is the desired canonical limit.
\end{enumerate}
The fact that the section ring $\bigoplus\limits_{m \geq 0}\pi_*\omega_{\C'/\Delta'}^m$ is a birational invariant implies that $C^{can}$ does not depend on the chosen resolution of singularities. Furthermore, the reader may easily check that this construction is independent of the base change in the sense that if $(\C')^{can} \rightarrow \Delta'$ and $(\C'')^{can} \rightarrow \Delta''$ are families produced by this process using two different base-changes, then $$(\C')^{can} \times_{\Delta} \Delta'' \simeq (\C'')^{can} \times_{\Delta}\Delta'.$$

Of course, this construction raises the question: which varieties actually show up as limits of smooth varieties in this procedure? When the relative dimension of $\C^* \rightarrow \Delta^*$ is $2$ or greater, this is an extremely difficult question which has only been answered in certain special cases. In the case of curves, our explicit construction of the map $\C \dashrightarrow \C^{can}$ in the proof of Proposition \ref{P:CanonicalRing} makes it easy to show that:

\begin{proposition}\label{P:LimitStable}
The set of canonical limits of smooth curves is precisely the class of stable curves (Definition \ref{D:StableCurve}).
\end{proposition}
\begin{proof}
After steps (1) and (2) of the canonical limiting procedure, we have a family of curves $\C \rightarrow \Delta$ with smooth generic fiber, nodal special fiber, and smooth total space. In the proof of Proposition \ref{P:CanonicalRing}, we showed that the map $\C \dashrightarrow \C^{can}$ is actually regular, and is obtained in two steps: the map $\C \rightarrow \C^{min}$ was obtained by successively contracting smooth rational curves satisfying $E^2=-1$ and the map $\C^{min} \rightarrow \C^{can}$ was obtained  by contracting all smooth rational curves E satisfying $E^2=-2$. It follows immediately that the special fiber $C^{can} \subset \C^{can}$ has no rational curves meeting the rest of the curve in one or two points, i.e. $C^{can}$ is stable.

Conversely, any stable curve arises as the canonical limit of a family of smooth curves. Indeed, given a stable curve $C$, let $\C \rightarrow \Delta$ be a
 smoothing of $C$ with a smooth total space. Since $\C$ has no smooth rational curves satisfying $E^2=-1$ or $E^2=-2$, we see that $\omega_{\C/\Delta}$ has positive degree on every curve contained in a fiber, and is therefore relatively ample. Thus, $\C=\Proj \bigoplus\limits_{m\geq 0}\pi_*\omega_{\C/\Delta}^m$ is its own canonical model, and $C$ is the canonical limit of the family $\C^* \rightarrow \Delta^*$.
\end{proof}
\begin{corollary}\label{C:stack-Mg}
There exists a proper moduli stack $\SM_{g}$ of stable curves, with an associated coarse moduli space $\M_{g}$.
\end{corollary}
\begin{proof}
The two conditions defining stable curves are evidently deformation open. The class of stable curves has the unique limit property by Proposition \ref{P:LimitStable}
 so Theorem \ref{T:Construction} gives a proper moduli stack of stable curves $\SM_{g}$, with corresponding moduli space $\M_{g}$.
\end{proof}

\begin{remark}
We only checked the existence and uniqueness of limits for one-parameter families with smooth generic fiber. In general, to show that an algebraic stack $\X$ is proper, it is sufficient to check the valuative criterion for maps $\Delta \rightarrow \X$ which send the generic point of $\Delta$ into a fixed open dense subset $\U \subset \X$ \cite[Theorem 4.19]{DM}. Thus, when verifying that a stack of curves has the unique limit property, it is sufficient to consider generically smooth families.
\end{remark}

Next, we explain how this canonical limiting process can be adapted to the case of pointed curves. The basic idea is that, instead of using
$\bigoplus\limits_{m \geq 0}\pi_*\omega_{\C/\Delta}^m$ as the canonical completion of a family, we use  
$\bigoplus\limits_{m \geq 0}\pi_*(\omega_{\C/\Delta}(\Sigma_{i=1}^n \sigma_i)^m)$. Again, 
the key point is that this is finitely-generated and a birational invariant. In fact, there is a further variation, due to Hassett \cite{Hweights}, worth considering. Namely, 
if we replace $\omega_{\C/\Delta}(\Sigma \sigma_i)$ by the $\Q$\nb-line-bundle $\omega_{\C/\Delta}(\Sigma a_i\sigma_i)$ for any
$a_i \in \Q\cap [0,1]$ satisfying $2g-2+\sum a_i\geq 0$, the same statements hold, i.e.  $\bigoplus\limits_{m \geq 0}\pi_*(\omega_{\C/\Delta}(\Sigma a_i\sigma_i)^m)$ is a
finitely generated algebra and it is a birational invariant. This generalization may be appear odd to those who are unversed in the yoga of higher dimensional geometry; the point is that it is precisely divisors of this form for which the general machinery of Kodaira vanishing, Kawamata basepoint freeness, etc. goes through to guarantee finite generation. For the sake of completeness, we will sketch the proof.
\begin{proposition}\label{P:LCRing}
Let $(\C \rightarrow \Delta, \sigman)$ be a family of $n$-pointed curves satisfying:
\begin{enumerate}
\item[(a)] $\C$ is smooth,
\item[(b)] $C \subset \C$ is a nodal curve,
\item[(c)]  The sections $\sigman$ are disjoint.
\end{enumerate}
Then
\begin{enumerate}
\item\label{invariant} The algebra $\bigoplus\limits_{m \geq 0}\omega_{\C/\Delta}(\Sigma_i a_i\sigma_i)^m$ is a birational invariant,
\item\label{fin-gen} The algebra $\bigoplus\limits_{m \geq 0}\omega_{\C/\Delta}(\Sigma_i a_i\sigma_i)^m$ is finitely generated.
\end{enumerate}
\end{proposition}
\begin{proof}
We follow the steps in the proof of Proposition \ref{P:CanonicalRing}. To prove that canonical algebra 
is invariant under birational transformations, we reduce to the case of a simple blow-up $\phi\co \C'\ra \C$. If the center of the blow-up does not lie 
on any section, then the proof proceeds as in Proposition \ref{P:CanonicalRing}. If the center of the blow-up lies on $\sigma_i$, then
we have $\omega_{\C'/\Delta}(\sum a_i\sigma'_i)=\phi^*\omega_{\C/\Delta}(\sum a_i\sigma_i)+(1-a_i)E$. By the projection formula,
$\phi_*(\omega_{\C'/\Delta}(\sum a_i\sigma'_i))=\omega_{\C/\Delta}(\sum a_i\sigma_i)$ if and only if $1-a_i\geq 0$.
This shows that the condition $a_i\leq 1$ is necessary and sufficient for \eqref{invariant}.

For (2), we construct the canonical model $\C^{can}$ precisely as in the proof of Proposition \ref{P:CanonicalRing}. We have only to observe that if $E \subset C$ is any curve such that $\omega_{\C/\Delta}(\sum a_i\sigma_i)\cdot E<0$, then the assumption that $a_i \geq 0$ implies $E$ is a smooth rational curve with $E \cdot E=-1$ and so can be blown down. Thus, we can construct a minimal model $\C^{min}$ on which $\omega_{\C/\Delta}(\sum a_i\sigma_i)$ is big and nef. On $\C^{min}$, we may apply the Kawamata basepoint freeness theorem to arrive at a family arrive at a family with $\omega_{\C/\Delta}(\sum a_i\sigma_i)$ relatively ample and \eqref{fin-gen} follows.
\end{proof}

For any fixed weight vector $\A=(a_1, \ldots, a_n) \in \Q^{n} \cap [0,1]$, we may now describe a canonical limiting process for a family $(\C\ra \Delta^*, \sigman)$ of 
pointed curves over a punctured disk as follows:
\begin{enumerate} \item Let $(\C \rightarrow \Delta, \sigman)$ be the flat limit with respect to an arbitrary projective embedding $\C^{*} \hookrightarrow \P^{n}_{\Delta^*} \subset \P^{n}_{\Delta}$. 
\item By the semistable reduction theorem, there exists a finite base-change $\Delta' \rightarrow \Delta$ and a family $(\C' \rightarrow \Delta', \sigman)$ such that
\begin{itemize}
\item $\C'|_{(\Delta')^*} \simeq \C \times_{\Delta} (\Delta')^*$
\item The special fiber $C' \subset \C'$ is a nodal curve and the sections meet the special fiber $C'$ at distinct smooth points of $C'$.
\end{itemize}
\item By Proposition \ref{P:LCRing}, $\bigoplus\limits_{m \geq 0}\pi_*(\omega_{\C'/\Delta'}(\Sigma_ia_i\sigma_i)^m)$ is a finitely generated $\O_{\Delta'}$\nb-algebra. Thus, we may consider the natural rational map 
$$\C' \dashrightarrow (\C')^{can}:=\Proj \bigoplus_{m \geq 0}\pi_*(\omega_{\C'/\Delta'}(\Sigma a_i\sigma_i)^m).$$
\item The special fiber $C^{can} \subset (\C')^{can}$ is the desired canonical limit.
\end{enumerate}

Using the proof of Proposition \ref{P:LCRing}, it is easy to see which curves actually arise as limits of smooth curves under this process, for a fixed weight vector $\A$. The key point is that the map
$$
\C' \rightarrow \Proj \bigoplus_{m \geq 0}\pi_*(\omega_{\C'/\Delta'}(\Sigma a_i\sigma_i)^m)
$$
successively contracts curves on which $\omega_{\C'/\Delta'}(\Sigma_ia_i\sigma_i)$ has non-positive degree. By adjunction, any such curve $E \subset C$ satisfies: 
\begin{enumerate}
\item $E$ is smooth rational, $E \cdot E=-1$, and $\sum_{\{i: p_i \in E\}}a_i \leq 1$.
\item $E$ is smooth rational, $E \cdot E=-2$, and there are no marked points supported on $E$.
\end{enumerate}
Contracting curves of the first type creates smooth points where the marked points $\{p_i \in E\}$ coincide, while contracting curves of the second type creates nodes. Thus, the special fibers $C^{can} \subset \C$ are precisely the \emph{$\A$-stable curves}, defined below.
\begin{definition}[$\A$-stable curves]\label{D:A-stable}
An $n$-pointed curve $(C, \pn)$ is $\A$-stable if it satisfies
\begin{enumerate}
\item The only singularities of $C$ are nodes,
\item If $p_{i_1}, \ldots, p_{i_k}$ coincide in $C$, then $\sum_{j=1}^{k}a_{i_j} \leq 1$,
\item $\omega_{C}(\Sigma_{i=1}^n a_ip_i)$ is ample.
\end{enumerate}
\end{definition}

\begin{corollary}[\cite{Hweights}]
There exist proper moduli stacks $\SM_{g,\A}$ for $\A$-stable curves, with corresponding moduli spaces $\M_{g,\A}$.
\end{corollary}
\begin{proof}
The conditions defining $\A$-stability are obviously deformation open, and we have just seen that they satisfy the unique limit property. The corollary follows from Theorem \ref{T:Construction}.
\end{proof}

\begin{remark}
An stable $n$-pointed curve is simply an $\A$-stable curve with weights $\A=(1, \ldots, 1)$, and the corresponding moduli space is typically denoted $\M_{g,n}$.
\end{remark}

Before moving on to the next section, let us take a break from all this high theory and do some concrete examples. We have just seen that it is always  possible to fill in a one-parameter family of smooth curves with a stable limit. As we will see in subsequent sections, a question that arises constantly is how to tell \emph{which} stable curve arises when applying this procedure to a family of smooth curves degenerating to a given singular curve. In practice, the difficult step in the procedure outlined above is finding a resolution of singularities of the total space of the family. Those cases in which it is possible to do stable reduction by hand typically involve a trick that allows one to bypass an explicit resolution of singularities. For example, here is a trick that allows one to perform stable reduction on any family of curves acquiring an $A_{2k}$-singularity ($y^2=x^{2k+1}$).
\begin{example}[Stable reduction for an $A_{2k}$-singularity]
\label{E:stable-reduction-A2k}
 Begin with a generically smooth family $\mathcal{C}\ra \Delta$ of pointed curves
whose special fiber $C_0$ has an isolated singularity of type $A_{2k}$ at a point $s\in C_0$, i.e. $\hat{\O}_{C_0,s} \simeq 
k[[x,y]]/(y^2-x^{2k+1})$. By the deformation theory of complete intersection singularities, the local equation of $\C$ near $s=(0,0)$ is 
\[
y^2=x^{2k+1}+a_{2k-1}(t)x^{2k-1}+\cdots+a_1(t)x+a_0(t),
\]
where $a_{2k-1}(0)=\cdots=a_0(0)=0$. Therefore, we can regard $\C$ as a double cover of $\Y:=\AA^{1}_x \times \Delta$. After a finite flat 
base change we can assume that 
\[x^{2k+1}+a_{2k-1}(t)x^{2k-1}+\cdots+a_1(t)x+a_0(t)=(x-\sigma_1(t))\cdots (x-\sigma_{2k+1}(t)).
\]

\noindent
First, we perform stable reduction for the family $(\AA^1_x\times \{t\}, \sigma_1(t), \dots, \sigma_{2k+1}(t))$ of pointed rational curves. 
Assume for simplicity that the sections $\sigma_i(t)$ intersect pairwise transversely at the point $(0,0)\in \AA^{1}_x\times \Delta$. Let 
 $\Bl_{(0,0)}\Y \ra \Y$ be the ordinary blow-up with the center $(0,0)$ and the exceptional divisor $E$. 
 Then the strict transforms $\overline{\sigma_i}$ of 
 sections $\{\sigma_i\}_{i=1}^{2k+1}$ meet $E$ in $2k+1$ distinct points. Unfortunately, the divisor 
 $\sum_{i=1}^{2k+1} \overline{\sigma_i}$ is not
 divisible by $2$ in $\Pic(\Bl_{(0,0)}\Y)$ and so we cannot simply construct a double cover of $\Bl_{(0,0)}\Y$ branched over it. 
 To circumvent
 this difficulty, we make another base change $\Delta'\ra \Delta$ of degree $2$ ramified over $0\in \Delta$. The surface 
 $\Y':=\Bl_{(0,0)}\Y\times_{\Delta}\Delta'$ has an $A_1$-singularity over the node in the central fiber 
 of $\Bl_{(0,0)}\Y\ra \Delta$. 
 Make an ordinary blow-up $\Y''\ra \Y'$ to resolve this singularity. Denote by $F$ the 
exceptional $(-2)$\nobreak-curve, by $\Sigma$ the preimage of  
$\sum_{i=1}^{2k+1} \overline{\sigma_i}$, and continue to denote by $E$ the preimage of $E$. 
The divisor $\Sigma+F$ is easily seen to be divisible by $2$ in $\Pic(\Y'')$, so we can construct the cyclic cover of degree $2$ branched over $\Sigma$ and $F$.
The normalization of the cyclic cover is a smooth surface. Its central fiber over $\Delta'$ 
decomposes as $\widetilde{C}_0 \cup F' \cup T$, where $\widetilde{C}_0$ is the normalization of $C_0$ at $s$, $F'$ is a smooth rational curve mapping $2:1$ onto $F$, and $T$ is a smooth hyperelliptic curve of genus $k$ mapping $2:1$ onto $E$ and ramified over 
$E\cap (\Sigma \cup F)$. Blowing down $F'$, we obtain the requisite stable limit $\widetilde{C}_0 \cup T$. 

Note that \emph{which} hyperelliptic curve actually arises as the tail of the stable limit depends on the 
initial family, i.e. on the slopes of the sections $\sigma_i(t)$. Varying the slopes of these sections, we can evidently arrange for any smooth hyperelliptic curve, attached to $\widetilde{C}_0$ along a Weierstrass point, to appear as the stable tail.

\end{example}

\subsection{Modular birational models via combinatorics}\label{S:modular-combinatorics}

While the moduli spaces $\M_{g,n}$ and $\M_{g,\A}$ are the most natural modular compactifications of $M_{g,n}$ from the standpoint of general Mori theory, they are far from unique in this respect. In this section, we will explain how to bootstrap from the stable reduction theorem to obtain many alternate deformation open classes of curves satisfying the unique limit property. The starting point for all these construction is the obvious, yet useful, observation that one can do stable reduction in reverse. For example, we have already seen that applying stable reduction to a family of curves acquiring a cusp has the effect of replacing the cusp by an elliptic tail (Example \ref{E:stable-reduction-A2k}). Conversely, starting with a family of curves whose special fiber has an elliptic tail, one can contract the elliptic tail to get a cusp. More generally, we have

\begin{lemma}[Contraction Lemma]\label{L:Contraction}
Let $\C \rightarrow \Delta$ be a family of curves with smooth generic fiber and nodal special fiber. Let $Z \subset C$ be a connected proper subcurve of the special fiber $C$. Then there exists a diagram
\[
\xymatrix{
\C \ar[rr]^{\phi} \ar[rd]_{\pi}&&\C' \ar[ld]^{\pi'}\\
&\Delta&
}
\]
satisfying
\begin{enumerate}
\item $\phi$ is proper and birational morphism of algebraic spaces, with $\Exc(\phi)=Z$.
\item $\pi'\co \C' \rightarrow \Delta$ is a flat proper family of curves, with connected reduced special fiber $C'$.
\item $\phi|_{\overline{C\setminus Z}}\co \overline{C \setminus Z} \rightarrow C'$ is the normalization of $C'$ at $p:=\phi(Z)$.
\item The singularity $p \in C'$ has the following numerical invariants:
\begin{align*}
m(p)&=|\overline{C \setminus Z} \cap Z|,\\
\delta(p)&=p_a(Z)+m(p)-1,
\end{align*}
\end{enumerate}
where $m(p)$ is the number of branches of $p$ and $\delta(p)$ is the $\delta$-invariant of $p$.
\end{lemma}
(For definition of $\delta(p)$ see \cite[Ex. IV.1.8]{Hartshorne}.)
\begin{proof}
If $Z$ is any proper subcurve of the special fiber $C$, then any effective $\Q$-divisor supported on $Z$ has negative self-intersection \cite[Proposition 2.6]{Z-stability}. Artin's contractibility criterion then implies the existence of a proper birational morphism $\phi\co \C \rightarrow \C'$ with $\Exc(\phi)=Z$, to 
a normal algebraic space  $\C'$ \cite[Corollary 6.12]{artin-formal-moduli-II}. Evidently, $\pi$ factors through $\phi$ so we may regard $\pi'\co \C' \rightarrow \Delta$ as a family of curves. (Flatness of $\pi'$ is automatic since the generic point of $\C'$ lies over the generic point of $\Delta$.)

To see that the special fiber $C'$ is reduced, first observe that $C'$ is a Cartier divisor in $\C'$, hence has no embedded points. On the other hand, no component of $C'$ can be generically non-reduced because it is a birational image of some component of $\overline{C \setminus Z}$. This completes the proof of (1) and (2).

Conclusion (3) is immediate from the observation that $\overline{C \setminus Z}$ is smooth along the points $Z \cap \overline{C \setminus Z}$ and maps isomorphically to $C'$ elsewhere. Since the number of branches of the singular point $p \in C'$ is, by definition, the number of points lying above $p$ in the normalization, we have
$$m(p)=|\overline{C \setminus Z} \cap Z|.$$
Finally, the formula for $\delta(p)$ is a consequence of the fact that $C$ and $C'$ have the same arithmetic genus (since they occur in flat families with the same generic fiber).
\end{proof}
We would like to emphasize that even when applying the Contraction Lemma to a family $\C\ra \Delta$
whose total space $\C$ is a scheme, we can expect $\C'$ to be only an algebraic space in general.
(This explains our definition of a family in Section \ref{S:notation}). Clearly, a necessary
condition for $\C'$ to be a scheme is an existence of a line bundle in a Zariski open 
neighborhood of $Z$ on $\C$ that restricts to a trivial line bundle on $Z$. 
Indeed, if $\C'$ is a scheme, the pullback 
of the trivial line bundle from an affine neighborhood of $p\in \C'$ restricts to $\O_Z$ on $Z$. 
\begin{exercise}
Let $C$ be a smooth curve of genus $g\geq 2$ and let $x\in C$ be a general point. Consider the family 
$\C\ra \Delta$ obtained by blowing up the point $x$ in the central fiber of the constant family 
$C\times \Delta$. Prove that the algebraic space obtained by contracting the strict transform of 
$C\times 0$ in $\C$ is not a scheme. 
\end{exercise}

Lemma \ref{L:Contraction} raises the question: given a family $\C \rightarrow \Delta$ and a subcurve $Z \subset C$, how do you figure out which singularity $p \in C'$ arises from contracting $Z$? When contracting low genus curves, the easiest way to answer this question is simply to classify all singularities with the requisite numerical invariants. For example, Lemma \ref{L:Contraction} (3) implies that contracting an elliptic tail must produce a singularity with one branch and $\delta$\nb-invariant $1$. However, the reader should have no difficulty verifying that the cusp is the unique singularity with these invariants (Exercise \ref{E:Cusp}). Thus, Lemma \ref{L:Contraction} (3) implies that contracting an elliptic tail must produce a cusp!

\begin{exercise}\label{E:Cusp}
Prove that the unique curve singularity $p \in C$ with one branch and $\delta$\nb-invariant $1$ is the standard cusp $y^2=x^3$ (up to analytic isomorphism). (Hint: the hypotheses imply that $\hat{\O}_{C,p} \subset k[[t]]$ is a codimension one $k$-subspace.)
\end{exercise}

For a second example, consider a family $\C \rightarrow \Delta$ with smooth generic fiber and nodal special fiber, and suppose $Z \subset C$ is an arithmetic genus zero subcurve meeting the complement $\overline{C \setminus Z}$ in $m$ points. What singularity arises when we contract $Z$? Well, according to Lemma \ref{L:Contraction} (3), we must classify singularities with $m$ branches and $\delta$-invariant $m-1$.

\begin{exercise}\label{E:rationalmfold}
Prove that the unique curve singularity $p \in C$ with $m$ branches and $\delta$-invariant $m-1$ is simply the union of the $m$ coordinates axes in $m$-space. More precisely,
\begin{align*}
\hat{\O}_{C,p} &\simeq k[[x_1, \ldots, x_m]]/I_{m},\\
I_{m}&:=(x_ix_j: 1 \leq i<j \leq m).
\end{align*}
We call this singularity the \emph{rational $m$-fold point}. Conclude that, in the situation described above, contracting $Z \subset C$ produces a rational $m$-fold point.
\end{exercise}

A special feature of these two examples is that the singularity produced by Lemma \ref{L:Contraction} depends only on the curve $Z$ and not on the family $\C \rightarrow \Delta$. In general, this will not be the case. We can already see this in the case of elliptic bridges, i.e. smooth elliptic curves meeting the complement in precisely two points. According to Lemma \ref{L:Contraction}, contracting an elliptic bridge must produce a singularity with two branches and $\delta$-invariant $2$. It is not difficult to show that there are exactly two singularities with these numerical invariants: the tacnode and the spatial singularity obtained by passing a smooth branch transversely through a planar cusp.
\begin{definition}
We say that $p \in C$ is a \emph{punctured cusp} if $$\hat{\O}_{C,p} \simeq k[[x,y,z]]/(zx, zy, y^2-x^3).$$
We say that $ p \in C$ is a \emph{tacnode} if $$\hat{\O}_{C,p} \simeq k[[x,y]]/(y^2-x^4).$$
\end{definition}

It turns out that contracting an elliptic bridge can produce \emph{either} a punctured cusp or a tacnode, depending on the choice of smoothing. The most convenient way to express the relevant data in the choice of smoothing is the indices of the singularities of the total space at the attaching nodes of the elliptic bridge. More precisely, suppose $C=E \cup F$ is a curve with two components, where $E$ is a smooth elliptic curve meeting $F$ at two nodes, say $p_1$ and $p_2$, and let $\C \rightarrow \Delta$ be a smoothing of $C$. Let $m_1$, $m_2$ be the integers uniquely defined by the property that $\hat{\O}_{\C,p_i} \simeq k[[x,y,t]]/(xy-t^{m_i})$. Then we have

\begin{proposition}\label{P:BridgeContraction}
With notation as above, apply Lemma \ref{L:Contraction} to produce a contraction $\phi\co \C \rightarrow \C'$ with $\Exc(\phi)=E$. Then
\begin{enumerate}
\item If $m_1=m_2$, $p:=\phi(E) \in C'$ is a tacnode.
\item If $m_1 \neq m_2$, $p:=\phi(E) \in C'$ is a punctured cusp.
\end{enumerate}
\end{proposition}
\begin{proof}
We will sketch an ad-hoc argument exploiting the fact that the tacnode is Gorenstein, while the punctured cusp is not \cite[Appendix A]{SmythEI}. This means that if $C$ has a tacnode then the dualizing sheaf $\omega_{C}$ is invertible, while if $C$ has a punctured cusp, it is not. First, let us show that if $m_1 \neq m_2$, $p:=\phi(E) \in C'$ is a punctured cusp. Suppose, on the contrary that the contraction $\phi\co \C \rightarrow \C'$ produces a tacnode. Then the relative dualizing sheaf $\omega_{\C'/\Delta}$ is invertible, and the pull back $\phi^*\omega_{\C'/\Delta}$ is an invertible sheaf on $\C$, canonically isomorphic to $\omega_{\C/\Delta}$ away from $E$. It follows that
$$
\phi^*\omega_{\C'/\Delta}=\omega_{\C/\Delta}(D),
$$
where $D$ is a Cartier divisor supported on $E$ and $\omega_{\C/\Delta}(D)|_{E} \simeq \O_{E}$. It is an easy to exercise to check that it is impossible to satisfy both these conditions if $m_1 \neq m_2$.

Conversely if $m_1=m_2=m$, then $\omega_{\C/\Delta}(mE)$ is a line bundle on $\C$ and $\omega_{\C/\Delta}(mE)|_{E} \simeq \O_{E}.$ One can show that $\omega_{\C/\Delta}(mE)$ is semiample, so that the contraction $\phi$ is actually induced by a high power of  $\omega_{\C/\Delta}(mE)$ \cite[Lemma 2.12]{SmythEI}. It follows that there is a line bundle $\L$ on $\C'$ which pulls back to $\omega_{\C/\Delta}(mE)$. We claim that $\L \simeq \omega_{\C'/\Delta}$, which implies that $\omega_{C'}$ is invertible, hence that $p \in C'$ is Gorenstein (hence a tacnode). To see this, simply note that $\L \simeq \omega_{\C'/\Delta}$ away from $p$ and then use the fact that  $\omega_{\C'/\Delta}$ is an $S_2$-sheaf \cite[Corollary 5.69]{kollar-mori}.
\end{proof}

Using the contraction lemma, it is easy to use the original stable reduction theorem to construct new classes of curves which are deformation open and satisfy the unique limit property. For example, let us consider the class of pseudostable curves which was defined in the introduction (Definition \ref{D:pseudostable}).
\begin{proposition}\label{P:Pseudostablelimits}
For $g \geq 3$, the class of pseudostable curves is deformation open and satisfies the unique limit property.
\end{proposition}
\begin{proof}
The conditions defining pseudostability are clearly deformation open, so it suffices to check the unique limit property. To prove existence of limits, let $\C^* \rightarrow \Delta^*$ be any family of smooth curves over the punctured disc, and let $\C \rightarrow \Delta$ be the stable limit of the family. Now, assuming $g \geq 3$, the elliptic tails in $C$ are all disjoint. Let $Z \subset C$ be the union of the elliptic tails, and apply Lemma \ref{L:Contraction} to obtain a contraction $\phi\co \C \rightarrow \C'$ replacing elliptic tails by cusps. The special fiber $C'$ is evidently pseudostable, so $\C' \rightarrow \Delta$ is the desired family. The uniqueness of pseudostable limits follows immediately from the uniqueness of the stable limits (If $C_1$ and $C_2$ are two pseudostable limits to the same family $\C^{*} \rightarrow \Delta^*$, then their associated stable limits $C_1^s$ and $C_2^s$ are obtained by replacing the cusps by elliptic tails. Obviously, $C_1^s \simeq C_2^s$ implies $C_1 \simeq C_2$).
\end{proof}

\begin{corollary}[\text{\cite{Schubert, HH1}}]\label{C:pseudostable-stack}
For $g\geq 3$, there exists a proper Deligne-Mumford moduli stack $\SM_{g}^{\, ps}$ of pseudostable curves, with corresponding moduli space $\M_{g}^{\, ps}$. \end{corollary}
\begin{proof}
Immediate from Theorem \ref{T:Construction} and Proposition \ref{P:Pseudostablelimits}.
\end{proof}

\begin{remark}[Reduction morphism $\M_{g} \rightarrow \M_{g}^{\, ps}$]\label{R:reduction-pseudostable}
For $g\geq 3$, there is a natural transformation of 
functors $\eta\co \Mg{g}\ra \Mg{g}^{\, ps}$ which is an 
isomorphism away from $\Delta_1$. It induces a corresponding 
morphism $\M_g\ra \M_g^{\, ps}$ on the coarse moduli spaces. The map is defined by sending a curve $C\cup E_1\cup \ldots\cup E_m$, where $E_i$ 
are elliptic tails attached to $C$ at points $q_1,\dots, q_m$ to the unique curve $C'$ with cusps $q'_1,\dots, q'_m\in C'$ and pointed 
normalization  $(C, q_1, \dots, q_m)$. This replacement procedure defines $\eta$ on points, and one can check that it extends to families. 
\end{remark}

The class of pseudostable curves was first defined in the context of GIT, and we shall return to GIT construction of $\M_{g}^{\, ps}$ in Section \ref{S:GIT}. However, we can also use these ideas to construct examples which have no GIT construction (at least none currently known). One such example is the following, in which we use Exercise \ref{E:rationalmfold} to define an alternate modular compactification of $M_{0,n}$ using rational $m$-fold points. 
\begin{definition}[$\psi$-stable]\label{D:psi-stable} Let $(C, \pn)$ be an $n$-pointed curve of arithmetic genus zero. We say that $(C, \pn)$  is $\psi$-stable if it satisfies: 
\begin{enumerate}
\item $C$ has only nodes and rational $m$-fold points as singularities.
\item Every component of $C$ has at least three distinguished points (i.e. singular or marked points).
\item Every component of $C$ has at least one marked point.
\end{enumerate}
\end{definition}
\noindent
Our motivation for calling this condition $\psi$-stability will become clear in Section \ref{S:MMPM0n}.

\begin{proposition}\label{P:Psistablelimits}
The class of $\psi$-stable curves is deformation open and satisfies the unique limit property.
\end{proposition}
\begin{proof}
The conditions defining $\psi$-stability are evidently deformation open, so we need only check the unique limit property. Let $\C^* \rightarrow \Delta^*$ be any family of smooth curves over
a punctured disk, and let $\C \rightarrow \Delta$ be the stable limit of the family. Now let $Z \subset C$ be the union of all the unmarked components of $C$, and apply Lemma \ref{L:Contraction} to obtain a contraction $\phi\co \C \rightarrow \C'$ replacing these components by rational $m$-fold points. The special fiber $C'$ is evidently $\psi$-stable, so $\C' \rightarrow \Delta$ is the desired limit family. The uniqueness of $\psi$-stable limits follows immediately from uniqueness of stable limits as in Proposition \ref{P:Pseudostablelimits}, and we leave the details to the reader.
\end{proof}

\begin{corollary}\label{C:psi-stack}
There exists a proper Deligne-Mumford moduli stack $\SM_{0,n}[\psi]$ of $\psi$\nb-stable curves, and an associated moduli space $\M_{0,n}[\psi]$. 
\end{corollary}
\begin{proof}
Immediate from Theorem \ref{T:Construction} and Proposition \ref{P:Psistablelimits}.
\end{proof}
\begin{remark}
Since $\psi$\nb-stable curves have no automorphisms, one actually has $\SM_{0,n}[\psi]
 \simeq \M_{0,n}[\psi]$, and so $\M_{0,n}[\psi]$ is a fine moduli space.
\end{remark}
\begin{remark}\label{R:psi-contraction}
As in Remark \ref{R:reduction-pseudostable}, there is a natural transformation of functors 
$\psi\co \M_{0,n}\ra \M_{0,n}[\psi]$. If $[C]\in \M_{0,n}$, then $\psi([C])$ is obtained from $C$ by contracting every maximal unmarked connected component
of $C$ meeting the rest of the curve in $m$ points to the unique $m$\nb-fold rational singularity. This replacement procedure defines $\psi$ on points, and one can check that it extends to families. Clearly, positive dimensional fibers of $\psi$ are families of curves with an unmarked moving component.
\end{remark}
At this point, the reader may be wondering how far one can take this method of simply selecting and contracting various subcurves of stable curves. This idea is investigated in \cite{Z-stability}, and we may summarize the conclusions of the study as follows:
\begin{enumerate}
\item As $g$ and $n$ become large, this procedure gives rise to an enormous number of modular birational models of $\M_{g,n}$, involving all manner of exotic singularities.
\item Nevertheless, this approach is not adequate for constructing moduli spaces of curves with certain natural classes of singularities, e.g. this procedure never gives rise to a class of curves containing only nodes ($y^2=x^2$), cusps ($y^2=x^3$), and tacnodes ($y^2=x^4$).
\end{enumerate}

Let us examine problem (2) a little more closely. What goes wrong if we simply consider the class of curves with nodes, cusps, and tacnodes, and disallow elliptic tails and elliptic bridges. Why doesn't this class of curves have the unique limit property? There are two issues: First, one cannot always replace elliptic bridges by tacnodes in a one-parameter family. As we saw in Proposition \ref{P:BridgeContraction}, contracting elliptic bridges sometimes gives rise to a punctured cusp. Second, whereas the elliptic tails of a stable curve are disjoint, so that one can contract them all simultaneously (at least when $g \geq 3$), elliptic bridges may interact with each other in such a way as to make this impossible. Consider, for example,  a one-parameter family of smooth genus three curves specializing to a pair of elliptic bridges (Figure \ref{F:manylimits}). How can one modify the special fiber to obtain a tacnodal limit for this family? Assuming the total space of the family is smooth, one can contract either $E_1$ or $E_2$ to obtain two non-isomorphic tacnodal special fibers, but there is no canonical way to distinguish between these two limits. 

Neither of these problems is necessarily insoluble. Regarding the first problem, suppose that $\C \rightarrow \Delta$ is a one-parameter family of smooth curves acquiring an elliptic bridge in the special fiber. As in the set-up of Proposition \ref{P:BridgeContraction}, suppose that the elliptic bridge of the special fiber is attached by two nodes, and that one of these nodes is a smooth point of $\C$ while the other is a singularity of the form $xy-t^2$. Proposition \ref{P:BridgeContraction} implies that contracting $E$ outirght would create a punctured cusp, but we can circumvent this problem as follows: First, desingularize the total space with a single blow-up, so that the strict transform of the elliptic bridge is connected at two nodes with smooth total space. Now we may contract the elliptic bridge to obtain a tacnodal special fiber, at the expense of inserting a $\P^{1}$ along one of the branches.

Regarding the second problem, one can try a similar trick: Given the family pictured in Figure \ref{F:manylimits}, one may blow-up the two points of intersection $E_1 \cap E_2$, make a base-change to reduce the multiplicities of the exceptional divisors, and then contract \emph{both} elliptic bridges to obtain a bi-tacnodal limit whose normalization comprises a pair of smooth rational curves. This limit curve certainly appears canonical, but it has an infinite automorphism group and contains the other two pictured limits as deformations. Of course, it is impossible to have a curve with an infinite automorphism group in a class of curves with the unique limit property. These examples suggest is that if we want a compact moduli space for curves with nodes, cusps, and tacnodes, we must make do with a weakly modular compactification, i.e. consider a slightly non-separated moduli functor. In Section \ref{S:GIT}, we will see that geometric invariant theory identifies just such a functor.

\begin{figure}
\scalebox{.50}{\includegraphics{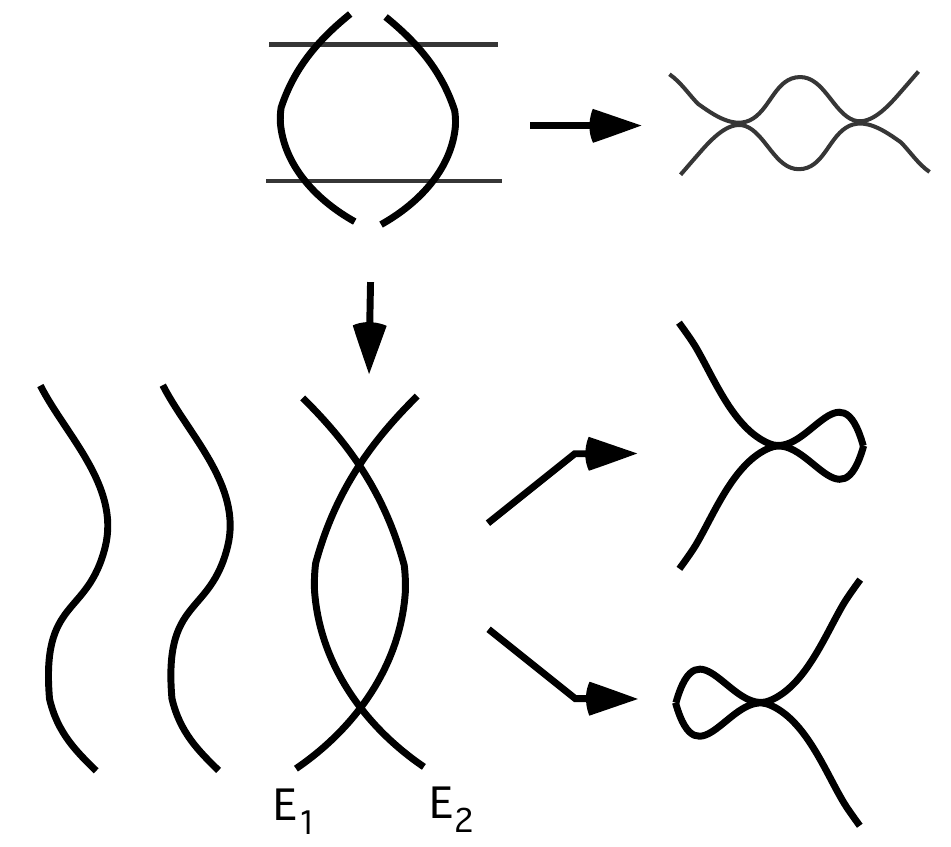}}
\caption{Three candidates for the tacnodal limit of a one-parameter family of genus three curves specializing to a pair of elliptic bridges.}\label{F:manylimits}
\end{figure}

It is worth pointing out that there is one family of examples, where the problem of interacting elliptic components does not arise and the na\"ive definition of a tacnodal functor actually gives rise to a class of curves with the unique limit property. This is the case of $n$-pointed curves of genus one. In fact, it turns out that there is a beautiful sequence of modular compactifications of $M_{1,n}$ which can be constructed by purely combinatorial methods, although the required techniques are somewhat more subtle than what we have described thus far. For the sake of completeness, and also because these spaces arise in the log MMP for $M_{1,n}$ 
(Section \ref{S:MMPM1n}), we define this sequence of stability conditions below.

The singularities introduced in this sequence of stability conditions are the so-called elliptic $m$-fold points, whose name derives from the fact that they are the unique Gorenstein singularities which can appear on a curve of arithmetic genus one \cite[Appendix A]{SmythEI}.

\begin{definition}[The elliptic $m$-fold point]\label{D:ellipticmfold} We say that $p \in C$ is an \emph{elliptic $m$-fold point} if
\begin{align*}
\hat{\O}_{C,p} \simeq &
\begin{cases}
k[[x,y]]/(y^2-x^3) & m=1\,\,\,\text{(ordinary cusp)}\\
k[[x,y]]/(y^2-yx^2) & m=2 \,\,\,\text{(ordinary tacnode)} \\
k[[x,y]]/(x^2y-xy^2) & m=3 \,\,\, \text{(planar triple point)}\\
k[[x_1, \ldots, x_{m-1}]]/J_m & m \geq 4, \text{($m$ general lines through $0\in \mathbb{A}^{m-1}$),}\\
\end{cases}\\
&\,\,\,\,\, J_{m}:= \left(  x_{h}x_i-x_{h}x_j \, : \,\, i,j,h \in \{1, \ldots, m-1\} \text{ distinct} \right).
\end{align*}
\end{definition}

We now define the following stability conditions.

\begin{definition}[$m$-stability]
Fix positive integers $m < n$. Let $(C,\pn)$ be an $n$-pointed curve of arithmetic genus one. We say that $(C, \pn)$ is \emph{$m$-stable} if
\begin{itemize}
\item[(1)] $C$ has only nodes and elliptic $l$-fold points, $l \leq m$, as singularities.
\item[(2)] If $E \subset C$ is any connected subcurve of arithmetic genus one, then $$|E \cap \overline{C \setminus E}|+|\{p_i \,|\, p_i \in E\}|>m.$$
\item[(3)] $\HH^0(C, \Omega_{C}^{\vee}(-\Sigma_{i=1}^n p_i))=0,$ i.e. $(C, \pn)$ has no infintesimal automorphisms.
\end{itemize}
\end{definition}

In order to provide some intuition for the definition of $m$-stability, Figure \ref{F:BoundaryM14} displays all topological types of curves in $\M_{1,4}(3)$, as well as the specialization relations between them.
\begin{figure}
\scalebox{.70}{\includegraphics{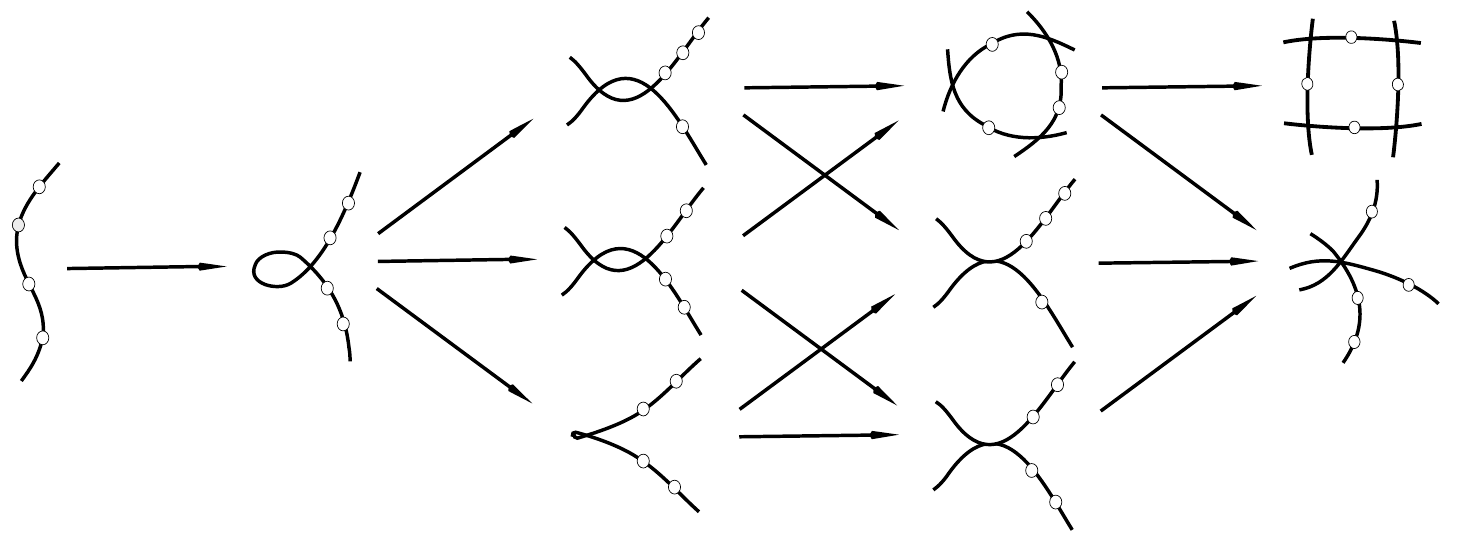}}
\caption{Topological types of curves in $\M_{1,4}(3)$. Every pictured component is rational, except the smooth elliptic curve pictured on the far left.}\label{F:BoundaryM14}
\end{figure}

\begin{proposition}[\cite{SmythEI}]
There exists a moduli stack $\SM_{1,n}[m]$ of $m$-stable curves, with corresponding moduli space $\M_{1,n}[m]$. 
\end{proposition}

While the proof of this theorem would take us too far afield, we may at least explain how to obtain the $m$-stable limit of a generically smooth family of $n$-pointed elliptic curves. The essential feature which makes this process more subtle than the constructions of $\M_{g}^{\, ps}$ and $\M_{0,n}[\psi]$ is that one must alternate between blowing-up and contracting, rather than simply contracting subcurves of the stable limit.

Given  a family of smooth $n$-pointed curves of genus one over a punctured disc $(\C^* \rightarrow \Delta^*, \sigman)$, we obtain the $m$-stable limit as follows. (The process is pictured in Figure \ref{F:valuativecriterion}.)
\begin{enumerate}
\item First, complete $(\C^* \rightarrow \Delta^*, \sigman)$ to a semistable family $(\C \rightarrow \Delta, \sigman)$ with smooth total space, i.e. take the stable limit and desingularize the total space.
\item Isolate the minimal elliptic subcurve $Z \subset C$, i.e. the unique elliptic subcurve which contains no proper elliptic subcurves, blow-up the marked points on $Z$, then contract the strict transform $Z$.
\item Repeat step (2) until the minimal elliptic subcurve satisfies
$$
|Z \cap \overline{C \setminus Z}| + |\{p_i \,|\, p_i \in Z\}| > m.
$$
\item Stabilize, i.e. blow-down all smooth $\P^{1}$'s which meet the rest of the fiber in two nodes and have no marked points, or meet the rest of the fiber in a single node and have one marked point. \
\end{enumerate}

\begin{figure}[h]
\scalebox{.55}{\includegraphics{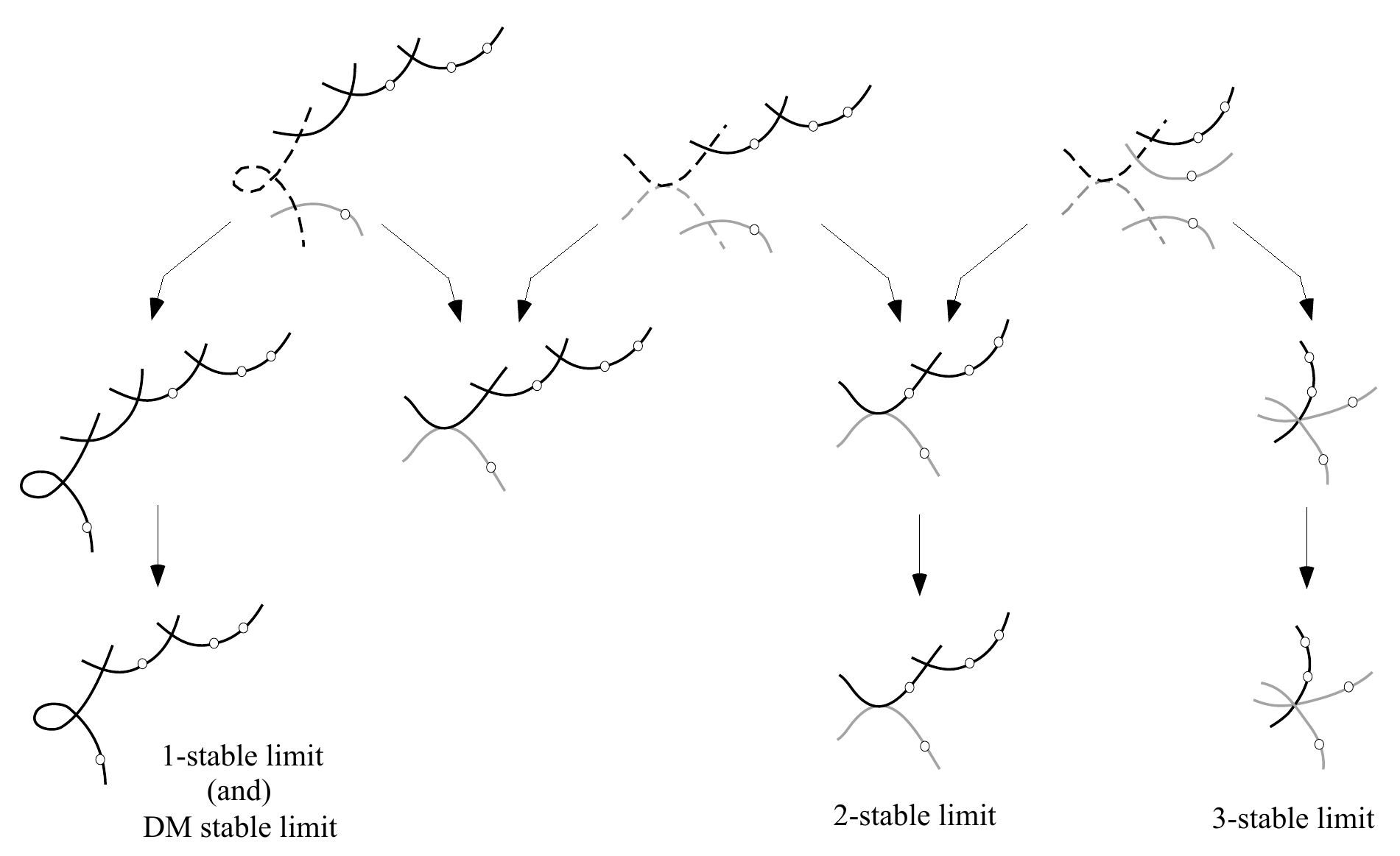}}
\caption{The process of blow-up/contraction/stabilization in order to extract the $m$-stable limit for each $m=1,2,3$. Every irreducible component pictured above is rational. The left-diagonal maps are simple blow-ups along the marked points of the minimal elliptic subcurve, and exceptional divisors of these blow-ups are colored grey. The right-diagonal maps contract the minimal elliptic subcurve of the special fiber, and exceptional components of these contractions are dotted. The vertical maps are stabilization morphisms, blowing down all semistable components of the special fiber.}\label{F:valuativecriterion}
\end{figure}

Finally, we should also note that the $m$-stability is compatible with the definition of $\A$-stability described above, thus yielding an even more general collection of stability conditions.

\begin{definition}[$(m,\A)$-stability]
Fix positive integers $m < n$, and let $\A=(a_1, \ldots, a_n) \in (0,1]^n$ be an $n$-tuple of rational weights. Let $C$ be a connected reduced complete curve of arithmetic genus one, and let $p_1, \ldots, p_n \in C$ be smooth (not necessarily distinct) points of $C$. We say that $(C,p_1, \ldots, p_n)$ is $(m, \A)$-stable if
\begin{itemize}
\item[(1)] $C$ has only nodes and elliptic $l$-fold points, $l \leq m$, as singularities.
\item[(2)] If $E \subset C$ is any connected subcurve of arithmetic genus one, then $$|E \cap \overline{C \setminus E}|+|\{p_i \,|\, p_i \in E\}|  >m.$$
\item[(3)] $\HH^0(C, \Omega_{C}^{\vee}(-\Sigma_{i}p_i))=0.$
\item[(4)] If $p_{i_1}=\ldots=p_{i_k} \in C$ coincide, then $\sum_{j=1}^{k}a_{i_j} \leq 1.$ 
\item[(5)] $\omega_{C}(\sum_{i}a_ip_i)$ is an ample $\Q$-divisor.
\end{itemize}
\end{definition}

In \cite{SmythEI}, it is shown that the class of $(m,\A)$-stable curves is deformation open and satisfies the unique limit property. Thus, there exist corresponding moduli stacks $\SM_{1,\A}(m)$ and spaces $\M_{1,\A}(m)$.

\subsection{Modular birational models via GIT}\label{S:GIT}

In this section, we discuss the use of geometric invariant theory (GIT) to construct modular and weakly modular 
birational models of $M_{g,n}$. Following Hassett and Hyeon \cite{HH2}, we will explain certain heuristics for 
interpreting GIT quotients as log canonical models of $\M_{g}$, and describe how to use these heuristics, in 
conjunction with intersection theory on $\M_{g}$, to predict the GIT-stability of certain Hilbert points.

GIT was invented by Mumford \cite{GIT} to solve the following problem: Suppose $G$ is a linearly reductive group 
acting on a normal projective variety $X$. We wish to form a projective quotient $X \rightarrow X\gitq G$, whose fibers 
are precisely the $G$\nb-orbits of $X$. Unfortunately, there are obvious topological obstructions to the existence of 
such a quotient. For example, if one orbit is in the closure of another, both are necessarily mapped to the same point of 
$X\gitq G$. GIT gives a systematic method for constructing projective varieties which can be thought of as best-
possible approximations to the desired quotient space. The GIT construction is not unique and depends on a choice of 
\emph{linearization} of the action \cite[p. 30]{GIT}. A linearization of the $G$-action simply consists of an ample line 
bundle $\L$ on $X$, and a $G$-action on the total space of $\L$ which is compatible with the given $G$-action on $X$. 
More precisely, if the action on $X$ is given by a morphism $G \times X \rightarrow X$, and the action on the line bundle 
by a morphism $G \times \underline{\Spec}_{\O_{X}} \Sym^{\ast} \L \rightarrow \underline{\Spec}_{\O_{X}} \Sym^{\ast} 
\L$, then we require the following diagram to commute:
\[
\xymatrix{
G \times  \underline{\Spec}_{\O_{X}}\! \Sym^{\ast} \L \ar[r] \ar[d]& \underline{\Spec}_{\O_{X}}\! \Sym^{\ast} \L\ar[d]\\
G \times X \ar[r] &  X\\
}
\]
Given a linearization of the $G$-action on $X$, the GIT quotient is constructed as follows. First, note that since $\L$ is ample, we have $$X=\Proj \bigoplus_{m \geq 0}\HH^0(X, \L^m).$$ The linearization gives a
$G$\nb-action on $\bigoplus_{m \geq 0}\HH^0(X, \L^m)$, so we may consider the ring of invariants
$$
\bigoplus_{m \geq 0}\HH^0(X, \L^m)^G \subset \bigoplus_{m \geq 0}\HH^0(X, \L^m).
$$
The GIT quotient is simply defined to be the associated rational map
$$
\mathrm{q}\co X= \Proj \bigoplus_{m \geq 0}\HH^0(X, \L^m) \dashrightarrow \Proj \bigoplus_{m \geq q}\HH^0(X, \L^m)^G.
$$
In order to describe the properties of this rational map, it is useful to make the following definitions.
\begin{definition}\label{D:stable-semistable} The \emph{semistable locus} and \emph{stable locus} of the linearized $G$-action on $X$ are defined by
\begin{align*}
X^{\ss}:=&\{x \in X \ | \ \exists f \in \HH^0(X, \L^{m})^{G} \text{ such that  $f(x)\neq 0$} \}, \\
X^{\s}:=&\{x \in X^{\ss}\  | \ G\cdot x \text{ is closed in $X^{\ss}$ and $\dim G \cdot x=\dim G$}\}.
\end{align*}
\end{definition}
It is elementary to check that $X^{\s} \subset X^{\ss} \subset X$ is a sequence of $G$\nb-invariant open immersions, 
and that $X^{\ss}$ is the locus where $\mathrm{q}$ is regular.
For this reason, it is customary to denote the GIT quotient by $X^{\ss} \gitq G$. Now we have a diagram:
\begin{equation}
\xymatrix{
X^\s \ar[d]^{\mathrm{q}} \ar[r] & X^{\ss}\ar[d]^{\mathrm{q}} \ar[r] & X\ar@{-->}[ld] \\
X^\s\gitq G \ar[r] & X^{\ss}\gitq G & 
}
\end{equation}
Here, $\mathrm{q}$ is a geometric quotient on the stable locus $X^{\s}$ and a categorical quotient on $X^{\ss}$ \cite[pp. 4, 38]{GIT}. This is made precise in the following proposition.
\begin{proposition}\label{P:stable-semistable} With notation as above, we have
\begin{enumerate}
\item For any $z \in X^{s} \gitq G$, the fiber $\phi^{-1}(z)$ is a single closed $G$-orbit.
\item For any $z \in X^{\ss} \gitq G$, the fiber $\phi^{-1}(z)$ contains a unique closed $G$-orbit.
\end{enumerate}
\end{proposition}
\begin{proof}
Since $G$ is linearly reductive, every $G$\nb-representation is completely 
reducible. Hence, taking invariants is an exact functor on the category of $G$-representations. In particular, if $Z_1, Z_2$ are disjoint $G$\nb-invariant closed subschemes of a distinguished open affine $X_f$, for some $f\in \HH^0(X,\L^k)^G$, 
then 
the surjection of $G$\nb-representations
$$\HH^0(X_f,\L^m) \ra \HH^0(Z_1,\L^m\vert_{Z_1}) \oplus \HH^0(Z_2, \L^m\vert_{Z_2}) \ra 0$$
gives rise to
$$\HH^0(X_f,\L^m)^G \ra \HH^0(Z_1,\L^m\vert_{Z_1})^G\oplus \HH^0(Z_2, \L^m\vert_{Z_2})^G \ra 0.
$$
By considering the preimages of $(1,0)$ and $(0,1)$, and scaling by a power of $f$, we deduce the existence of $G$\nb-invariant
sections of $\L^n$, $n\gg 0$, separating $Z_1$ and $Z_2$. 

Note that if $x\in X^{\s}$, the orbit 
$G\cdot x$ is closed of maximal dimension inside some $X_f$, hence for any other point $y\notin G\cdot x$, we have
$G\cdot x\cap \overline{G\cdot y}=\varnothing$. 
Both (1) and (2) now follow from the just established fact that
$G$\nb-invariant sections of $\L^n$, $n\gg 0$,
separate any two closed disjoint $G$\nb-invariant subschemes of $X^{\ss}$.
\end{proof}

While Definition \ref{D:stable-semistable} in theory determines the semistable locus $X^{\ss}$,
it requires the knowledge of {\em all} invariant sections of $\L^m$ that we rarely possess in practice. In order to use GIT, we clearly need a more algorithmic method to determine the stability of points in $X$. It was for this purpose that Mumford developed the so-called {\em numerical criterion for stability} \cite[Chapter 2.1]{GIT}. The first step is to observe that $x\in X$ is semistable if and only if for every nonzero lift $\bar{x}$ of $x$ to the affine cone $\Spec \bigoplus_{m\geq 0} \HH^0(X,\L^m)$, the closure of the $G$-orbit of $\bar{x}$ does not contain the origin. Mumford's key insight was that {\em this} property can be tested by looking at one-parameter subgroups of $G$, i.e. all subgroups $k^* \subset G$.
More precisely, if we are given any one-parameter subgroup $\rho\co k^* \ra G$, we may diagonalize the action of $\rho$ on 
$\HH^0(X,\L)$ and hence on the 
$\L$-coordinates $(x_0,\ldots, x_n)$ of $\overline{x}$. We then define the {\em Hilbert-Mumford index} of $x$ with respect to $\rho$
by 
\[
\mu_{\rho}(x):=\min_{x_i\neq 0} \{w \ : \ \rho(t)\cdot x_i=t^wx_i\},
\]
and we say that $x\in X$
is {\em stable} (resp. {\em semistable, nonsemistable}) {\em with respect to $\rho$} if $\mu_{\rho}(x)<0$ (resp. $\mu_{\rho}(x)\leq 0$, $\mu_{\rho}(x)>0$). With this terminology, Mumford's numerical criterion is easy to state.

\begin{proposition}[Hilbert-Mumford numerical criterion]\label{P:numerical}
A point $x\in X$ is {\em stable} (resp. {\em semistable}) if and only if it is stable (resp. semistable) with respect to all one-parameter subgroups of $G$. A point $x\in X$ is {\em nonsemistable} if it is $\rho$-nonsemistable for some one-parameter subgroup $\rho$.
\end{proposition} 
In words, $x\in X$ is stable if {\em every} one-parameter subgroup of $G$ acts on the nonzero $\L$-coordinates of $x$ with both positive and negative weights, and $x\in X$ is nonsemistable if there exists a one-parameter subgroup 
of $G$ acting on the nonzero $\L$-coordinates of $x$ with either all positive or all negative weights. We will see an example of how to use the numerical criterion in Example \ref{E:NCriterion}. First, however, let us step back and explain how the entire geometric invariant theory is applied to construct compactifications of the moduli space of curves.

First, fix 
an integer $n \geq 2$. If $C$ is any smooth curve of genus $g \geq 2$, a choice of basis for the vector space $
\HH^0(C, \omega_{C}^n)$ determines an embedding
$$
|\omega_{C}^{n}|:C \hookrightarrow \P^{N},
$$
where $N=(2n-1)(g-1)-1$. The subscheme $C \subset \P^{N}$ determines a point $[C] \in \Hilb_{P(x)}(\PP^{N})$, the 
Hilbert scheme parametrizing subschemes of $\P^{N}$ with Hilbert polynomial $P(x):=n(2g-2)x+(1-g).$ Now let
$$
\Hilb_{g,n} \subset \Hilb_{P(x)}(\P^N)
$$
denote the locally closed subscheme of all such $n$-canonically embedded smooth curves, and let $\overline{\Hilb}_{g,n}$ denote the closure of $\Hilb_{g,n}$ in $\Hilb_{P(x)}(\P^N)$. 

The natural action of $\SL(N+1)$ on $\P^{N}$ induces an action on $\Hilb_{P(x)}(\P^N)$, hence also on $\overline{\Hilb}_{g,n}$ and $\Hilb_{g,n}$.
\begin{exercise} Verify that the action of $\SL(N+1)$ on $\Hilb_{g,n}$ is proper and the stack quotient 
$\left[\Hilb_{g,n} / \SL(N+1)\right]$ is canonically isomorphic to $\mathcal{M}_{g}$.
\end{exercise}
Our plan therefore is to apply GIT to the action of $\SL(N+1)$ on the projective variety 
$\overline{\Hilb}_{g,n}$ and \emph{hope} that $\Hilb_{g,n}$ is contained in the corresponding stable locus $\overline{\Hilb}_{g,n}^{\, \s}$. If this can be verified then the GIT quotient 
$\overline{\Hilb}_{g,n}^{\, \ss} \gitq \SL(N+1)$ will give a compactification of $M_{g}$. 

In order to apply GIT, we must linearize the $\SL(N+1)$-action on $\overline{\Hilb}_{g,n}$. This is accomplished as follows. First, by Castelnuovo-Mumford regularity \cite[Lecture 14]{mumford-curves}, there exists an integer 
$m \geq 0$ such that $\HH^1(\P^N, I_{C}(m))=0$ for \emph{any} curve $C \subset \P^{N}$ with Hilbert polynomial $P(x)$. Now set 
$$W_m=\bigwedge^{P(m)} \HH^0\left(\PP^N, \O_{\P^N}(m)\right).$$
We obtain an embedding
$$\Hilb_{P(x)}(\P^N)\hookrightarrow \PP W_m$$
as follows: Corresponding to a point $[C] \subset \Hilb_{P(x)}(\P^N)$, we have the surjection 
\begin{align}\label{E:hilbert-point}
\HH^0(\P^N,\O_{\P^N}(m)) \rightarrow \HH^0(C, \O_{C}(m)) \rightarrow 0,
\end{align}
which is called the {\em $m^{th}$ Hilbert point} of $C\hookrightarrow \PP^N$. 
Taking the $P(m)^{th}$ exterior power of this sequence gives a one-dimensional quotient of $W_m$, hence a point of 
$\P(W_m)$. Now $\O_{\P(W_m)}(1)$ restricts to an ample line bundle on $\Hilb_{P(x)}(\P^N)$, and the $\SL(N+1)$-action extends to an action on the total space of $\O_{\P(W_m)}(1)$ because $W_m$ is an $\SL(N+1)$\nb-representation.
Corresponding to this linearization, we obtain a GIT quotient $\overline{\Hilb}_{g,n}^{\, \ss, m} \gitq \SL(N+1).$

This is an undeniably slick construction, but one glaring question remains: Can we actually determine the stable and semistable locus $\overline{\Hilb}_{g,n}^{\, \s, m} \subset \overline{\Hilb}_{g,n}^{\, \ss, m} \subset \overline{\Hilb}_{g,n}$? Before addressing this issue, let us take a step back and review the following two aspects of the construction that we've just presented
\begin{enumerate}
\item What choices 
went into the construction of $\overline{\Hilb}_{g,n}^{\, \ss, m}\gitq SL(N+1)$?
\item To what extent is $\overline{\Hilb}_{g,n}^{\, \ss, m}\gitq \SL(N+1)$ modular?
\end{enumerate}

Regarding (1), the construction depended on two numerical parameters: the integer $n$ which determined the Hilbert scheme, and the integer $m$ which determined the linearization. An interesting problem, which will be addressed presently, is to understand how the quotient $\overline{\Hilb}_{g,n}^{\, \ss, m}\gitq \SL(N+1)$ changes as one varies the parameters $n$ and $m$.

Regarding (2), it is not a priori obvious that one should be able to identify $\left[\overline{\Hilb}_{g,n}^{\, \ss, m} / \SL(N+1)\right]$ with an open substack of $\U_g$, the stack of genus $g$ curves. It turns out, however, that under relatively mild hypotheses on $\overline{\Hilb}_{g,n}^{\, \ss, m}$, this will be the case. More precisely, we have

\begin{proposition}\label{P:ModularInterp}
Assume that
\begin{enumerate}
\item $\overline{\Hilb}_{g,n}^{\, \ss, m} \neq \varnothing$,
\item Every point $[C] \in \overline{\Hilb}_{g,n}^{\, \ss, m}$ corresponds to a Gorenstein curve.
\end{enumerate}
Then we have
\begin{enumerate}
\item $\overline{\Hilb}_{g,n}^{\, \ss, m} \gitq \SL(N+1)$ is a weakly modular birational model of $M_{g,n}$,
\item If $\overline{\Hilb}_{g,n}^{\, \ss, m}=\overline{\Hilb}_{g,n}^{\, \s, m}$, then $\overline{\Hilb}_{g,n}^{\, \ss, m} \gitq \SL(N+1)$ is a modular birational model of $M_{g,n}$.
\end{enumerate}
\end{proposition}
\begin{proof}
Consider the universal curve
\[
\xymatrix{
\C \ar[r] \ar[d]^{\pi}&\P^{N} \times \overline{\Hilb}^{\, \ss, m}_{g,n} \ar[dl]\\
\overline{\Hilb}^{\, \ss, m}_{g,n}&\\
}
\]
Since the fibers of $\pi$ are Gorenstein, the relative dualizing sheaf $\omega_{\C/\overline{\Hilb}^{\, \ss, m}_{g,n}}$ is a line bundle, and we claim that $\omega_{\C/\overline{\Hilb}^{\, \ss, m}_{g,n}}^{n} \simeq \O_{\P^N}(1) \otimes \pi^*\L$ for some line bundle $\L$ on $\overline{\Hilb}^{\, \ss, m}_{g,n}$. This is immediate from the fact that $\omega_{\C/\overline{\Hilb}^{\, \ss, m}_{g,n}}^{n} \simeq \O_{\P^N}(1)$ on the fibers of $\C$ over the nonempty open set $\Hilb_{g,n} \cap \overline{\Hilb}_{g,n}^{\, \ss, m}$.

Now consider the natural map from $\overline{\Hilb}^{\, \ss, m}_{g,n}$ to the stack $\U_{g}$ (of all reduced connected genus $g$ curves) induced by the universal family. Since any two $n$-canonically embedded curves $C \subset \P^N$ and $C' \subset \P^N$ are isomorphic iff there exists a projective linear transformation taking $C$ to $C'$, this map factors through the $\SL(N+1)$\nb-quotient to give $\left[ \overline{\Hilb}^{\, \ss, m}_{g,n} / \SL(N+1) \right] \rightarrow \U_{g}$. We claim that this map is an open immersion. Equivalently, that $\overline{\Hilb}^{\, \ss, m}_{g,n}$ parameterizes a deformation open class of curves. For this, we must see that if $[C] \in \overline{\Hilb}^{\, \ss, m}_{g,n}$ then every abstract deformation of $C$ can be realized as an embedded deformation of $C \subset \P^N$. By \cite[Corollary B.8]{Z-stability}, this follows from the fact that $\HH^1(C, \O_{C}(1))=\HH^1(C, \omega_{C}^n)=0$ for $n \geq 2$.

Thus, we may identify $\left[\overline{\Hilb}_{g,n}^{\, \ss, m} / \SL(N+1)\right]$ with an open substack of the stack of curves, and we have a diagram 
\[
\xymatrix{
\U_{g} & \\
\left[\overline{\Hilb}_{g,n}^{\, \ss, m} / \SL(N+1)\right] \ar[r]^{\phi} \ar[u]^{i}&\overline{\Hilb}_{g,n}^{\, \ss, m} \gitq \SL(N+1)
}
\]
It only remains to check that $\phi$ is a good moduli map. The first property of good moduli maps, that they are categorical with respect to algebraic spaces, is a general feature of GIT \cite{GIT}, \cite{alper}. The second property, namely that $\phi^{-1}(x)$ contains a unique closed point for each $x \in \overline{\Hilb}_{g,n}^{\, \ss, m} \gitq \SL(N+1)$ is an immediate consequence of Proposition \ref{P:stable-semistable}.

\end{proof}

Finally, it remains to consider the problem of how one actually computes the semistable locus $$\overline{\Hilb}^{\, \ss, m}_{g,n}  \subset \Hilb_{P(x)}(\P^N).$$ Historically, the important breakthrough was Gieseker's asymptotic stability result: Gieseker showed that for any $n \geq 10$, a generic smooth $n$-canonically embedded curve $C \subset \P^N$ is asymptotically Hilbert stable \cite{gieseker} and Mumford extended this result\footnote{Strictly speaking, Mumford
works with the Chow variety.}
 to $5$\nb-canonically embedded curves \cite{mumford-stability}:\begin{theorem}
For $n \geq 5$ and $m \gg 0$, 
$$\overline{\Hilb}^{\, \ss, m}_{g,n} =\overline{\Hilb}^{s,m}_{g,n}=\{[C] \in \Hilb_{P(x)}(\P^N) \,|\, \emph{$C \subset \P^N$ is stable and $\O_{C}(1) \simeq \omega_{C}^n$ }\}.$$
Equivalently, $\overline{\Hilb}^{\, \ss, m}_{g,n} \gitq \SL(N+1) \simeq \M_{g}$. 
\end{theorem}
Rather than describe the proof of this theorem, which is well covered in other surveys (e.g. \cite{HM} and \cite{morrison-git}), we shall consider here the natural remaining question: what happens for smaller values of $n$ and $m$? Recently, Hassett, Hyeon, Lee, Morrison, and Swinarski have taken the first steps toward understanding this problem, and a beautfiul picture is emerging in which the GIT quotients corresponding to various values of $n$ and $m$ admit natural interpretations as log canonical models of $\M_{g}$
\cite{HH1, HH2, HL, HHL, morrison-swinarski}.  Before describing their results in detail, however, let us present an important heuristic for predicting which curves $C \subset \P^N$ \emph{ought} to be semistable for given values of $n$ and $m$. 
It is worth nothing that, in practice, having an educated guess for what the semistable locus $\overline{\Hilb}^{\, \ss, m}_{g,n}$ ought to be is the most important step in actually describing it.

To begin, let us assume that $\overline{\Hilb}_{g,n}^{\, \ss, m}$ satisfies the hypotheses
\begin{enumerate}
\item $\overline{\Hilb}_{g,n}^{\, \ss, m} \neq \varnothing$,
\item Every point $[C] \in \overline{\Hilb}_{g,n}^{\, \ss, m}$ corresponds to a Gorenstein curve,
\item The locus of moduli non-stable curves in  
$\overline{\Hilb}_{g,n}^{\, \ss, m} \gitq \SL(N+1)$ has codimension at least two.
\end{enumerate}
Under assumptions (1) and (2), Proposition \ref{P:ModularInterp} says that we may identify the quotient stack $\left[ \overline{\Hilb}^{\, \ss, m}_{g,n} / \SL(N+1)\right]$ with a weakly modular birational model of $\M_{g}$. Assumption (3) implies that we have a birational contraction (cf. Definition \ref{D:contraction})
$$
\phi\co \M_{g} \dashrightarrow \overline{\Hilb}^{\, \ss, m}_{g,n} \gitq \SL(N+1). 
$$
The key idea is to interpret this contraction as the rational map associated to a certain divisor on $\M_{g}$. To set this up, first note that we may define \emph{Weil divisor} classes $\lambda$ and $\delta$ on $\overline{\Hilb}^{\, \ss, m}_{g,n} \gitq \SL(N+1)$ simply by pushing forward the corresponding classes from $\M_{g}$. (The fact that $\phi$ is a birational contraction implies that push-forward of Weil divisors is well-defined;  since we have \emph{a priori}  no control over the singularities of $\overline{\Hilb}^{\, \ss, m}_{g,n} \gitq \SL(N+1)$, $\lambda$ and $\delta$ may very well fail to be Cartier.) Next, we observe that the divisor class of the polarization of $\overline{\Hilb}^{\, \ss, m}_{g,n} \gitq \SL(N+1)$ coming from the GIT construction can be computed as an explicit linear combination of $\lambda$ and $\delta$. 
\begin{proposition}[\text{\cite[Section 5]{HH2}}]
\label{P:Hilbert-polarization}
Suppose that assumptions (1)--(3) above hold. Then there is an ample line bundle 
$\Lambda_{m, n}$ on $\overline{\Hilb}^{\, \ss, m}_{g,n} \gitq \SL(N+1)$ with divisor class given by:
 \begin{align}\label{E:git-hilbert-polarizations}
& \Lambda_{m,n}=
 \begin{cases}
\lambda+(m-1)\left[ (4g+2)m-g+1)\lambda-\frac{gm}{2}\delta\right], & \ \text{if $n=1$}, \\
 (m-1)(g-1)\left[ (6mn^2-2mn-2n+1)\lambda-\frac{mn^2}{2}\delta \right], & \ \text{if $n>1$.}\\
 \end{cases}
 \end{align}
\end{proposition}
\begin{proof}
Consider the universal curve
\[
\xymatrix{
\C \ar[r] \ar[d]^{\pi}& \P^N \times \overline{\Hilb}^{\, \ss, m}_{g,n} \ar[dl]\\
\overline{\Hilb}^{\, \ss, m}_{g,n} &\\
}
\]
As we have described above, the linearization on the Hilbert scheme comes from
its embedding into $\PP(W_m)$, where $W_m=\bigwedge^{P(m)} \HH^0\left(\PP^N, \O_{\P^N}(m)\right))$. It follows from the construction of this embedding that the tautological quotient line bundle $\O_{\PP(W_m)}(1)$ 
pulls back to the line bundle $\bigwedge^{P(m)}\pi_*\O_{\C}(m)$ on $\overline{\Hilb}_{g,n}$.
Hence, $\bigwedge^{P(m)}\pi_*\O_{\C}(m)$
is an $\SL(N+1)$-equivariant line bundle on 
$\overline{\Hilb}^{\, \ss, m}_{g,n}$ which descends to an ample line bundle, denoted $\Lambda_{m,n}$, on the quotient. 
To compute $c_1\bigr(\bigwedge^{P(m)}\pi_*\O_{\C}(m) \bigl)$, and hence the class of  $\Lambda_{m,n}$, one 
observes that 
\begin{enumerate}
\item[(a)] $\O_{\C}(1)=\omega_{\pi}^{n}\otimes \pi^*\L$ for some line bundle $\L$ on $\overline{\Hilb}^{\, \ss, m}_{g,n}$ (this is simply the definition of being a family of $n$\nb-canonically embedded curves), 
\item[(b)] $\pi_*\O_{\C}(1)$ is a trivial $\PP^N$\nb-bundle (this says that all curves 
in $\overline{\Hilb}^{\ss,m}_{g,n}$ are embedded by complete linear systems).
\end{enumerate}
 As we've remarked in the proof of Proposition \ref{P:ModularInterp}, assumptions (1)-(2) imply that $R^1\pi_*\O_\C(1)=0$. The class of $\Lambda_{m,n}$ now can be computed using the Grothendieck-Riemann-Roch formula.
\end{proof}
\begin{corollary}\label{C:GIT-ample}
If $\Lambda_{m,n}$ denotes the divisor class on $\M_{g}$ as above, then $\phi_*\Lambda_{m,n}$ is Cartier and ample on $\overline{\Hilb}^{\, \ss, m}_{g,n} \gitq \SL(N+1)$.
\end{corollary}
\begin{proof}
This is just a restatement of the proposition.
\end{proof}

According to the corollary, $\M_{g} \dashrightarrow \overline{\Hilb}^{\, \ss, m}_{g,n} \gitq \SL(N+1)$ will 
be precisely the contraction associated to $\Lambda_{m,n}$ once we verify that 
\begin{align}
\Lambda_{m,n}-\phi^*\phi_*\Lambda_{m,n} \geq 0 \in \N^1(\M_g).
\end{align}
We may now state our heuristic informally as: 
\begin{principle}[GIT Heuristic]\label{principle-git}
 The singularities of curves appearing in the GIT quotient of $n$-canonically embedded curves using the 
 $m^{th}$ Hilbert point polarization should be precisely those whose variety of stable limits in $\M_{g}$ is covered by curves on which $\Lambda_{m,n}$ is non-positive.
 \end{principle}

For a more precise description of the {\em variety of stable limits} associated to a singularity, as well as a more detailed justification for this heuristics, we refer the reader to Section \ref{S:heuristic}. Here, let us just give a simple example. Using the formulae for $\Lambda_{m,n}$, the reader may easily check that  
$$
\lim_{m\to \infty} \Lambda_{m,3} \sim \frac{32}{3}\lambda-\delta.
$$
Now it is easy to see that this line bundle is negative on the family of varying elliptic tails
 (cf. Example \ref{E:elliptic-bridges}), 
which is precisely the variety of stable limits of the cusp (cf. Example \ref{E:stable-reduction-A2k}). Our heuristic thus suggests that for $m \gg 0$, cuspidal curves should be contained in the semistable locus $\overline{\Hilb}_{g,3}^{\, \ss, m}$. In Section \ref{S:test-families}, we describe the variety of stable limits of any ADE singularity (among others) and describe the threshold slopes at which these varieties are covered by curves on which $s\lambda-\delta$ is negative. Using these computations, it is easy to compute the threshold values of $m$ and $n$ at which we may expect curves with given singularities to become Hilbert (semi)stable. 

The intersection theory heuristic, however useful, 
can be used only as a guide and will not by itself establish stability of any given curve. 
Methods used by Gieseker to 
prove stability of $5$-canonically embedded stable curves are asymptotic in nature and 
are of no help in the case of finite Hilbert stability. Having said this, we note that
there is also a precise if slightly unwieldy method for determining stability of Hilbert points 
of an embedded curve $C$. 
By tracing through the construction of the $m^{th}$ Hilbert point in Equation \eqref{E:hilbert-point} and the definition 
of the resulting linearization on 
$\overline{\Hilb}_{g,n}^{\, m}$, and applying Proposition \ref{P:numerical} we deduce the following result 
(see \cite[Proposition 4.23]{HarMor}).
\begin{proposition}[Stability of Hilbert points]\label{P:hilbert-stability}
The $m^{th}$ Hilbert point 
of an embedded curve $C\subset \PP^N$ is stable (resp. semistable) 
if and only if for every one-parameter subgroup $\rho$ of $\SL(N+1)$ and a basis $(x_0,\dots, x_N)$
of $\HH^0(C, \O_C(1))$ diagonalizing the action of $\rho$, there is a {\em monomial basis}
of $\HH^0(C,\O_C(m))$
consisting of degree $m$ monomials in $x_0,\dots, x_N$ such that the sum of the $\rho$-weights of these 
monomials is negative (resp. non-positive). 
\end{proposition}
While the criterion of Proposition \ref{P:hilbert-stability} gives in theory a way to prove stability, 
in practice it can be efficiently used only to prove nonsemistability or to prove stability with respect
to a fixed torus. Occasionally, in the presence of nontrivial automorphisms, 
the theory of worst one-parameter subgroups developed by Kempf \cite{kempf} 
allows to reduce verification of stability to a fixed torus, where the numerical criterion can be applied. 
This technique is used in \cite{morrison-swinarski} to prove finite Hilbert stability of certain low genus 
bicanonical curves.  We will settle here for an illustration
of how the numerical criterion of Proposition \ref{P:hilbert-stability} 
can be used to {\em destabilize} a curve. We do so on the example of a generically non-reduced
curve of arithmetic genus $4$. 

\begin{example}[Genus $4$ ribbon]\label{E:NCriterion}
A ribbon is a non-reduced curve locally isomorphic to $\AA_1\times \spec k[\epsilon]/(\epsilon^2)$.
We refer to \cite{bayer-eisenbud} for the general theory of such curves and focus here on the ribbon $C$ of genus $4$ 
defined by the ideal $$I_C=(xz-y^2, z^3+xw^2-2yzw)$$ in $\PP^4$. This curve is interesting because it is 
a flat limit of one-parameter families of canonically embedded curves whose stable limit is a hyperelliptic curve in $\Mg{4}$.
Since hyperelliptic curves are not canonically embedded, $C$ is a natural candidate for a semistable point 
that replaces hyperelliptic curves in the GIT quotient $\overline{\Hilb}_{4,1}^{\, m, \ss}$. However, as the following
application of the numerical criterion shows, 
$C$ has nonsemistable $m^{th}$ Hilbert point for all $m$.
(This was first observed by Fong \cite[p.298]{fong}; see also \cite[Example 8.10]{morrison-swinarski}.)

The generators of $I$ are homogeneous with respect to the one-parameter subgroup $\rho\co \spec k[\lambda,\lambda^{-1}] \ra \SL(4)$ acting by 
$$\lambda\cdot (x,y,z,w)=(\lambda^{-3}x, \lambda^{-1} y, \lambda z, \lambda^{3}w).$$
Since $\rho$ is a subgroup of $\Aut(C)\subset \SL(4)$,
every monomial basis
of $\HH^0(C,\O_C(m))$ has the same $\rho$-weight, which we now compute. We pick the basis of 
$\HH^0(C,\O_C(m))$
consisting of monomials not divisible by $y^2$ and $xw^2$:
\begin{multline*}
\{ x^iz^{m-i} \}_{i=0}^m, \{x^iz^{m-i-1}w\}_{i=0}^{m-1}, \{ z^{m-i}w^i \}_{i=2}^m, \\
 \{x^iyz^{m-i-1} \}_{i=0}^{m-1}, \{x^iyz^{m-i-2}w\}_{i=0}^{m-2}, \{ yz^{m-i-1}w^i \}_{i=2}^{m-1}.
\end{multline*}
 The $\rho$-weights of the monomials in this basis sum up to $3m-4>0$.
We conclude by Proposition \ref{P:hilbert-stability} 
that the $m^{th}$ Hilbert point of $X$ is {\em nonsemistable for all $m$}.  
\end{example}

Having discussed heuristics for verifying finite Hilbert stability and having seen one example 
of the numerical criterion in action, we now summarize what is known and what is yet to be done.
We display the known results on the GIT quotients of $n$-canonically embedded curves and the predictions based on Principle \ref{principle-git} in Table 1. 
In order to understand the results displayed there, we need only to make three additional remarks: First, one can allow fractional values of the polarization;  
 these are discussed in \cite[Remark 3.2]{morrison-swinarski}.
 Second, rather than using the Hilbert scheme, one can also use the Chow variety of $n$\nb-canonically embedded curves. 
 Following a procedure similar to that which we have sketched above, one obtains the following ample divisor on $\overline{\Chow}^{\, \ss}_{n}\gitq \SL(N+1)$ (see. \cite[Theorem 5.15]{mumford-stability} or \cite[Proposition 5.2]{HH2}):
  \begin{align}
 \label{E:git-chow-polarizations}
\begin{cases} (4g+2)\lambda-\frac{g}{2}\delta, & \ \text{if $n=1$}, \\
 (g-1)n\left[ (6n-2)\lambda-\frac{n}{2}\delta\right], & \ \text{if $n>1$.}
 \end{cases}
 \end{align}
Third, we must describe the meaning of the notation $A_{k}^{\dag}$ and $D_k^{\dag}$ that appears in Table 1. 
Note that, in general, the semistability of an $n$-canonically embedded curve is not a local analytic question: It depends not only on the singularities of the curve but also on their global arrangement. For example, 
the classic stability analysis of plane quartics shows that a quartic with a node ($A_1$) and an oscnode ($A_5$) is semistable if it is a union of two conics, and not semistable if it is a union of a cubic and its flex line \cite[Chapter 3.2]{GIT}. In fact, Principle \ref{principle-git} can be refined to understand the thresholds at which different global arrangements of singularities appear, provided one understands the different varieties of stable limits associated to these different arrangements. For the purposes of Table 1, 
the one essential global distinction is the following:
\begin{definition}[Dangling $A_{2k-1}$ and $D_{2k+2}$] \label{D:danglingA}
We say that a proper curve 
$C$ of genus $g$ 
has a {\em dangling $A_{2k-1}$-singularity (or simply, an $A_{2k-1}^{\dag}$-singularity)} if $C=R\cup S$ where $R\simeq \PP^1$, $p:=R\cap S$ in a smooth point of $S$ and the intersection multiplicity of $R$ and $S$ is $k$ ($k \geq 3$). Similarly, we say that a proper curve
$C$ of genus $g$ 
has a {\em dangling $D_{2k+2}$-singularity (or simply, a $D_{2k+2}^{\dag}$-singularity)} if $C=R\cup S$ 
where $R\simeq \PP^1$, $p:=R\cap S$ is a node of $S$ and the intersection multiplicity of $R$ with one of the branches of $S$ at $p$ is $k$ ($k \geq 2$).
\end{definition}
\begin{exercise} Show that a general genus $g\geq k$ curve $C$ with an $A_{2k-1}^{\dag}$ or $D_{2k+2}^{\dag}$-singularity has $\omega_{C}$ ample and $\Aut(C)$ finite.
\end{exercise}
It is not difficult to see that the locus of stable limits of a curve with an $A_{2k-1}^{\dag}$ or $D_{2k+2}^{\dag}$-singularity is covered by curves with different numerical properties than the locus of stable limits of an irreducible curve with an $A_{2k-1}$ or $D_{2k+2}$-singularity. This is why these two curves must be considered separately and arise at different threshold values in Table 1. 

\begin{table}[p]\label{T:GITPredictions}
\begin{center}
\renewcommand{\arraystretch}{2}
\begin{tabular}{|p{1.3in}|c|c|c|c|}
\hline 
GIT quotient & Status & Moduli space & Singularities \\
\hline
$\overline{\Hilb}^{\, \ss, m}_{g,n} \gitq \SL(N+1)$, $n \geq 5, m\gg 0$
& classical                                                                                                               & $\M_{g}$                                                                             & $A_1$ \\
 \hline

$\overline{\Chow}^{\, \ss}_{g,4} \gitq \SL(7g-7)$         
& established \cite{hyeon-morrison}                                                                    &$\M_{g}(\frac{9}{11})=\M_{g}\left[A_2^* \right]$         & $A_1-A_2$ \\
\hline
 
$\overline{\Hilb}^{\, \ss, m}_{g,4} \gitq \SL(7g-7)$, $m\gg0$          
& established \cite{hyeon-morrison}                                                                 &$\M_{g}(\frac{9}{11}-\epsilon)=\M_{g}\left[A_2\right]$         & $A_1-A_2$ \\
\hline

$\overline{\Chow}^{\, \ss}_{g,3}\gitq \SL(5g-5)$
& established \cite{Schubert}                                                             & $\M_{g}(\frac{25}{32})=\M_{g}\left[A_2\right]$                     & $A_1-A_2$  \\
\hline

$\overline{\Hilb}^{\, \ss, m}_{g,3} \gitq \SL(5g-5)$,  $m\gg 0$        
& conjectural                                                               &$\M_{g}(\frac{25}{32}-\epsilon)=\M_{g}\left[A_2\right]$         & $A_1-A_2$ \\
\hline
 
$\overline{\Chow}^{\, \ss}_{g,2}\gitq \SL(3g-3)$
& established \cite{HH2}                                                                                 & $\M_{g}(\frac{7}{10})=\M_{g}\left[A_3^* \right]$          & $A_1-A_3$  \\
\hline
 
$\overline{\Hilb}^{\, \ss, m}_{g,2}\gitq \SL(3g-3)$,  $m\gg 0$
& established \cite{HH2}                                                                                 & $\M_{g}(\frac{7}{10}-\epsilon)=\M_{g}\left[A_3\right]$          & $A_1-A_3$  \\
\hline

 $\overline{\Hilb}^{\, \ss, 6}_{g,2}\gitq \SL(3g-3)$ 
& conjectural                                                                                                            & $\M_{g}(\frac{2}{3})$                   & $A_1-A_4$  \\
\hline

$\overline{\Hilb}^{\, \ss, \frac{9}{2}}_{g,2}\gitq \SL(3g-3)$ 
& conjectural                                                                                                            & $\M_{g}(\frac{19}{29})$                & $A_1-A_4, A_{5}^{\dag}$ \\                                                                                                                                                                                
\hline

$\overline{\Hilb}^{\, \ss, \frac{9}{4}}_{g,2}\gitq \SL(3g-3)$ 
& conjectural                                                                                                           & $\M_{g}(\frac{17}{28})$                    & $A_1-A_5$  \\
\hline

$\overline{\Hilb}^{\, \ss, \frac{27}{14}}_{g,2}\gitq \SL(3g-3)$ 
& conjectural                                                                                                            & $\M_{g}(\frac{49}{83})$                    & $A_1-A_6$  \\
\hline
       $\overline{\Hilb}^{\, \ss, \frac{3}{2}}_{g,2}\gitq \SL(3g-3)$ 
       & conjectural                                                                                                     & $\M_{g}(\frac{5}{9})$                    & $A_1-A_6, D_4, D_{5}^{\dag}, D_{6}^{\dag}$ \\

  \hline
 & $\cdots$ &   $\cdots$ & $\cdots$   \\
\hline
$\overline{\Chow}^{\, \ss}_{g,1}\gitq \SL(g)$
& conjectural                                                                               & $\M_{g}(\frac{3g+8}{8g+4})$        & \multirow{2}{1in}{ADE, Ribbons, $X_9$, $J_{10}$, $E_{12}$, Others?}  \\
& & & \\
\hline

$\overline{\Hilb}^{\, \ss, m}_{g,1}\gitq \SL(g)$,  $m\gg 0$
& conjectural                                                                        & $\M_{g}(\frac{3g+8}{8g+4}-\epsilon)$        &  \multirow{2}{1in}{ADE, Ribbons, $X_9$, $J_{10}$, $E_{12}$, Others?}   \\
& & & \\
  \hline
  
\end{tabular}
\smallskip
\caption{GIT quotients of $n$-canonically embedded curves}
\end{center}
\end{table}

\subsubsection{Established results on GIT} 

As noted in Table 1, 
the semistable locus has $\overline{\Hilb}_{g,n}^{\, \ss, m}$ has actually been completely determined in a number of cases
by Hassett, Hyeon, Lee, Morrison and Schubert. To finish out our discussion of GIT, we present their 
descriptions of these loci and the connections between the resulting moduli spaces.

\begin{definition} Let $C$ be any curve. We say that
\begin{enumerate}
\item \emph{$C$ has an elliptic tail} if there exists a morphism $i\co (E,p) \hookrightarrow C$ satisfying
\begin{itemize}
\item $(E,p)$ is a $1$-pointed curve of arithmetic genus one,
\item $i|_{E-p}$ is an open immersion.
\end{itemize}
\item \emph{$C$ has an elliptic bridge} if there exists a morphism $i\co (E,p,q) \hookrightarrow C$ satisfying
\begin{itemize}
\item $(E,p,q)$ is a $2$-pointed curve of arithmetic genus one,
\item $i|_{E-\{p,q\}}$ is an open immersion,
\item $i(p)$, $i(q)$ are nodes of $C$.
\end{itemize}
\item \emph{$C$ has an elliptic chain} if there exists a morphism $i\co \bigcup_{i=1}^{k}(E_i,p_i,q_i) \hookrightarrow C$ satisfying
\begin{itemize}
\item $(E_i,p_i,q_i)$ is a 2-pointed curve of arithmetic genus one,
\item $i|_{E_i-\{p_i, q_i\}}$ is an open immersion.
\item $i(q_i)=i(p_{i+1}) \in C$ is a tacnode of $C$, for $i=1, \ldots, k-1$.
\item $i(p_1)$ and $i(q_k) \in C$ are nodes of $C$.
\end{itemize}
\end{enumerate}
\end{definition}

\begin{definition}[Stability conditions]\label{D:stability-conditions}
If $C$ is any curve with $\omega_{C}$ ample, we say that
\begin{enumerate}
\item $C$ is $A_2^*$-stable if
\begin{enumerate}
\item $C$ has only nodes and cusps as singularities.
\end{enumerate}
\item $C$ is $A_2$-stable if
\begin{enumerate}
\item $C$ has only nodes and cusps as singularities,
\item $C$ has no elliptic tails.
\end{enumerate}
\item $C$ is $A_3^*$-stable if
\begin{enumerate}
\item $C$ has only nodes, cusps, and tacnodes as singularities,
\item $C$ has no elliptic tails.
\end{enumerate}
\item $C$ is $A_3$-stable if
\begin{enumerate}
\item $C$ has only nodes, cusps, and tacnodes as singularities,
\item $C$ has no elliptic tails and no elliptic bridges or chains.
\end{enumerate}
\end{enumerate}
\end{definition}
\begin{remark}\label{R:comparing-stability-conditions}
These stability conditions appear with different names in the literature: pseudostability for $A_2$\nb-stability, $c$-semistability for $A_3^*$\nb-stability, and $h$\nb-semistability for $A_3$\nb-stability. Pseudostability was introduced in \cite{Schubert}, as we discussed in Section \ref{S:modular-combinatorics}, while the latter two stability conditions were introduced in \cite{HH2}. Our change of notation is motivated by the desire to have the name of the stability condition indicate the singularities allowed.
\end{remark}
The main results of GIT so far, as indicated in Table 1 
is the following theorem, complementing the results of Gieseker and Mumford.
\begin{theorem}\label{T:GIT-construction} For any $g \geq 3$:
\begin{enumerate}
\item For $n=4$:
$$\overline{\Chow}^{\, \ss}_{g,n} =\{[C] \in \Chow(\P^N) \,|\, \emph{$C \subset \P^N$ is $A_2^*$-stable and $\O_{C}(1) \simeq \omega_{C}^n$ }\}.$$
Equivalently, $\overline{\Chow}^{\, \ss}_{g,n} \gitq \SL(N+1) \simeq \M_{g}[A_2^*]$.
\item For $n=3$: 
$$\overline{\Chow}^{\, \ss}_{g,n} =\{[C] \in \Chow(\P^N) \,|\, \emph{$C \subset \P^N$ is $A_2$-stable and $\O_{C}(1) \simeq \omega_{C}^n$ }\}.$$
Equivalently, $\overline{\Chow}^{\, \ss}_{g,n} \gitq \SL(N+1) \simeq \M_{g}[A_2]$ for $m \gg 0$.
\item For $n = 2$: 
$$\overline{\Chow}^{\, \ss}_{g,n} =\{[C] \in \Chow(\P^N) \,|\, \emph{$C \subset \P^N$ is $A_3^*$-stable and $\O_{C}(1) \simeq \omega_{C}^n$ }\}.$$
Equivalently, $\overline{\Chow}^{\, \ss}_{g,n} \gitq \SL(N+1) \simeq \M_{g}[A^*_3]$.
\item For $n = 2$ and $m \gg 0$: 
$$\overline{\Hilb}^{\, \ss,m}_{g,2} =\{[C] \in \Hilb_{P(x)}(\P^N) \,|\, \emph{$C \subset \P^N$ is $A_3$-stable and $\O_{C}(1) \simeq \omega_{C}^n$ }\}.$$
Equivalently, $\overline{\Hilb}^{\, \ss, m}_{g,n} \gitq \SL(N+1) \simeq \M_{g}[A_3]$ for $m \gg 0$.
\end{enumerate}
\end{theorem}
\begin{proof}
The case of $n=4$ is in \cite{hyeon-morrison}; the case of $n=3$ is in \cite{Schubert}; the case of $n=2$ is in \cite{HH2}.
\end{proof}

For the purposes of the log MMP for $\M_{g}$, the important point is how these spaces relate to each other. The essential points are summarized in the following proposition:
\begin{proposition}\label{P:birational-maps}
\begin{enumerate}
\item[]
\item There exists a birational contraction $\eta\co \M_{g} \rightarrow \M_{g}[A_2]$ with $\Exc(\eta)=\Delta_1$. Furthermore, the restriction of $\eta$ to $\Delta_1$ contracts the fiber of the projection
$$
\Delta_1 \simeq \M_{g-1,1} \times \M_{1,1} \rightarrow \M_{g-1,1}.
$$

\item There exists a small birational contraction $\phi^{-}\co \M_{g}[A_2] \rightarrow \M_{g}[A_3^*]$ with $\Exc(\phi^{-1})$ being the locus of elliptic bridges, defined to be the union of the images of the following maps:
\begin{align*}
\M_{1,2} \times \M_{g-2,2} &\ra  \M_{g} \rightarrow \M_{g}[A_2]  \\
 \M_{1,2}\times \M_{i,1} \times \M_{j,1} &\ra \M_{g} \rightarrow \M_{g}[A_2], \quad i+j=g-1.
\end{align*}
Furthermore, the restriction of $\phi^{-}$ to each of these strata contracts
the (images under $\eta$ of) fibers of projections \begin{align*}
\M_{1,2} \times \M_{g-2,2} &\rightarrow \M_{g-2,2}\\
M_{1,2} \times \M_{i,1} \times \M_{j,1} &\rightarrow \M_{i,1} \times \M_{j,1}, \quad i+j=g-1.
\end{align*}
\item There exists a small birational contraction $\phi^{+}\co \M_{g}[A_3] \rightarrow \M_{g}[A_3^*]$ with $\Exc(\phi^{+})$ precisely the locus of tacnodal curves.
\end{enumerate}
\end{proposition}
\begin{proof}
For the full proof of this fact, we must refer the reader to \cite{HH2}. Nevertheless, we will sketch the argument. For (1), the key observation is that there exists a natural transformation at the level of stacks $\SM_{g} \rightarrow \SM_{g}[A_2]$, replacing elliptic tails by cusps, which induces the map $\eta$. The fact that $\eta$ is an isomorphism outside $\Delta_1$ is then clear, while the fact that $\eta|_{\Delta_1}$ induces the given contraction is a consequence of the fact that the data of the $j$-invariant of the elliptic tail is not retained in the cuspidal curve which replaces it.

The analysis of $\phi^-$ and $\phi^+$ is more subtle. At the level of stacks, there are obvious open immersions
$$
\SM_{g}[A_2] \hookrightarrow \SM_{g}[A_3^*] \hookleftarrow \SM_{g}[A_3].
$$
Yet at the level of coarse moduli spaces, these maps induce birational contractions. How can this be? The key point is that although $\SM_{g}[A_3^*]$ contains more curves than $\SM_{g}[A_2]$ or $\SM_{g}[A_3]$, it is only \emph{weakly} modular and many curves which represent separate points of $\M_{g}[A_2]$ or $\M_{g}[A_3]$ are identified in $\M_{g}[A_3^*]$.

Indeed, all curves of the form $C \cup E$, where $(C,p_1,p_2)$ is a fixed smooth 2-pointed curve and $(E, q_1, q_2)$ ranges over all smooth 2-pointed elliptic curves glued along $p_i \sim q_i$, become identified in $\M_{g}[A_3^*]$ because they all specialize isotrivially to the curve $C \cup E_0$, where $E_0$ is the union of two $\PP^{1}$s meeting tacnodally (See Figure 
\ref{F:isotrivial-elliptic-bridge}). This is the essential reason that $\phi^-$ induces the stated contractions.

\begin{figure}[htb]\label{F:isotrivial-elliptic-bridge}
\includegraphics{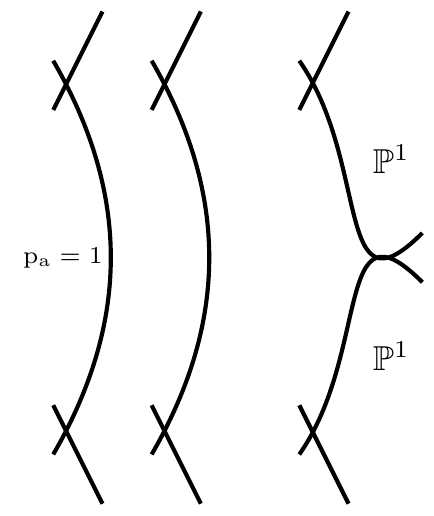}
\caption{Isotrivial specialization of an elliptic bridge to a maximally degenerate tacnodal curve}
\end{figure}

To show that $\phi^{-}$ and $\phi^+$ are isomorphisms outside the locus of elliptic bridges and tacnodes respectively, one shows that if $C$ is any $A_2$-stable curve (resp. $A_3$-stable curve) without elliptic bridges (resp. without tacnodes) then $C$ does not specialize isotrivially to any $A_3^*$-stable curve, i.e. there is no curve through $C$ that is contracted in the above manner. This requires a careful analysis of the isotrivial degenerations of $A_3^*$-stable curves, which is carried out in \cite{HH2}. Finally, the fact that $\phi^-$ and $\phi^+$ are small, i.e. isomorphisms in codimension one, follows immediately from a simple dimension count: the locus of elliptic bridges (resp. tacnodal curves) evidently have codimension two (resp. three) in moduli.
\end{proof}

\subsection{Looking forward}\label{S:Classification}

A comprehensive list of all the modular and weakly modular compactifications we have encountered thus far is displayed in Table 2.
 \begin{table}[htb]\label{T:Classification}
\begin{center}
\renewcommand{\arraystretch}{1.5}
\begin{tabular}{p{2in} p{2.5in}}
Modular Compactifications & Weakly Modular Compactifications  \\
\hline
$\M_{g,n}$,  $\M_{g,\A}$ , $\M_{g}[A_2]$ $(g \geq 3)$, & $\M_{2}[A_2]$, $\M_{g}[A_2^*]$, $\M_{g}[A_3^*]$ ($g \geq 3$), \\

$\M_{0,n}[\psi]$, $\M_{1,n}[m]$, $\M_{4}[A_3]$  & $\M_{g}[A_3]$ $(g=3, g \geq 5)$ 

\end{tabular}
\caption{Modular and Weakly Modular Compactifications}
\end{center}
\end{table} 
Staring at this table, the curious reader is bound to wonder: How comprehensive it? Is it possible to obtain a complete classification of modular and weakly modular birational models of $\M_{g,n}$ for each $g$ and $n$? On the one hand, there is certainly no reason to expect this list to be comprehensive. There are many natural classes of singularities (e.g. deformation open classes of ADE singularities, deformation open classes of toric singularities) generalizing the class $\{\text{nodes, cusps, tacnodes}\}$ and we see no obvious obstruction to constructing stability conditions for every such class. On the other hand, judging by the phenomena encountered in the tacnodal case, it seems clear that, in formulating the appropriate stability conditions for these classes of singularities, we will almost certainly have to settle for weakly modular compactifications.

Since the only general method for constructing weakly modular compactifications is GIT, the most straightforward approach for producing new stability conditions is to push the GIT machinery further. 
On the other hand, there are  compactifications of $M_{g,n}$ for which no GIT construction is known ($\M_{1,n}[m]$ being one example).  
For this reason, it would be interesting to have a more systematic method of constructing weakly modular birational models of $\M_{g,n}$, along the lines of what we have sketched in Section \ref{S:modular-combinatorics}. Indeed, using the ideas of Section \ref{S:modular-combinatorics}, the second author has come very close to giving a complete classification of modular compactifications of 
$\M_{g,n}$ \cite{Z-stability}. Given systematic methods for constructing mildly non-separated functors combinatorially, we might approach such a classification for weakly modular compactifications. The natural starting point for such a program would be the be the construction of pointed versions of the stability conditions discussed above, i.e. the construction of moduli spaces $\M_{g,n}[A_2^*]$, $\M_{g,n}[A_2]$, $\M_{g,n}[A_3^*]$, $\M_{g,n}[A_3]$.

Finally, we should discuss the projectivity of these alternate birational models. All the models in the displayed table are projective, though the methods of proof vary. The models $\M_{g}[A_2^*], \M_{g}[A_2], \M_{g}[A_3^*]$ and $\M_{g}[A_3]$ are constructed via GIT and therefore automatically come with a polarization. For the models 
$\M_{g,\A}$, Koll\'{a}r's semipositivity provides a direct proof as we will
see in Section \ref{S:kollar}. On the other hand, for the models $\M_{0,n}[\psi]$ and $\M_{1,n}[m]$, the only proofs of projectivity rely on direct intersection theory techniques. In light of these differences, it seems natural to wonder whether Koll\'{a}r's techniques be generalized to curves with worse-than-nodal singularities, e.g. could they be used to give a direct proof of projectivity of $\M_{g}[A_2]$ or $\M_{1,n}[m]$? Even more ambitiously, we can ask whether they could be applied in the case of weakly modular compactifications, e.g. to give a direct proof of the projectivity of $\M_{g}[A_3]$. Such techniques would be interesting inasmuch as they might be applicable to moduli functors where no GIT construction exists.

\section{Birational geometry of moduli spaces of curves}
\label{S:birational-geometry-moduli}
\subsection{Birational geometry in a nutshell}
\label{S:birational-geometry}
The purpose of this section is to recall those definitions from higher-dimensional geometry which will be used repeatedly in the sequel. Let $X$ be a normal, projective, $\Q$-factorial variety. We would like to classify the set of all projective birational models of $X$. Of course, no such classification is feasible if we allow birational models dominating $X$, so it is reasonable to consider models which are, in some sense, smaller than $X$. This gives rise to the following definition.
\begin{definition}\label{D:contraction}
A \emph{birational contraction} is a birational map $\phi \co X \dashrightarrow Y$ 
such that $Y$ is a normal projective variety and $\Exc(\phi^{-1})\subset Y$ has codimension at least
$2$. A birational contraction is \emph{$\QQ$\nb-factorial} if $Y$ is. 
\end{definition}

Two birational contractions $\phi_1\co X \dashrightarrow Y_1$ and $\phi_2\co X \dashrightarrow Y_2$ are equivalent if there is an isomorphism $Y_1 \simeq Y_2$ making the obvious diagram commute. Remarkably, the set of all birational contractions up to equivalence, while not necessarily finite, can in some sense be classified purely in terms of the numerical divisor theory of $X$. To explain this, we recall a few standard definitions from higher dimensional geometry.

If $C_1, C_2 $ are proper curves in the Chow group $A_1(X)$, we say that $C_1$ and $C_2$ are numerically equivalent (and write $C_1 \equiv C_2$) if $\L.C_1 =\L.C_2$ for all line bundles $\L \in \Pic(X)$. Dually, we say that two line bundles $\L_1, \L_2 \in \Pic(X)$ are numerically equivalent (and write $\L_1 \equiv \L_2$) if $\L_1.C=\L_2.C$ for all curves 
$C \subset X$. We set
\begin{align*}
\N_1(X):=&A_1(X) \otimes \Q / \equiv \\
\N^1(X):=&\Pic(X) \otimes \Q /\equiv
\end{align*}
It follows immediately from the definition, that there is a perfect pairing
$$
\N^{1}(X) \times \N_{1}(X) \rightarrow \Q,
$$
induced by the intersection pairing $(\L, C) \rightarrow \L.C$, so $\N^{1}(X)$ and $\N_{1}(X)$ are dual. The theorem of the base asserts that $\N^{1}(X)$ and $\N_{1}(X)$ are finite-dimensional vector spaces \cite[Proposition 1.1.16]{Laz1}. 

For understanding the birational geometry of $X$, two closed convex cones in $\N^1(X)$ play a central role, namely the \emph{nef cone} $\Nef(X)$ and the \emph{pseudoeffective cone} $\overline{\Eff}(X)$. The nef cone is simply the closed convex cone in $\N^1(X)$ generated by all nef divisor classes\footnote{Recall that a $\Q$-divisor $D$ is \emph{nef} if $D.C \geq 0$ for all curves $C \subset X$.}, while the pseudoeffective cone is the closed convex cone in $\N^1(X)$ generated by all effective divisor classes. Note that the cone generated by all effective divisor classes (typically denoted $\Eff(X)$) is, in general, neither open nor closed so the pseudoeffective cone is the closure of $\Eff(X)$. A standard result characterizes the interior of the nef cone as the open cone $\Amp(X)$ generated by classes of ample divisors, and the interior of the pseudoeffective cone as the open cone $\BIG(X)$ generated by classes of big divisors \footnote{Recall that a $\Q$-divisor $D$ is \emph{big} if $\dim \HH^0(X, mD)=cm^{n}+O(m^{n-1})$ for some constant $c$ and all $m\gg 0$ sufficiently divisible.}, i.e. we have $\overline{\Amp(X)}=\Nef(X)$ and $\overline{\BIG}(X)=\overline{\Eff}(X)$ \cite[Theorems 1.4.23 and 2.2.26]{Laz1}. The reader who is not familiar with these definitions can find several examples of these cones in \cite[Sections 1.5 and 2.3]{Laz1}.

Now if $\phi \co X \dashrightarrow Y$ is any $\QQ$\nb-factorial birational contraction, there exist open sets $U \subset X$, $V \subset Y$ satisfying 
\begin{itemize}
\item $\phi$ restricts to a regular morphism $\phi|_{U} \co U \rightarrow V$
\item $\codim (X-U) \geq 2$ and $\codim(Y-V) \geq 2$.
\end{itemize}
Then there are natural homomorphisms $\phi_*\co \N^{1}(X) \rightarrow \N^{1}(Y)$ induced by restricting a codimension one cycle to $U$, pushing forward to $V$ and then taking the closure in $Y$. Similarly, there is a pull-back 
$\phi^* \co \N^{1}(Y) \rightarrow \N^{1}(X)$ induced by restricting a $\QQ$\nb-line-bundle on $Y$ to $V$, pulling back along 
$\phi\vert_{U}$ and then extending to a $\QQ$\nb-line-bundle on $X$.  
It is an exercise to check that these operations respect numerical equivalence \cite[1.2.12 and 1.2.18]{rulla}.

\begin{definition}
Let $\phi\co X \dashrightarrow Y$ be a birational contraction with exceptional divisors $E_1, \ldots, E_m$.
The \emph{Mori chamber of $\phi$} is defined as:
$$
\Mor(\phi):=\{\phi^*D+\sum_{i=1}^{m}a_iE_i \ |\ a_i \geq 0, \, D \in \Amp(Y) \} \subset \N^1(X).
$$
\end{definition}
\begin{lemma}\label{L:Mori-chamber}
\par
\noindent
\begin{enumerate}
\item $\overline{\Mor(\phi)}$ is a convex set of dimension $\dim \N^1(Y)$.
\item $\phi_1 \sim \phi_2$ iff $\Mor(\phi_1) \cap \Mor(\phi_2) \neq \varnothing$ iff $\Mor(\phi_1) = \Mor(\phi_2)$.
\item $D \in \Mor(\phi)$ iff $R(X,D)$ is finitely generated and $Y=\Proj R(X,D)$.
\end{enumerate}
\end{lemma}
\begin{proof}
This is standard, see e.g. \cite{hu-keel}.
\end{proof}
This lemma suggests an approach to the classification of birational contractions
of $X$. If one can find a set of birational contractions 
$\phi_j \co X \dashrightarrow Y_j$ such that the associated Mori chambers partition 
$\PEff(X)$:
$$
\PEff(X) = \overline{\Mor(\phi_1)} \cup \ldots \cup \overline{\Mor(\phi_k)},
$$
then Lemma \ref{L:Mori-chamber} guarantees that these must be all the birational contractions of $X$ and this decomposition of the pseudoeffective cone $\overline{\Eff}(X)$ is called the 
{\em Mori chamber decomposition}. If a Mori chamber decomposition exists for the entire pseudoeffective cone $\overline{\Eff}(X)$, we call $X$ a \emph{Mori dream space} \cite{hu-keel}. Such a chamber decomposition may fail to exist, however, for two reasons. First, $X$ may have infinitely many birational contractions, in which case the chamber decomposition is not finite. Second, a divisor $D$ may fail to have a finitely generated section ring, in which case it belongs to no Mori chamber \cite[Example 2.1.30]{Laz1}. The following powerful result (which will be used repeatedly in the sequel) gives a sufficient condition for basepoint freeness. 
\begin{theorem}[\text{Kawamata basepoint freeness theorem}] \label{T:kawamata}
Let $(X, \Delta)$ be a proper klt pair with $\Delta$ effective. Let $D$ be a nef Cartier divisor such that $aD-K_X-\Delta$ is nef and big for 
some rational $a>0$. Then $|mD|$ is basepoint free for $m\gg 0$.
\end{theorem}
We refer to \cite[Section 3]{kollar-mori} for the definition of klt pairs, the proof of Theorem \ref{T:kawamata}, and the history of this result.

While it is too much to hope, in general, for a chamber decomposition of the entire effective cone of a variety, the theory of higher dimensional geometry, culminating in the recent work of Siu \cite{siu} and Birkar, Cascini, Hacon, and McKernan \cite{BCHM} shows that there \emph{is} a certain region within the effective cone 
where it is possible to give a Mori chamber decomposition. For our purposes, the most relevant result is:

\begin{theorem}
Let $\Delta=\sum_{i\in S}\Delta_i \subset \M_{g,n}$ be the boundary, where $S$ is the indexing set for irreducible boundary components. For a rational positive number $\alpha\in (0,1]$, let $\Delta(\alpha) \subset \N^1(\M_{g,n})$ be the polytope
$$\{ K_{\M_{g,n}}+\sum_{i\in S} \alpha_i\Delta_i \, | \, \alpha_i \in [\alpha,1] \}$$
Then $\Delta(\alpha)$ admits a finite Mori chamber decomposition. More precisely, there exists a finite set of birational contractions $\phi_j\co X \dashrightarrow Y_j$ such that 
\begin{enumerate}
\item $\overline{\Mor(\phi_i)} \cap [0,1]^{\vert S\vert}$ is a piecewise linear polytope.
\item $[\alpha,1]^{\vert S\vert} = \overline{\Mor(\phi_1)} \cup \ldots \cup \overline{\Mor(\phi_k)}$. 
\end{enumerate}
\end{theorem}
\begin{proof}
See \cite[Corollary 1.2.1]{BCHM}.
\end{proof}

This theorem raises the question: can we compute this chamber decomposition explicitly? Of course, the Hassett-Keel log minimal model program is just the special case, in which one restricts attention to the ray $\alpha_0=\alpha_1=\ldots=\alpha_{\lfloor g/2\rfloor}$. We will address what is known about this problem in Section \ref{S:log-MMP}.

\subsection{Effective and nef cone of $\M_{g,n}$}\label{S:nef-cone}

\begin{table}[htbp]\label{T:Cones}
\begin{center}
\renewcommand{\arraystretch}{1}
\begin{tabular}{|c|c|c|c|}
\hline
& $\M_{g}$ & $\M_{0,n}$  & $\M_{g,n}$ \\
\hline
\multirow{3}{1.2in}{Rays of restricted effective cone}

& \multirow{3}{1.2in}{$\delta, s\lambda-\delta$, $\frac{60}{g+4}<s\leq 6+\frac{12}{g+1}$} &
 \multirow{3}{1.2in}{$-\frac{n-1}{\lfloor \frac{n}{2}\rfloor (n-\lfloor \frac{n}{2}\rfloor)}\psi+\delta$, $\frac{n-1}{2(n-2)}\psi-\delta$}  &
unknown
\\
&&& \\
&&& \\
\hline
\multirow{2}{1.2in}{Full efffective cone}
& unknown
& \multirow{2}{1.2in}{Castravet-Tevelev Conjecture}
& unknown\\
& & & \\
\hline
\multirow{3}{1.2in}{Rays of restricted nef cone}
& \multirow{2}{1.2in}{$\lambda, 11\lambda-\delta$}
& \multirow{3}{1.2in}{$\psi$ ($n\geq 8$) and $\frac{2}{3}\psi-\delta$ ($n\geq 5$)}
&   \multirow{2}{1.2in}{$\lambda, \psi, 11\lambda+\frac{2}{3}\psi-\delta$ ($n\geq 3)$}
\\
& & & \\
&&&\\
\hline
\multirow{1}{1.2in}{Full nef cone}  & F-conjecture             & F-conjecture                  & F-conjecture     \\
\hline
\multirow{2}{1.2in}{Mori Dream Space?} & \multirow{2}{1.12in}{`Yes' for $g=2$, unknown for $g\geq 3$} & \multirow{2}{1.2in}{`Yes' for $n \leq 6$, unknown for $n \geq 7$}   &     \multirow{2}{1.4 in}{`No' for $g\geq 3$, $n=1$ and for $g\geq 2$, $n\geq 2$.  } \\
& & & \\
& & & \\
& & & \\
\hline

\end{tabular}

\end{center}

\caption{Effective and nef cones of $\M_{g,n}$}
\end{table}

In this section, we describe most of what is known about the nef cone and the effective cone of $\M_{g,n}$. Since much of this material is well-covered in 
the existing surveys (e.g. \cite{farkas-seattle}, \cite{morrison-mori}), we will focus on statements of results, emphasizing mainly those details that are necessary for the results of Section \ref{S:log-MMP} on Mori chamber decompositions of $\Eff(\M_{g,n})$. To begin with, we should emphasize that neither the nef cone nor the effective cone of $\M_{g,n}$ is known in full generality. When seeking partial results, it is common to focus on the \emph{restricted effective cone} $\Eff_\rst(\M_{g,n})$ (respectively, the \emph{restricted nef cone} $\Nef_\rst(\M_{g,n})$), which we define to be the intersection of the effective cone (respectively, the nef cone) with the subspace $\Q\{\lambda, \psi, \delta \} \subset \N^1(\M_{g,n})$. 
Note that in the special cases $n=0$ (resp. $g=0$), we have $\psi=0$ (resp. $\lambda=0$), so these restricted cones are simply two-dimensional convex cones in a two-dimensional vector space. Our present state of knowledge concerning all these cones is displayed in Table 3. 

Before elaborating on the individual entries in Table 3, 
let us briefly recount the historical progression of ideas which preceded and motivated this work. Long before the concepts and vocabulary of Mori theory became the standard approach to birational geometry, the problem of finding rational or unirational parametrizations for $M_{g}$ vexed (and inspired) many mathematicians. Perhaps the first recognizable foray into the birational geometry of $M_{g}$ was made in the context of invariant theory. For example, the fact that every genus two curve admits a unique double cover of $\P^1$ branched at 6 points implies that $M_{2}$ is naturally birational to the orbit space of binary sextic forms under $\SL(2)$. The generators and relations for the corresponding ring of invariants were computed by Clebsch in 1872 \cite{clebsch-book}, who thus gave explicit coordinates on $M_{2}$ (analagous to the Weierstrass $j$-function for $M_{1}$)  and proved the rationality of $M_{2}$ as a by-product (see also \cite{Igusa}).  Similar results for $M_{g}$, $g \leq 6$, have been obtained only recently, with Shepherd-Barron proving rationality of $M_{4}$ and $M_{6}$ \cite{barron-4, barron-6} and Katsylo proving rationality of $M_{3}$ and $M_{5}$ \cite{katsylo-3, katsylo-5}. 

While approaches to the rationality of $M_{g}$ remain essentially algebraic, Severi realized that the weaker problem of proving unirationality (i.e. of exhibiting merely a dominant rational, rather than birational, map $\P^N \dashrightarrow M_{g}$) was accessible by geometric methods \cite{severi-plane}. Severi used the variety $V_{d,g}$ parameterizing plane curves of degree $d$ and geometric genus $g$ (now called \emph{Severi varieties}) to prove the unirationality of $M_{g}$ for $g \leq 10$. Severi's idea was to represent a general curve of genus $g$  as a plane curve of degree $d(g)=\lceil \frac{2g+6}{3}\rceil$ with
exactly $\delta(g)=\binom{d-1}{2}-g$ nodes. For genus $g \leq 10$, a simple dimension count shows the existence of plane curves of degree $d(g)$ with $\delta(g)$ nodes in general position in $\P^2$. Thus, $V_{d,g}$ is birational to a projective bundle over $(\P^2)^{\delta(g)}$, hence rational. The natural dominant rational map $V_{d,g} \dashrightarrow M_{g}$ then proves the unirationality of $M_{g}$. As explained in \cite{arbarello-sernesi}, Severi's argument made implicit use of several unproven assumptions about $V_{d,g}$ -- Arbarello and Sernesi proceed to fill the gaps in Severi's proof thus giving 
a rigorous proof of unirationality of $\M_g$ for $g\leq 10$. 

The question of whether $M_g$ remains unirational for all $g \geq 10$ was open until the pioneering work of Mumford and Harris \cite{HMKodaira}, who 
showed that $\M_g$ is of general type for odd $g \geq 25$ and has Kodaira dimension at least zero for $g=23$. This work was later generalized by Eisenbud and Harris to show that $\M_{g}$ is of general type for \emph{all} $g \geq 24$ \cite{EHKodaira}. In particular, $M_{g}$ is not unirational for any $g \geq 23$. From a historical perspective, this work is a marvelous reflection (and application) of the emerging understanding of the importance of Kodaira dimension in the study of birational classification problems. Note that since Kodaira dimension is meaningful only for projective varieties, this perspective could not be brought to bear on the unirationality problem without the existence of a suitable projective compactification $M_{g} \subset \M_{g}$. The moduli space of stable curves, constructed less than 15 years previously by Deligne and Mumford \cite{DM}, provided the requisite compactification.

Now, let us briefly outline the argument of Harris and Mumford. The two key ingredients are:
\begin{enumerate} 
\item $\M_g$ has canonical singularities for all $g\geq 2$.
\item For odd $g \geq 25$, the canonical divisor $K_{\M_g}$ is big.
\end{enumerate}
(2) says that the canonical divisor has lots of sections, and (1) says that these sections extend to an arbitrary desingularization of $\M_g$. Together, they imply that $\M_{g}$ is of general type for odd $g \geq 25$. 

The proof of (1) is a straight-forward (albeit arduous) application of the Reid-Tai criterion (cf. \cite{reid3folds, tai-kodaira}), which gives an explicit method for determining whether a finite quotient singularity is canonical. To explain the proof of (2), it is useful to make a preliminary definition: If $D \subset \M_{g}$ is any Weil divisor, we define the \emph{slope of $D$} to be 
\begin{align}\label{E:slope}
s(D):=\inf \left\{ \frac{a}{b} \ : \ a\lambda-b\delta-D\geq 0\right\}.
\end{align}
For any irreducible non-boundary effective divisor $D\subset \M_g$, we have 
$$D \equiv a\lambda-\sum_{i=0}^{\lfloor g/2 \rfloor}b_i\delta_i \in \N^1(\M_{g}),$$
where $a$ and $b_i$ are positive rational numbers. Consequently, 
$$s(D)=\frac{a}{\min_{ 0\leq i \leq \lfloor g/2 \rfloor} b_i}.$$

The proof of (2) is connected with the problem of understanding slopes of effective divisors by the following two observations: First, $\lambda$ is a big divisor (indeed, it is semiample and defines the Torelli morphism $\tau\co \M_{g} \rightarrow \overline{\A}_{g}$ to the Satake compactification of principally polarized abelian varieties \cite{namikawa}). Second, since $K_{\SM_{g}}=13\lambda-2\delta$ and $\SM_{g} \rightarrow \M_{g}$ is ramified along $\Delta_1$ (see \cite[Section 2]{HMKodaira}), the canonical class of the coarse moduli space is given by the formula
$$ 
K_{\M_{g}}=13\lambda-2\delta_0-3\delta_1-2\sum_{i=2}^{\lfloor g/2 \rfloor}\delta_i.
$$
It follows that if $D \subset \M_{g}$ is any effective divisor of slope $s(D)<13/2$, we have a numerical proportionality
$$
K_{\M_{g}}\equiv D+a\lambda+\sum_{i}b_i\delta_i,
$$
where $a$ and $b_i$ are positive rational constants. In particular, since $\lambda$ is big and $D+\sum_{i}b_i\delta_i$ is effective, $K_{\M_{g}}$ is big. In other words, to show that $K_{\M_{g}}$ is big, it is sufficient to find an effective divisor of slope less than 13/2.

The effective divisor $D \subset \M_{g}$ originally considered by Mumford and Harris was the closure of the locus of $d$-gonal curves where $d=\frac{g+1}{2}$. Subsequently, Eisenbud and Harris considered the more general {\em Brill-Noether divisors} $D_{r,d} \subset \M_{g}$, defined as the closure of the locus of curves possessing linear series 
$g^{r}_{d}$ satisfying $g-(r+1)(g+r-d)=-1$ (as well as the so-called Gieseker-Petri divisors whose definition we omit). In each case, the key technical difficulty is the computation of the numerical class of these divisors, i.e. the coefficients $a$ and $\{b_i\}_{i=0}^{\lfloor g/2 \rfloor}$ such that $D_{r,d} \equiv a\lambda+\sum_{i=0}^{\lfloor g/2 \rfloor}b_i \delta_i$. One can extract these coefficients by the method of test curves, provided one knows which singular curves in $\M_{g}$ actually lie in $D_{d}$ or $D_{r,d}$. To answer this question, Harris and Mumford developed the theory of admissible covers \cite{HMKodaira}, and Eisenbud and Harris developed the theory of limit linear series \cite{EHLimit}. The final result of this computation is a numerical proportionality
$$
D_{r,d} \sim \left(6+\frac{12}{g+1}\right)\lambda-\delta_0 - \sum_{i=1}^{\lfloor g/2\rfloor} a_i\delta_i, \quad a_i \geq 0.
$$ 
In particular, for $g \geq 24$, $s(D_{r,d})<13/2$ so $\M_{g}$ is of general type as desired.

The discovery that $\M_g$ is of general type for $g\geq 24$ spurred an interest to the intermediate cases of $11\leq g\leq 23$. The $g=12$ had actually been settled by Sernesi prior to the work of Harris and Mumford \cite{sernesi12}, but Chang and Ran used a more general method to establish the unirationality of $\M_g$ for $g=11, 12, 13$ \cite{ChangRan13}. In addition, they showed that the slopes of $\M_{15}$ and 
$\M_{16}$ are greater than $6\frac{1}{2}$, thus
proving that $\M_{15}$ and 
$\M_{16}$ have negative Kodaira dimension \cite{ChangRan15, ChangRan16}. It follows that these spaces are uniruled \cite{bdpp}, and Bruno and Verra proved the stronger result that $\M_{15}$ is rationally connected \cite{bruno-verra}. 
The unirationalty question has been recently revisited by Verra who
proved that $\M_{14}$ is unirational \cite{verra14}. 
As we discuss below, Farkas has shown that $\M_{22}$ is of general type \cite{farkas-aspects}. As far 
as we know, the Kodaira dimension of $\M_g$ is completely unknown for $17\leq g\leq 21$. 

With this historical background in place, let us return to Table 3 
and explain each of the rows in greater detail.
\subsubsection{Effective cones}
The work of Eisenbud, Harris, and Mumford gave rise to an intensive study of the restricted effective of $\M_{g}$. Since the boundary $\delta$ obviously lies on the edge of the effective cone for $g \geq 3$, the key question is: what is the {\em slope} 
$$s_g:=\inf\{s(D) \ : \ D\subset \M_g \text{  irreducible, not contained in the boundary}\}
$$ i.e. the slope of lower boundary of the effective cone $\Eff(\M_{g})$. For a long 
time, the Brill-Noether divisors $D_{r,d}$ were the divisors of the smallest known slope. This led Harris and Morrison 
to formulate the \emph{Slope Conjecture}, namely that \emph{every effective divisor has slope at least $6+\frac{12}{g+1}$} \cite{harris-morrison}. 
The obvious way to test the conjecture is by computing the slopes of various effective divisors in $\M_{g}$, and ultimately Farkas and Popa produced a counterexample by doing just that \cite{farkas-popa}. Their counterexample is given by the divisor $K_{10}\subset \M_{10}$ of curves isomorphic to a hyperplane section of a K3 surface, first introduced by Cukierman and Ulmer \cite{CU}. 
Farkas and Popa computed the class of $K_{10}$ and showed that $s(K_{10})=7<6+\frac{12}{11}$. Subsequently, more counterexamples to the Slope Conjecture were produced by Farkas in 
\cite{farkas-syzygy, farkas-koszul} and Khosla \cite{khosla, khosla2}. 
In particular, in \cite{farkas-koszul} it is shown that all known effective divisors of small slope (Brill-Noether, Gieseker-Petri, $K_{10}$,
Khosla's divisors) can be obtained from a single syzygy-theoretic construction. These so-called {\em Koszul divisors} provide an infinite sequence of 
counterexamples to the Slope Conjecture. Furthermore, when $g=22$, there is a Koszul divisor of slope $17121/2636<6\frac{1}{2}$ making $\M_{22}$ a variety of general type \cite[Theorem 7.1]{farkas-aspects}. Curiously, all Koszul divisors have slopes no smaller than $6+\frac{10}{g}$ 
which leaves open the possibility of a positive genus-independent lower bound for the slopes $s_g$. As far as we know, no one has constructed an effective divisor of slope less than $6$. 

In a different direction, there have been several attempts to produce lower bounds for the slope via the method of moving curves. Here, the key idea is that if $C$ is a numerical curve class whose deformations cover $\M_{g}$ (such curve classes are called \emph{moving}), then evidently $D.C \geq 0$ for any effective divisor $D=a\lambda-\sum_{i=0}^{\lfloor g/2\rfloor} b_i\delta_i$. In particular,
$$
s(D)=\inf_{i=0}^{\lfloor g/2\rfloor} a/b_i \geq \delta.C/\lambda.C.
$$
Thus, if we define the \emph{slope} of a curve $C$ to be the ratio $\delta.C/\lambda.C$, then the slope of any moving curve class gives a lower bound for the slope $s_{g}$. Harris and Morrison exploited this idea by using admissible covers to compute the slopes of one-parameter families of $k$-gonal curves \cite{harris-morrison}. When $k\geq (g+1)/2$, such families 
define moving curves in $\M_g$ and using this the authors establish the Slope Conjecture for $g\leq 5$, and give an asymptotic lower bound $s_g\geq \frac{576}{5g}$ for $g\gg 0$. The idea of using families of admissible covers to obtain lower bounds on $s_g$ was recently revisited by Dawei Chen who writes down infinite Zariski dense families of branched covers of an elliptic curve in \cite{dawei-thesis2, dawei-elliptic} and branched 
covers of $\PP^1$ in \cite{dawei-quadratic}. 
Somewhat surprisingly, the resulting asymptotic lower bound on $s_g$ in 
\cite{dawei-elliptic}
turns out to be $\frac{576}{5g}$ as well, which led Chen to conjecture that $s_g\sim \frac{576}{5g}$ for $g\to \infty$.

There are many other approaches to producing moving curves on $\M_g$. For example, taking linear sections of the Severi variety $V_{d,g}$ gives moving curves on $\M_{g}$ when $d \geq \lfloor \frac{2g+8}{3} \rfloor$. The corresponding curve classes were computed in the first author's thesis \cite{fedorchuk-thesis}, and yield an asymptotic bound of the form $o(1/g)$. In a different vein, if $\pi\co \M_{g,1} \rightarrow \M_{g}$ is the universal curve, the fact that  $\psi_1$ is the limit of ample divisors implies that  $\pi_*(\psi_1^{3g-4})$ is the limit of moving curve classes on $\M_{g}$. Pandharipande \cite{pand-slopes} computed the slope of this 
curve class $\pi_*(\psi_1^{3g-4})$ to be $\frac{60}{g+4}$ using Hodge integrals, thus giving a lower bound on the slope for all $g$. Alas, there is a large gap between $6+\frac{10}{g}$ and $\frac{60}{g+4}$, especially when $g$ is large, and 
 the precise boundary of the (restricted) effective cone remains unknown.

The study of the effective cone of $\M_{0,n}$ has a much different flavor; since $\lambda=0$ on $\M_{0,n}$, the question simply becomes: what linear combinations of boundary divisors are effective? Since every boundary divisor of $\M_{0,n}$ clearly lies on the boundary of the effective cone, it is natural to 
hope that the effective cone is simply spanned by the boundary divisors. However, Keel and Vermeire independently produced divisors in $\M_{0,6}$ which are not effective sums of boundary divisors \cite{vermeire}. Hassett and Tschinkel proved that Keel-Vermeire divisors together with the 
boundary divisors generate $\Eff(\M_{0,6})$ \cite{brendan-yuri}. Recently, Castravet and Tevelev \cite{castravet-tevelev} have constructed a plethora of new 
examples of effective divisors which are not effective sums of boundary divisors -- their so-called \emph{hypertree divisors}. They have conjectured that the hypertree divisors, together with the boundary divisors, should span the entire effective cone. 

As for the restricted effective cone of $\M_{0,n}$, it is an easy consequence of a theorem of Keel and McKernan, who showed that every effective divisor on the symmetrized moduli space $\M_{0,n}/S_{n}$ \emph{is} numerically equivalent to an effective sum of $S_{n}$\nb-invariant boundary divisors \cite[Theorem 1.3]{KMc}. Since 
$$
\psi=\sum_{i=1}^{\lfloor n/2\rfloor} \frac{i(n-i)}{n-1} \Delta_i
$$
is an $S_{n}$-invariant sum of boundary divisors, it follows that any divisor of the form $a\psi+b\delta$ is effective iff the coefficients of $\Delta_i$'s in the sum 
$$
\sum_{i=1}^{\lfloor n/2\rfloor} \left(a+b\frac{i(n-i)}{n-1}\right) \Delta_i
$$
 are positive.  From this, one sees that the restricted effective cone of $\M_{0,n}$ is spanned by the two rays indicated in Table 3. 
 
Finally, while Logan \cite{logan-kodaira} and Farkas \cite{farkas-koszul} have extended the work of Eisenbud, Harris, and Mumford to $\M_{g,n}$ (see \cite[Theorem 1.10]{farkas-koszul} for a comprehensive list of $(g,n)$ where $\M_{g,n}$ is known to be of general type), no analogues of the slope conjecture have been offered or investigated for $\M_{g,n}$. Needless to say, there is no reason to expect the restricted effective cone of $\M_{g,n}$ to be easier to study than that of $\M_{g}$.

\subsubsection{Nef cones} Next, we consider the nef cone of $\M_{g}$. The Hodge class $\lambda$ is well-known to be nef on $\M_{g}$; indeed it is the pullback of the canonical polarization on $\Abar_{g}$ via the
extended Torelli morphism $\tau\co \M_{g} \rightarrow \Abar_{g}$. The nefness of $11\lambda-\delta$ is the result of  Cornalba and Harris \cite{CH}, who exploit the fact that the canonical embedding of a non-hyperelliptic smooth curve of genus $g$ is asymptotically Hilbert stable to deduce 
that 
the divisor $\left(8+\frac{4}{g}\right)\lambda-\delta$ is positive on families with smooth non-hyperelliptic generic fibers. 
Combining this with a separate analysis of the hyperelliptic and 
nodal case gives the following.
\begin{theorem}[\text{\cite[Theorem 1.3]{CH}}] \label{T:CH}
The divisor $s\lambda-\delta$ is ample for $s>11$. Moreover, $11\lambda-\delta$ has degree $0$ precisely on curves whose 
only moving components are elliptic tails.
\end{theorem}
In fact, Cornalba and Harris proved a stronger result regarding generically smooth families of stable curves. We do not 
state it here since it has 
been since sharpened by Moriwaki \cite{moriwaki1} as follows.
\begin{proposition}\label{P:moriwaki}
The divisor
\begin{align*}
\left(8g+4\right)\lambda-g\delta_{0}-\sum_{i=1}^{\lfloor g/2\rfloor} 4i(g-i)\delta_i
\end{align*}
has non-negative degree on every complete one-parameter family of generically smooth curves of genus $g\geq 2$. 
 \end{proposition}

The full nef cone of $\M_{g}$ is not known, but there is a fascinating conjectural description -- the so-called F-conjecture -- which would in principle determine $\Nef(\M_{g,n})$ for all $g$ and $n$. Furthermore, the F-conjecture, which is described in greater detail in Section \ref{S:F-conjecture} below, would imply that the nef cone of $\M_{g,n}$ is finite polyhedral for all $g$ and $n$, which is surprising from a purely Mori-theoretic point of view. Nevertheless, two recent results comprise substantial evidence for the F-conjecture: First, Gibney, Keel, and Morrison have shown that if the F-conjecture holds for $\M_{0,g+n}/S_{g}$, then it holds for $\M_{g,n}$ \cite{GKM}. In particular, it is sufficient to prove the conjecture for $\M_{0,n}$ for all $n$, where it appears somewhat more plausible. Second, Gibney has developed efficient computational methods for checking the conjecture, and has proved that it holds for $\M_{g}$ up to $g=24$ \cite{Gib}.

\subsubsection{Mori dream spaces}

Finally, we briefly address the question of whether $\M_{g,n}$ is a Mori dream space. Keel \cite[Corollary 3.1]{keel-annals}
showed that the divisor class $\psi$ on $\M_{g,1}$ is big and nef but not semiample (in characteristic zero) for $g\geq 3$. The same argument shows that $\psi$ is big and 
nef but not semiample for $(g,n)\geq (2,2)$.
Hence $\M_{g,n}$ cannot be a Mori dream space if $g \geq 3$ and $n=1$ or $g\geq 2$ and $n\geq 2$. This leaves the cases of $\M_{g}$, $\M_{0,n}$, $\M_{1,n}$, $\M_{2,2}$. It is easy to see by a direct computation that $\M_{2}$ and $\M_{0,5}$ are Mori dream spaces (see Sections \ref{S:MMPM2} and \ref{S:MMPM0n}). A direct proof that $\M_{0,6}$ is a Mori dream space was given by Castravet in \cite{castravetM06}. 
That $\M_{0,n}$ is a Mori dream space for $n\leq 6$ also follows from the fact that it is log Fano and
from the results of \cite{BCHM}. The remaining cases are, so far as we know, open.

Obviously, a full treatment of the results mentioned in the preceding paragraphs would require a survey in itself, so we confine ourselves here to just a few highlights: In Section \ref{S:F-conjecture}, we shall describe the F-conjecture in greater detail and explain the partial results which are used in the log MMP for $\M_{g}$. In Section \ref{S:kollar}, we shall sketch a simplified version of the Koll\'{a}r's method of producing nef divisors on $\M_{g, \A}$. These methods do not reproduce the sharp results of Cornalba-Harris (Theorem \ref{T:CH}) and Moriwaki (Proposition \ref{P:moriwaki}), but they do give a feel for the basic ideas involved. Furthermore, the resulting nef divisors allow us to interpret $\M_{g, \A}$ as log canonical models of $\M_{g,n}$ and they also play a key role in the log MMP for  $\M_{0,n}$.

\subsubsection{F-conjecture}\label{S:F-conjecture}
Recall that $\M_{g,n}$ admits a natural stratification by topological type, in which the locus of curves with $m$ nodes form the strata of codimension $m$. Thus, the closure of the locus of curves with $3g-4+n$ nodes comprises a collection of irreducible curves on $\M_{g,n}$, the so-called \emph{F-curves}. The F-conjecture simply asserts

\begin{conjecture}[F-conjecture]\label{Fconjecture}
A divisor $D$ on $\M_{g,n}$ is nef if and only if $D$ is F-nef, i.e. $D$ intersects every F-curve non-negatively.
\end{conjecture}

There are finitely many F-curves, and it is straightforward to compute their intersection numbers with the divisor classes of $\M_{g,n}$. Let us give some explicit examples of F-curves that will be relevant in Section \ref{S:log-MMP}. 

\begin{example}[Elliptic tails and elliptic bridges]\label{E:elliptic-bridges}
We define the F-curve of elliptic tails, denoted $\mathrm{T}_1$, by taking a pencil of plane cubics with a marked base point section --
a degree $12$ cover of the universal family over $\Mg{1,1}$ -- 
 and attaching a fixed maximally degenerate genus $g-1$ curve (i.e. a genus $g-1$ curve with $3g-7$ nodes) along the section. One easily computes the intersection numbers: 
\begin{equation}\label{E:elliptic-tail}
\begin{aligned}
\lambda\cdot \mathrm{T}_1 &=1, & \delta_{0}\cdot \mathrm{T}_1 &=12,  \\
\delta_1\cdot \mathrm{T}_1 &=-1, & \delta_{i}\cdot \mathrm{T}_1&=0, \ i\neq 0,1. 
\end{aligned}
\end{equation}

We define the F-curve of separating elliptic bridges (resp. non-separating elliptic bridges), denoted $\mathrm{EB}_i^{s}$ (resp. $\mathrm{EB}^{ns}$) as follows:
Take a $4$\nb-pointed rational curve
$(\PP^1, p_1,p_2,p_3,p_4)$ and two $1$\nb-pointed maximally degenerate curves $(E_1, q_1)$ of genus $i$
and $(E_2, q_2)$ of genus $g-i-1$ (resp. a $2$\nb-pointed maximally degenerate curve $(C, q_1, q_2)$ of genus $g-2$). Identify $p_1\sim q_1$, $p_4\sim q_2$ and $p_2\sim p_3$, and consider the family obtained by varying the $4$ points on $\P^1$.

\begin{figure}[h]
\scalebox{1}{\includegraphics{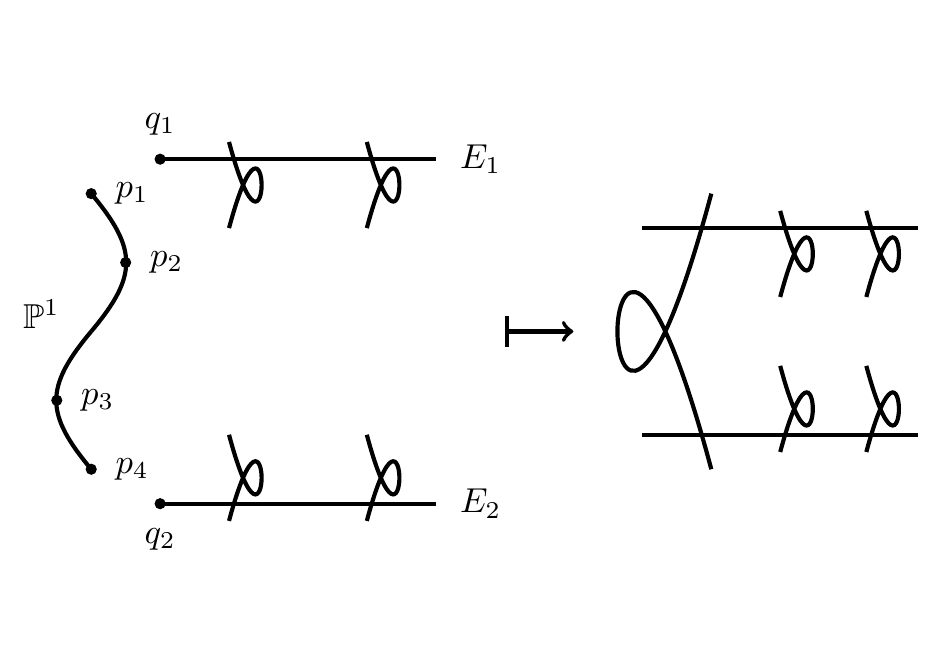}}
\caption{Separating elliptic bridge}
\label{F:elliptic-bridge}
\end{figure}
\noindent
The non-trivial interesection numbers of these F-curves are easily computed as
$$
\quad \lambda\cdot\textrm{EB}^s_i =0, \quad         \delta_1\cdot\textrm{EB}^s_i  =1, \quad        \delta_i\cdot\textrm{EB}^s_i  =-1,   \quad       \delta_{g-i-1}\cdot\textrm{EB}^s_i  = -1 
$$
in the separating case, and 
$$
\lambda\cdot\textrm{EB}^{ns} =0,  \quad \delta_1\cdot\textrm{EB}^{ns}=1, \quad \delta_0\cdot\textrm{EB}^{ns}=-2,
$$ 
in the non-separating case.

Finally, a general {\em family of  elliptic bridges} is obtained by taking a one-parameter
family of $2$-pointed genus $1$ 
curves in $\Mg{1,2}$ and attaching constant curves to obtain a family of genus $g$ curves. 
Because of the relation $\lambda=\delta_0/12$ and $\psi=2\lambda+2\delta_{0,\{1,2\}}$ in 
$\Pic(\Mg{1,2})$, the curve class of every 
one-parameter family of elliptic bridges 
is an effective
linear combination of classes $\textrm{T}_1$ and $\textrm{EB}^s_i$ if the bridge is separating,
and  classes $\textrm{T}_1$ and $\textrm{EB}^{ns}$ if the bridge is non-separaring.

\end{example}

Continuing as in the example, one may assemble a list of the intersection numbers of F-curves. Thus, the
 F-conjecture may be formulated as saying that a divisor $D=a\lambda-\sum_{i=0}^{\lfloor g/2 \rfloor}b_i\delta_i$ is nef if a certain finite set of linear inequalities holds 
 (see \cite[Theorem 2.1]{GKM} for a comprehensive list). In particular, the F-conjecture would imply the cone is finite polyhedral.

As we remarked above, perhaps the most striking evidence in favour of the F-conjecture is the theorem of Gibney, Keel and Morrison which says that if the F-conjecture holds for $\M_{0,g+n}/S_g$, then it holds for $\M_{g,n}$. Their argument proceeds by considering the map
$$
i\co \M_{0,g+n}/S_g \rightarrow \M_{g,n},
$$
obtained by attaching fixed genus $1$ curves onto the first $g$ (unordered) marked points. They then prove \cite[Theorem 0.3]{GKM}:
\begin{theorem}\label{T:bridge}
A divisor $D$ is nef on $\M_{g,n}$ if and only if $D$ is F-nef and $i^*D$ is nef on $\M_{0,g+n}$.
\end{theorem}
From this, they obtain two corollaries.
\begin{corollary}
If the F-conjecture holds for all $\M_{0,n}$, then it holds for all $\M_{g,n}$.
\end{corollary}
\begin{proof}
If $D$ is F-nef on $\M_{g,n}$, then $i^*D$ is F-nef on $\M_{0,g+n}$. If the F-conjecture holds for $\M_{0,n}$, then $i^*D$ is nef so the theorem implies $D$ is nef.
\end{proof}

We pause here to remark that the F-conjecture for $\M_{0,n}$ is a special case of a conjecture of Fulton stating that any effective $k$-cycle on $\M_{0,n}$ is an {\em effective} combination of $k$-dimensional strata of curves with $n-3-k$ nodes. As we have discussed above, the divisors constructed by Keel and Vermeire show that the conjecture is false for $k=n-4$. However, there is a similar statement, dubbed {\em Fulton's conjecture}, which if true would imply Conjecture \ref{Fconjecture}.
\begin{conjecture}[Fulton's conjecture]\label{Fulton-conjecture}
Every F-nef divisor on $\M_{0,n}$ is an effective combination of boundary divisors.
\end{conjecture} 
This conjecture was verified for $n=5, 6$ by Farkas and Gibney \cite{FG} and for $n=7$ by Larsen 
\cite{larsen}. In the special case of $S_m$-invariant ($m\geq n-3$) divisors it was proved by Fontanari  \cite{fontanari}.

We now return to Theorem \ref{T:bridge} and deduce a result that enables first several steps of the 
log minimal model program of $\M_{g}$.
\begin{corollary}[\text{\cite[Proposition 6.1]{GKM}}]\label{C:Nefness} Suppose that a divisor
$D=a\lambda-\sum_{i=0}^{\lfloor g/2\rfloor} b_i\delta_i$
is F-nef on $\M_g$, and for each $i\geq 1$, either $b_i=0$ or $b_i\geq b_0$, then $D$ is nef.
\end{corollary}
\begin{proof}
Consider the natural morphism $f\co \M_{0,2g} \ra \M_g$ defined by identifying $g$ pairs of points. The image of $f$ includes the locus of genus $g$ curves 
obtained from $\M_{0,g}$ by gluing on a fixed {\em nodal} elliptic curve. Hence by Theorem \ref{T:bridge} it suffices to show that $f^*D$ is nef. Suppose that $b_i\geq b_0$, for all $i\geq 1$. Then the divisor $f^*D$ is of the form $K_{\M_{0,2g}}+\sum_{i} a_i \Delta_i$ with $a_i\in [0,1]$. Applying \cite[Theorem 1.2(2)]{KMc}, we conclude that
$f^*D$ is nef. We refer to \cite{GKM} for details of the proof in the case $b_i=0$ for some $i$.
 \end{proof}

This result forms the main input into the following Proposition, which will be used in Section \ref{S:log-MMP}.
\begin{proposition}\label{P:first-F-curves} We have that in $\N^1(\M_g)$:
\begin{enumerate}
\item \label{Fcurves1} The divisor $11\lambda-\delta$ is nef  and has degree $0$ precisely on the 
families of elliptic tails, i.e. the F\nb-curve $\mathrm{T}_1$ of Example \ref{E:elliptic-bridges}.
\item \label{Fcurves2} The divisor $10\lambda-\delta-\delta_1$ is nef  and has degree $0$ precisely 
on the families of elliptic bridges. These are effective linear combinations of F\nb-curves $\mathrm{T}_1$, $\mathrm{EB}^{ns}$, and $\mathrm{EB}^{s}_i$ by Example \ref{E:elliptic-bridges}.
\end{enumerate}
\end{proposition}
\begin{proof}
These divisors satisfy the assumptions of Corollary \ref{C:Nefness}, hence are nef. Note also that \eqref{Fcurves1} follows from Theorem \ref{T:CH}. We proceed
to prove  \eqref{Fcurves2}.
One easily checks that the listed are the only F\nb-curves on which $10\lambda-\delta-\delta_1$ is zero. 
It remains to show that for
any curve $B\subset \Mg{g}$, the only moving components of the stable family over $B$ on which $10\lambda-\delta-\delta_1$ has degree zero are elliptic tails or elliptic bridges. 

Let $\X\ra B$ be a moving component with 
a smooth generic fiber of genus $h$ and $n$ sections. The divisor $10\lambda-\delta-\delta_1$ 
restricts to the divisor $$D:=10\lambda-\delta+\psi-\delta_{1, \varnothing}+\psi'$$ 
on $\Mg{h, n}$, where 
$\psi'$ corresponds to attaching sections of elliptic tails. 
If $h\geq 2$, we apply to Exercise \ref{E:forgetting} below to observe that the degree 
of $D$ is greater than the degree of $10\lambda-\delta_{0}-\delta_{1}$ 
on an unpointed family of stable curves of genus $h$. 
It then follows from Proposition \ref{P:moriwaki} that 
$(10\lambda-\delta_{0}-\delta_{1})\cdot B\geq 0$, and we are done. 
If $h=0$, then $\delta_{1, \varnothing}=0$ and so 
$D$ is positive on $B$ because it is a sum of $\psi-\delta$ and 
$\psi'$ which are, respectively, 
ample and nef on $\M_{0,n}$. Finally, suppose that $h=1$. By \eqref{E:elliptic-tail},
$10\lambda-\delta-\delta_1$ has degree $0$ on a family of elliptic tails. Suppose now $n\geq 2$, 
then, as before, $10\lambda-\delta-\delta_1$
restricts to $D=10\lambda-\delta+\psi-\delta_{1, \varnothing}+\psi'$, where $\psi'\cdot B\geq 0$.
 Using the following
relations in $\Pic(\Mg{1,n})$ \cite[Theorem 2.2]{AC}
\begin{align*}
\lambda =\delta_{0}/12, \qquad
\psi&=n\lambda+\sum_{S} |S|\delta_{0,S},
\end{align*}
and noting that $\delta_{1, \varnothing}=\delta_{0,\{1,\dots,n\}}$, we rewrite
\[
10\lambda-\delta+\psi-\delta_{1, \varnothing}=(n-2)\lambda+\sum_{S: \, 2\leq |S|<n}(|S|-1)\delta_{0,S}+(n-2)\delta_{0,\{1,\dots,n\}}.
\]
Evidently this can be zero only if $n=2$, i.e. when $\X \ra B$ is a family of elliptic bridges.
\end{proof}
\begin{exercise}\label{E:forgetting}
Suppose $B\subset \Mg{g,n}$ is a one-parameter family of generically smooth $n$\nb-pointed ($n\geq 1$) curves of genus $g\geq 2$, and let $B'$ be the family in $\Mg{g}$ obtained by forgetting marked points and stabilizing.
Then 
\[
(10\lambda-\delta-\delta_{1,\varnothing}+\psi)\cdot B > (10\lambda-\delta-\delta_{1})\cdot B'.
\]
\end{exercise}
For later reference, we record the loci swept out by curves numerically equivalent to an effective 
combination of the F-curves $\mathrm{T}_1$, $\mathrm{EB}^{ns}$, and $\mathrm{EB}^{s}_i$.
\begin{corollary}\label{C:Fcurves}
\begin{enumerate}
\item[]
\item A curve $C \subset \M_{g}$ is numerically equivalent to $\mathrm{T}_1$ iff every moving component of $C$ is an elliptic tail.
\item A curve $C \subset \M_{g}$
 is an effective linear combination of $\mathrm{T}_1$, $\mathrm{EB}^{ns}$, $\mathrm{EB}^{s}_i$ iff every moving component of $C$ is an elliptic tail or an elliptic bridge.
 \end{enumerate}
\end{corollary}
\begin{proof}
This is a restatement of Proposition \ref{P:first-F-curves}.
\end{proof}

\subsubsection{Koll\'{a}r's Semipositivity and nef divisors on $\M_{g,\A}$}\label{S:kollar}
Here, we explain a simple method for producing nef and ample divisors on spaces $\M_{g, \A}$ introduced in Definition \ref{D:A-stable}, which will be essential for describing the Mori chambers of $\M_{g,n}$ corresponding to these models. The main idea, due to Koll\'{a}r \cite{kollar-projectivity}, is to exploit the positivity of the canonical polarization. More concretely, if 
$(f\co \C \rightarrow B, \{\sigma_i\}_{i=1}^n)$ is any one-parameter family of $\A$-stable curves, Koll\'{a}r deduces that 
$\omega_{\C/B}(\sum a_i\sigma_i)$ is nef from the semipositivity of 
$f_*\bigl((\omega_{\C/B}(\sum a_i\sigma_i))^{\ell}\bigr)$, where $\ell\geq 2$ is such that $\ell a_i\in \ZZ$; see \cite[Corollary 4.6 and Proposition 4.7]{kollar-projectivity}. 
By interesecting with other curve classes on $\C$ and pushing forward, 
one gets a variety of nef divisor classes on $\M_{g,n}$, as we shall see in Corollary
\ref{C:nef-divisors-weighted}. 
 \begin{proposition}\label{P:nef-divisors-weighted}
Let
 $\pi\co (\C; \sigma_1, \dots, \sigma_n) \ra \Mg{g,\A}$ be the universal family. 
 Then the line bundle $\omega_{\pi}(\sum_{i=1}^n a_i\sigma_i)$ is nef on $\C$.
 \end{proposition}
\begin{proof}
See \cite{FedAmple} for a direct, elementary proof in the spirit of \cite[Theorem 0.4]{keel-annals}.
\end{proof}

Using Proposition \ref{P:nef-divisors-weighted}, we can describe several nef divisor classes on $\M_{g,\A}$. First, however, we need a bit of notation. For each $i, j$ such that $a_i+a_j\leq 1$, there is an irreducible boundary divisor $\Delta_{i,j} \subset \M_{g,\A}$ parameterizing curves where the marked point $p_i$ coincides with the marked point $p_j$. Indeed, if $\pi\co \C \rightarrow \M_{g,\A}$ is the universal curve with universal sections $\{\sigma_i\}_{i=1}^{n}$, then $\Delta_{i,j}:=\pi_*(\sigma_i \cap \sigma_j)$. 
The remaining boundary divisors of $\Mg{g,\A}$ parameterize nodal curves and for this reason we denote the total 
class of such divisors by $\delta_{nodal}$. Note that Mumford's relation gives $\kappa:=\omega_{\pi}^2=12\lambda-\delta_{nodal}$  \cite{AC}.
 \begin{corollary}\label{C:nef-divisors-weighted}
The following divisors 
\begin{align*}
A &=A(a_{1},\dots,a_{n})= 12\lambda-\delta_{nodal}+\psi+\sum_{i<j}(a_i+a_j)\Delta_{ij},  \\
B &=B(a_{1},\dots,a_{n}) = 12\lambda-\delta_{nodal}+\sum (2a_i-a_i^2)\psi_i+\sum_{i<j}(2a_ia_j) \Delta_{ij},\\
C &=C(a_{1},\dots,a_{n}) = \sum (1-a_{i})\psi_{i}+\sum_{i<j}(a_i+a_j)\Delta_{ij},
\end{align*}
are nef on $\Mg{g,\A}$. Moreover, the divisor $A$ is ample.
 \end{corollary}
 \begin{proof}
Let $f\co \X \ra T$ be a family over a complete smooth curve. The divisor 
$L:=\omega_{\X/T}+\sum a_i \sigma_i$ is nef by Proposition \ref{P:nef-divisors-weighted}, hence
pseudoeffective and has a non-negative 
 self-intersection. We will show that the intersection 
 numbers of $T$ with
 $A$, $B$ and $C$ are non-negative by expressing each of them as an intersection of $L$ 
 with an effective curve class on $\X$. 
 For $A$, we note that $\omega+\sum_{i=1}^n\sigma_i$ is an effective combination of $L$ and $\sigma_i$, $1\leq i\leq n$. 
 Therefore,
\[
0\leq (\omega+\sum_{i=1}^n \sigma_i)\cdot L=\kappa+\psi+\sum_{i<j}(a_i+a_j)\Delta_{ij}.
\]
 For $B$, we have
 \[
 0\leq L^2=(\omega+\sum_{i=1}^n a_i\sigma_i)^2=\kappa+\sum_{i=1}^n (2a_i-a_i^2)\psi_i+\sum_{i<j}(2a_ia_j) \Delta_{ij}. 
 \]
For $C$, we have
\[
0\leq L\cdot \sum_{i=1}^n\sigma_{i}=\sum (1-a_{i})\psi_{i}+\sum_{i<j}(a_j+a_j)\Delta_{ij}.
\]
For the proof of ampleness in the case of $A$, we refer to \cite{FedAmple}, where it is established using
Kleiman's criterion on $\M_{g,\A}$.
 \end{proof}

Now we can describe all $\Mg{g,\A}$ as log canonical models of $\Mg{g, n}$. While these models do not (except in the case $g=0$) appear in the Hassett-Keel minimal model program for $\Mg{g,n}$, they do correspond to fairly natural Mori chambers. 

\begin{corollary}\label{C:MgA-canonical}
\[
\overline{M}_{g,\A}=\proj \bigoplus_{m\geq 0}\HH^{0}\bigl(\Mg{g,n}, \lfloor m(K_{\Mg{g,n}}+
\sum_{i<j: a_i+a_j\leq 1}(a_{i}+a_{j}-1)\Delta_{0, \{i,j\}}+\delta)\rfloor \bigr),
\]
where, $\delta$ is the total boundary of $\Mg{g,n}$.
\end{corollary}

\begin{proof} 
Let $D:=K_{\Mg{g,n}}+
\sum_{i<j: a_i+a_j\leq 1}(a_{i}+a_{j}-1)\Delta_{0, \{i,j\}}+\delta$. Consider the birational reduction morphism $\phi \co \Mg{g,n} \ra \Mg{g,\A}$; see \cite[Section 4]{Hweights}.
Then
$$
\phi_*D:=K_{\Mg{g,\A}}+\sum\limits_{i<j: a_i+a_j\leq 1}(a_{i}+a_{j})\Delta_{i,j}+\delta_{nodal}
$$
may be expressed as a sum of tautological classes using the Grothendieck-Riemann-Roch formula 
(see \cite[Section 3.1.1]{Hweights})
$
K_{\Mg{g,\A}}=13\lambda-2\delta_{nodal}+\psi.
$

 We must show that
\begin{enumerate}
\item $\phi_*D$ is ample on $\Mg{g,\A}$,
\item $D-\phi^*\phi_*D$ is effective.
\end{enumerate}
Indeed, (1) implies $\Mg{g,\A}=\proj \bigoplus_{m\geq 0}\HH^{0}\bigl(\Mg{g,\A}, \lfloor m\phi_*D \rfloor)$, and (2) implies $$\bigoplus_{m\geq 0}\HH^{0}\bigl(\Mg{g,n}, \lfloor mD \rfloor)=\bigoplus_{m\geq 0}\HH^{0}\bigl(\Mg{g,\A}, \lfloor m\phi_*D \rfloor),$$ so together they yield the desired statement.

For (1), simply observe that $D=A+\lambda$, where $A$ is as in Corollary \ref{C:nef-divisors-weighted}. It follows that $D$ is a sum of an ample divisor $A$ and a semiample divisor $\lambda$, so is ample. 

For (2),  the morphism $\phi$ sends a stable $n$\nb-pointed curve to an $\A$\nb-stable curve
obtained by collapsing all rational components on which $\omega_{\pi}(\sum a_i \sigma_i)$ has non-positive degree. From this, one easily sees that the exceptional divisors of $\phi$ are given by
$\Exc(\phi)=\bigcup \Delta_{0,S}$ for all $S\subset \{1,\dots, n\}$ such that $|S|\geq 3$ and $\sum_{i\in S} a_{i}\leq 1$. Thus, we may write
$$D-\phi^*\phi_*D =\sum_{|S| \geq 3}a_{S}\delta_{0,S},$$
where $a_{S} \in \Q$ is the \emph{discrepancy} of $\Delta_{0,S}$. A simple computation with test curves shows that $a_{S}=\left(|S|-1\right)\left(1-\sum_{i\in S} a_i\right)\geq 0$, which completes the proof of (2).
\end{proof}

\section{Log minimal model program for moduli spaces of curves}\label{S:log-MMP}
In this section, we will describe what is currently known regarding the Mori chamber decomposition of
 $\Eff(\M_{g,n})$, with special emphasis on those chambers corresponding to modular birational models of $\M_{g,n}$.  As we have seen in Section \ref{S:nef-cone}, the full effective cone of 
$\M_{g,n}$ is completely unknown, so it is reasonable to focus attention on the \emph{restricted effective cone} $\Eff_\rst(\M_{g,n})$, i.e. the intersection of the effective cone with the subspace $\Q\{\lambda, \psi, \delta \} \subset \N^1(\M_{g,n})$. Consequently, 
most the results of this section will concern the restricted effective cone.

In Section \ref{S:MgIntro}, we focus on the Mori chamber decomposition of $\M_{g}$. In Sections \ref{S:MMPM2} and \ref{S:MMPM3}, we give complete Mori chamber decompositions for the restricted effective cones of $\M_{2}$ and $\M_{3}$, following Hassett and Hyeon-Lee \cite{Hgenus2}, \cite{HL}. In Section \ref{S:MMPMg}, we describe two chambers of the restricted effective cone of $\M_{g}$ (for all $g \geq 3$), corresponding to the first two steps of the log minimal model program for $\M_{g}$, as carried out by Hassett and Hyeon \cite{HH1, HH2}.

In Section \ref{S:MMPM0n}, we turn our attention to the Mori chamber decomposition of $\M_{0,n}$. We will give a complete Mori chamber decomposition for half the restricted effective cone (divisors of the form $s\psi-\delta$) while the other half (divisors of the form $s\psi+\delta$) remains largely mysterious. We will see that every chamber in the first half-space corresponds to a modular birational model of $\M_{0,n}$, namely $\M_{0,\A}$ for a suitable weight vector $\A$.
Finally, in Section \ref{S:MMPM1n}, we will tackle $\M_{1,n}$, the first example where the restricted effective cone is three-dimensional. As with $\M_{0,n}$, we will give a complete Mori-chamber decomposition for half the restricted effective cone. We will see that every chamber corresponds to a modular birational model of $\M_{1,n}$, with both $\M_{1,\A}$, the moduli spaces of weighted stable curves, and $\M_{1,n}(m)$, the moduli spaces of $m$-stable curves, making an appearance. 

Before proceeding, let us indicate the general strategy of proof which is essentially the same in each of the cases considered. Given a divisor $D \in \N^1(\M_{g,n})$ and a 
birational 
contraction 
$\phi\co \M_{g,n} \dashrightarrow \M$, where $\M$ is some alternate modular compactification, to prove that $D \in \Mor(\phi)$ requires two calculations.
First, one must show that
$D-\phi^*\phi_*D \geq 0$, so that $\HH^0(\M_{g,n}, mD)=\HH^0(\M, mD)$ for all $m \geq 0$. Second, one must show that $\phi_*D$ is ample on $\M$. Together, these two facts immediately imply that
\[
\M=\Proj R(\M, \phi_*D)=\Proj R(\M_{g,n}, D),
\]
so $D \in \Mor(\phi)$ as desired.

To carry out the first step, one typically uses the method of test curves: Write
$$
D-\phi^*\phi_*D= \sum a_iD_i,
$$
where $D_i$ are generators for the Picard group of $\M_{g,n}$ and $a_i \in \Q$ are undetermined coefficients. If $C$ is any curve which is both contained in the locus where $\phi$ is regular and is contracted by $\phi$, then we necessarily have $\phi^*\phi_*D\cdot C=0$. Thus, if the intersection numbers $D.C$ and $D_i.C$ can be determined, one may solve for the coefficients $a_i$.

In order to show that $\phi_*D$ is ample, there are essentially two strategies. If $\phi$ is regular, one may consider the pull-back $\phi^*\phi_*D$ to $\M_{g,n}$. If $\phi^*\phi_*D$ is nef, has degree zero only on $\phi$\nb-exceptional curves, and can be expressed as $K_{\M_{g,n}}+\sum D_i$ with 
$(\M_{g,n}, \sum D_i)$ a klt pair, then one may conclude $\phi_*D$ is ample using the Kawamata basepoint freeness theorem (Theorem \ref{T:kawamata}). 
If $\phi$ is not regular, one may show $\phi_*D$ is ample using Kleiman's criterion, i.e. by proving that $\phi_*D$ is positive on every curve in $\M$. The key point is that $\M$ is modular, so one can typically write $\phi_*D$ as a linear combination of tautological classes on $\M$ whose intersection with one-parameter families can be evaluated by geometric methods.


\subsection{Log minimal model program for $\M_{g}$}\label{S:MgIntro}
The restricted effective cone of $\M_{g}$ is two-dimensional, and we have
\begin{align*}
\Eff_\rst(\M_{2}) &= \Q\{\delta_1, \delta_{0}\}\\
\Eff_\rst(\M_{g}) &=\Q\{\delta, s_{\eff}\lambda-\delta\}, \ \text{for } g \geq 3,
\end{align*}
where the slope of the effective cone $s_{\eff}$ satisfies $\frac{60}{g+4}<s_{\eff}\leq 6+\frac{12}{g+1}$ (as we have discussed in Section \ref{S:nef-cone}, $s_{\eff}$ remains unknown for all but finitely many $g$).

Let us note that one chamber of $\Eff(\M_{g})$ is easily accounted for. By \cite{namikawa}, the Torelli morphism $\tau\co M_{g} \hookrightarrow A_{g}$ extends to give a birational map
$$
\tau\co \M_{g} \rightarrow \tau(\M_{g}) \subset \Abar_{g},
$$
where $\Abar_{g}$ is the Satake compactification of the moduli space of polarized abelian varieties
of dimension $g$. The map $\tau$ contracts the entire boundary of $\M_{g}$ (resp. $\delta_{0}$) when $g \geq 3$ (resp. $g=2$), and satisfies $\tau^*\O_{\Abar_g}(1)=\lambda$, where $\O_{\Abar_g}(1)$ is the canonical polarization on $\Abar_{g}$ coming from its realization as the $\Proj$ of the graded algebra of Siegel modular forms \cite{namikawa}. It follows immediately that
\begin{align*}
\Mor(\tau)&=\Q\{\lambda, \delta_{0}\}, \ g=2, \\
\Mor(\tau)&=\Q\{\lambda, \delta\}, \ \ g\geq 3.
\end{align*}
Thus, the only interesting part of the restricted effective cone lies in the quadrant spanned by $\lambda$ and $-\delta$. Any divisor in this quadrant is proportional to a uniquely defined divisor of the form $s\lambda-\delta$,
where $s$ is the {\em slope} of the divisor. 
On the other hand, since $K_{\Mg{g}}=13\lambda-2\delta,$ we also have
$$
s\lambda-\delta=K_{\SM_{g}}+\alpha\delta
$$
for $\alpha=2-13/s$, so we may use either the parameter $\alpha$ or $s$ to express the boundary thresholds in the Mori chamber decomposition of $\Eff_\rst(\M_{g})$. From the point of view of higher-dimensional geometry, $\alpha$ is more natural, while from the point of view of past developments in moduli of curves, the slope is more natural. In order to have all information available, we will label our figures with both.

\subsubsection{Mori chamber decomposition for $\M_{2}$.}
\label{S:MMPM2}
The geometry of $\M_{2}$ is sufficiently simple that it is possible to give a complete Mori chamber decomposition of $\Eff(\M_{2})$. We have encountered precisely two alternate birational models of $\M_{2}$ so far, $\Abar_2$ and $\M_2[A_2]\simeq \M_2^{\, ps}$. 

\begin{remark}\label{R:M2}
While the construction of $\M_{g}^{\, ps}=\M_{g}[A_2]$ in Corollary \ref{C:pseudostable-stack} was restricted to $g \geq 3$, one may still consider the stack $\SM_{2}[A_2]$ of $A_2$-stable curves of genus two. This stack is not separated as one can see from the fact that one point of $\SM_2[A_2]$, namely
the unique rational curve with cusps at $0$ and $\infty$, has automorphism group $\GG_m$. Nevertheless, $\SM_{2}[A_2]$ gives rise to a weakly modular compactification $\M_{2}[A_2]$. The good moduli map $\SM_{2}[A_2] \rightarrow \M_{2}[A_2]$ maps all cuspidal curves to a single point $p \in \M_{2}[A_2]$, and the aforementioned rational bicuspidal curve is the unique closed point in $\phi^{-1}(p)$. For a full discussion of these matters, we refer the reader to \cite{Hgenus2}. For our purposes, the essential fact we need is the analogue of Proposition \ref{P:birational-maps}, i.e. 
the existence of a birational contraction $\eta\co\M_{2} \rightarrow \M_{2}[A_2]$ with $\Exc(\eta)=\Delta_1$ and $\eta(\Delta_1)=p$. 
\end{remark}

We proceed to describe the Mori chambers associated to the models $\M_{2}[A_2]$ and $\Abar_{2}$.
Recall that $\N^1(\M_{2})$ is two-dimensional, generated by $\delta_{0}$ and $\delta_1$, with the relation $\lambda=\frac{1}{10}\delta_{0}+\frac{1}{5}\delta_1$ \cite{AC}. Thus, the Mori chamber associated to any $\Q$\nb-factorial rational contraction will be a two-dimensional polytope, spanned by two extremal rays.
\begin{lemma}
$\Mor(\Abar_2)$ is spanned by $\lambda$ and $\delta_{0}$.
\end{lemma}
\begin{proof}
Since $\tau\co \M_{2} \rightarrow \Abar_{2}$ is a divisorial contraction, 
$\N^1(\Abar_2)$ is one-dimensional, generated by $\O_{\Abar_{2}}(1)$. Since $\tau^*\O_{\Abar_{2}}(1)=\lambda$ and $\delta_{\irr}$ is $\tau$-exceptional, $\Mor(\Abar_2)$ is spanned by $\lambda$ and $\delta_{\irr}$.
\end{proof}
\begin{lemma}
$\Mor(\M_{2}[A_2])$ is spanned by $11\lambda-\delta$ and $\delta_1$.
\end{lemma}
\begin{proof}
By Remark \ref{R:M2}, there is a divisorial contraction $\eta\co \M_{2} \rightarrow \M_{2}[A_2]$  contracting $\delta_1$, so it is sufficient to show that $11\lambda-\delta$ is semiample, pulled back from an ample divisor on $\M_{2}[A_2]$. 

Since $\eta$ is an extremal divisorial contraction, $\M_{2}[A_2]$ is $\Q$-factorial \cite[Corollary 3.18]{kollar-mori}. 
Thus, $\lambda:=\eta_*\lambda$ and $\delta:=\eta_*\delta$ make sense as numerical divisor classes. We claim that
\begin{equation} \label{E:M2-discrepancy}
\begin{aligned}
\eta^*\lambda &=\lambda+\delta_1,\\
\eta^*\delta &=\delta_{0}+12\delta_1, \\
\eta^*(11\lambda-\delta)&=11\lambda-\delta.
\end{aligned}
\end{equation}
To see this, note that the morphism $\eta$ contracts a family $E$ of elliptic tails whose
intersection numbers (see Example \ref{E:elliptic-bridges}) are $E\cdot \lambda=1$, 
$E\cdot \delta_{0}=12$, and $E\cdot \delta_1=-1$.
Writing $\eta^*(\lambda)=\lambda+a\delta_1$ and intersecting with $E$, we obtain $a=1$. The second formula is proved in similar fashion. The third
follows immediately from the first two.

Now the divisor $11\lambda-\delta$ is nef and has degree zero only on curves lying in $\Delta_1$ by Proposition \ref{P:first-F-curves}. Since
$$
11\lambda-\delta \sim K_{\Mg{2}}+\frac{9}{11}\delta,
$$
the Kawamata basepoint freeness theorem \ref{T:kawamata} implies that $11\lambda-\delta$ is semiample. In other words, $\eta^*(11\lambda-\delta)$ is semiample and contracts only $\eta$-exceptional curves. It follows that $11\lambda-\delta$ is ample on $\M_2[A_2]$ as desired.
\end{proof}

 \begin{corollary}\label{C:M2-eff-nef}
 \begin{enumerate}
 \item[]
 \item\label{C:M2-eff}
 $\Eff(\M_{2})$ is spanned by $\delta_{0}$ and $\delta_1=\frac{10}{13}(K_{\Mg{2}}+\frac{7}{10}\delta)$.
\item\label{C:M2-nef}  $\Nef(\M_{2})$ is spanned by $11\lambda-\delta$ and $\lambda$.
\end{enumerate}
\end{corollary}
\begin{proof} \par

(1) Since $\delta_{0}$ and $\delta_1$ are each contracted by a divisorial contraction, they are extremal rays of the effective cone. That $\delta_1=\frac{10}{13}(K_{\Mg{2}}+\frac{7}{10}\delta)$ 
follows from relations $K_{\Mg{2}}=13\lambda-2\delta$
and $\lambda=\frac{1}{10}\delta_0+\frac{1}{5}\delta_1$. 

(2) We have seen that $11\lambda-\delta$ and $\lambda$ are semiample, and since $\Nef(\M_{2})$ is two-dimensional, they span the nef cone.
 \end{proof}
 
 \begin{corollary}
 The Mori chamber decomposition has precisely three chambers.
 \end{corollary}
 \begin{proof}
 Evidently, the Mori chambers of $\Abar_2$, $\M_{2}[A_2]$, and $\M_{2}$ span the effective cone.
 \end{proof}
 
 The full Mori chamber decomposition of $\M_{2}$ is displayed in Figure \ref{F:M2-chamber}. For the sake of future comparison, let us formulate this result in terms of the log minimal model program for $\M_{2}$. This is, of course, equivalent to listing the Mori chamber decomposition for the first quadrant, using the parameter $\alpha$. We have
 
\begin{equation}
 \M_{2}(\alpha)=
\begin{cases}
\M_{2} 		& \text{ iff } \alpha \in (9/11, 1],\\
\M_{2}[A_2] 	& \text{ iff } \alpha \in (7/10,9/11],\\
\text{point}         & \text{ iff } \alpha = 7/10.
\end{cases}
\end{equation}

\begin{figure}
\scalebox{0.7}{\includegraphics{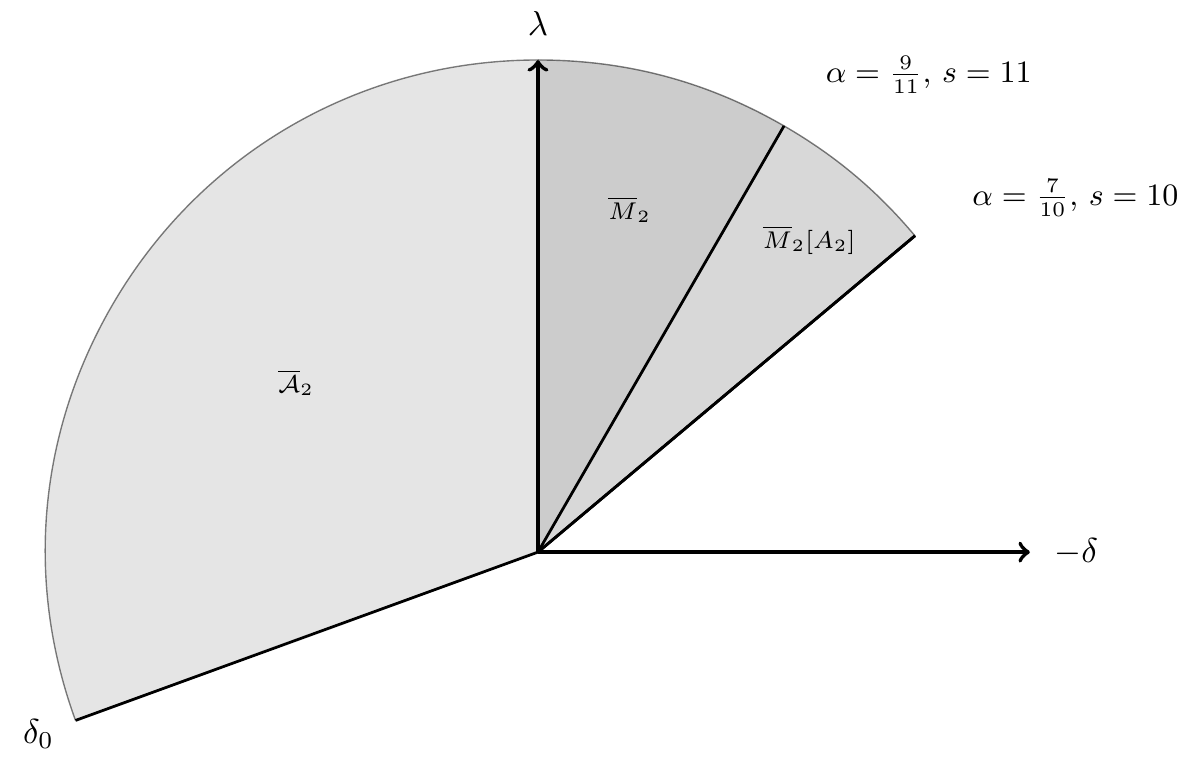}}
\caption{Restricted effective cone of $\M_2$.}
\label{F:M2-chamber}
\end{figure}

\subsubsection{Mori chamber decomposition for $\M_{3}$.}
\label{S:MMPM3}
The birational geometry of $\M_3$ is far more intricate than that
of $\M_2$. For example, as we mentioned in Section \ref{S:nef-cone}, the rationality of $\M_3$ was not established until 1996 \cite{katsylo-3}. Even before the advent of the Hassett-Keel program, several alternate birational models of $\M_3$ had been constructed and studied, including the Satake compactification
$\tau(\M_3)$, Mumford's GIT quotient of the space of plane quartics, and Schubert's space of pseudostable curves.  A systematic study of the Mori chamber decomposition of $\M_3$ was undertaken by Rulla \cite{rulla}, who showed (among other things) the existence of a divisorial contraction $\M_{3} \rightarrow X$ contracting $\Delta_1$, which is \emph{not} isomorphic to the contraction $\M_3\ra\M_3^{\, ps}$. Here, $X$ can be realized as the image of $\M_{3}$ under the natural map to Alexeev's space of semi-abelic pairs \cite{alexeev-annals}. Around the same time, Kondo proved that $M_3$
is birational to a complex ball quotient \cite{kondo-3}. This gives rise to a proper birational model of $\M_3$ by taking the Baily-Borel compactification 
of the ball quotient. Using the theory of $K3$ surfaces, Artebani constructed an alternative compactification of $M_3$ that admits a regular morphism to the Kondo's space \cite{artebani}. 
These developments were rounded off by Hyeon and Lee who described
all the log canonical models $\M_3(\alpha)$ and their relation with 
previously known birational models of $\M_3$. In particular, their work implies that the Artebani's compactification is isomorphic to $\M_{3}[A_3]$ and that the Kondo's compactification is isomorphic to $\M_{3}[A_3^*]$.

To describe where the Mori chambers of the known birational models of $\M_3$ fall inside the effective cone, we begin with the well-known fact that the effective cone inside $\N^1(\M_{3})=\Q\{\lambda, \delta_{0}, \delta_1\}$ is generated by $9\lambda-\delta_{0}-3\delta_1$, $\delta_{0}$, and $\delta_1$, where $9\lambda-\delta_{0}-3\delta_1$ is the class of the hyperelliptic divisor (c.f. \cite{rulla}). In particular, the restricted effective cone is spanned by $\delta$ and $9\lambda-\delta$, or, equivalently, by $\delta$ and $K_{\SM_{3}}+\frac{5}{9}\delta$. The following proposition, due to Hyeon and Lee \cite{HL}, shows that weakly modular birational models account for a complete Mori chamber decomposition of the restricted effective cone of $\M_3$ (see Figure \ref{F:M3-chamber}).
\begin{proposition}\label{P:M3}
 $$
 \M_{3}(\alpha)=
\begin{cases}
\M_{3} 		&\text{ iff } \alpha \in (9/11, 1], \\
\M_{3}[A_2] 	& \text{ iff } \alpha \in (7/10,9/11], \\
\M_{3}[A_3^*] 	&\text{ iff } \alpha=7/10, \\
\M_{3}[A_3] 	&\text{ iff } \alpha=(17/28, 7/10), \\
\M_{3}[Q] 	&\text{ iff } \alpha=(5/9, 17/28], \\
\text{point}			&\text{ iff } \alpha = 5/9.
\end{cases}
$$
These log canonical models and the morphisms between them fit into the following diagram
\begin{equation}\label{E:M3-maps}
\xymatrix{
\M_{3}\ar[d]^{\eta}&&\\
\M_{3}[A_2]\ar[rd]^{\phi^-}&&\M_{3}[A_3] \ar[ld]_{\phi^+} \ar[rd]^{\psi}\\
&\M_{3}[A_3^*]&&\M_{3}[Q]\\
}
\end{equation}
\end{proposition}

\begin{figure}[t]
\includegraphics{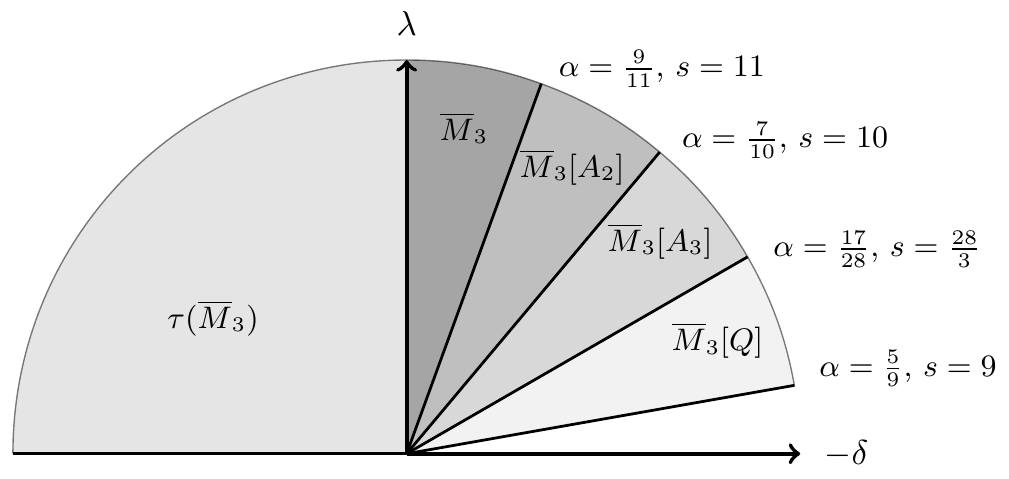}
\caption{Restricted effective cone of $\M_3$.}
\label{F:M3-chamber}
\end{figure}

Recall that the spaces $\M_{3}[A_2]$, $\M_{3}[A_3^*]$, $\M_{3}[A_3]$ were described using 
GIT in Section \ref{S:GIT} and 
the maps $\eta$, $\phi$, $\phi^{+}$ have been described in Proposition \ref{P:birational-maps}.
By contrast, the space $\M_{3}[Q]$, where $Q$ stands for `quartic,' has not yet been described. It is the GIT quotient of the space of degree $4$ plane curves, i.e. $\Hilb_{3,1}^{\ss} \gitq \SL(3)= \PP \HH^0(\PP^2,\O_{\PP^2}(4))\gitq \SL(3)$,
with the uniquely determined linearization. Whereas the GIT analysis of semistable points for the Hilbert scheme of canonically embedded curves has never been carried out for all $g$, the genus three case is a classic example, and the stable and semistable points $\Hilb_{3,1}^{\ss} \subset \PP \HH^0(\PP^2,\O_{\PP^2}(4))$ are described in \cite[Chapter 3.2]{GIT}. The nonsemistable points are precisely quartics with a point of multiplicity $3$ and quartics which are 
a union of a plane cubic with its flex line. The stable points are plane quartics with at worst cusps. The rest are semistable: these are plane quartics with at best $A_3$ and worst $A_7$ singularities, as well as double (smooth) conics. Note that since all strictly semistable curves specialize isotrivially to a double conic, and any two double conics are projectively equivalent, all strictly semistable points correspond to a unique point $p \in \M_{3}[Q]$.

It is straightforward to see that the natural rational map $\psi\co \M_{3}[A_3] \dashrightarrow \M_{3}[Q]$ is regular, contracts the hyperelliptic locus in $\M_{3}[A_3]$  to $p$, and is an isomorphism elsewhere. That it is an isomorphism away from the hyperelliptic locus is an immediate consequence of the fact that any $A_3$-stable curve which is not hyperelliptic is still $Q$-stable, i.e. contained in the semistable locus $\Hilb_{3,1}^{\ss}$.  To see that $\psi$ maps the hyperelliptic divisor to $p$, it suffices to observe that the stable limit of any family of 
quartics degenerating to a {\em stable} quartic does not lie in the hyperelliptic locus.
Therefore, when the generic genus $3$ curve specializes to the hyperelliptic locus, the semistable limit in 
$\M_3[Q]$ is $p$.

\begin{proof}[Proof of Proposition \ref{P:M3}]
Throughout the proof, we will abuse notation by using the symbols $\lambda$, $\delta_{0}$, and $\delta_1$ to denote the divisor classes obtained by pushing-forward to each of the birational models the divisor classes of the same name on $\M_{3}$.

We proceed chamber by chamber. First, observe that
$$
K_{\SM_{3}}+\alpha\delta=13\lambda-(2-\alpha)\delta \sim \frac{13}{2-\alpha}\lambda-\delta.
$$
For $\alpha>9/11$, this divisor has slope greater than $11$, hence is ample by Theorem \ref{T:CH}. 
This immediately implies that
$$
\M_{3}(\alpha)=\M_{3} \text{ for $\alpha \in \bigl(9/11, 1\bigr]$}.
$$

At $\alpha=\frac{9}{11}$, the divisor is numerically proportional to $11\lambda-\delta$, which is nef and has degree zero precisely on $\mathrm{T}_1$ -- the curve class of elliptic tails (see Proposition \ref{P:first-F-curves}). Applying the Kawamata basepoint freeness theorem \ref{T:kawamata} to $K_{\SM_{3}}+\frac{9}{11}\delta$, we obtain a birational contraction\footnote{Since $\M_{3}$ is mildly singular, it is not completely obvious that $(\M_{3}, \frac{9}{11}\Delta)$ is a klt pair; this is verified in \cite{HL}. In order to simplify exposition, we will omit standard discrepancy calculations needed to justify the use of the basepoint freeness theorem.}
$$
\M_{3} \rightarrow \M_{3}\bigl(9/11\bigr),
$$
contracting the curve class $T_1$. Since this is precisely the curve class contracted by the map $\eta\co \M_{3} \rightarrow \M_{3}[A_2]$ (Proposition \ref{P:birational-maps}), we may identify  $\M_{3}(\frac{9}{11}) \simeq \M_{3}[A_2]$. In particular, $K_{\Mg{3}[A_2]}+\frac{9}{11}\delta$ is Cartier and ample on $\M_{3}[A_2]$. In order to determine the chamber of $\M_{3}[A_2]$, the key question is: As we scale $\alpha$ down from $\frac{9}{11}$, how long does $K_{\Mg{3}[A_2]}+\alpha\delta$ remain ample? The simplest way to answer this question is to compute the pull back
\begin{equation}\label{E:discrepancy}
\eta^*(K_{\Mg{3}[A_2]}+\alpha\delta)=13\lambda-(2-\alpha)\delta_{0}-(11-12\alpha)\delta_1
=K_{\Mg{3}}+\alpha\delta+(11\alpha-9)\delta_1
\end{equation}
using \eqref{E:M2-discrepancy} (which holds for all $g\geq 2$), and observe that $13\lambda-(2-\alpha)\delta_{0}-(11-12\alpha)\delta_{1}$ is a positive linear combination of $11\lambda-\delta$ and $10\lambda-\delta-\delta_1$ for $\alpha \in (\frac{7}{10}, \frac{9}{11})$. It follows from Proposition \ref{P:first-F-curves} that
for  $\alpha \in (\frac{7}{10}, \frac{9}{11})$, the divisor $\eta^*(K_{\M_{3}[A_2]}+\alpha\delta)$ is nef and has degree zero precisely on the curve class $\mathrm{T}_1$, hence descends to an ample divisor on $\M_{3}[A_2]$. We conclude that
\[
\M_{3}(\alpha)=\M_{3}[A_2] \text{  for $\alpha \in \bigl(7/10, 9/11\bigr]$}.
\]
\noindent
Furthermore, at $\alpha=\frac{7}{10}$,
$$
\eta^*(K_{\Mg{3}[A_2]}+\alpha\delta)=10\lambda-\delta-\delta_1.
$$
This divisor has degree zero on the curve classes described in Proposition \ref{P:first-F-curves} (2),
namely the class $\mathrm{T}_1$ of the family of elliptic tails and the curve classes 
$\mathrm{EB}^{s}_i$, $\mathrm{EB}^{ns}$ of families of elliptic bridges. 
Thus,  $K_{\Mg{3}[A_2]}+\frac{7}{10}\delta$ is nef on $\M_{3}[A_2]$ and has degree zero precisely on the curve classes $\eta(\mathrm{EB}^{s}_i)$ and $\eta(\mathrm{EB}^{ns})$. Applying the Kawamata basepoint freeness to $K_{\Mg{3}[A_2]}+\frac{7}{10}\delta$,  we obtain a map $\M_{3}[A_2] \rightarrow \M_{3}(\frac{7}{10})$ which contracts these curve classes. By Corollary \ref{C:Fcurves}, these curve classes sweep out precisely the locus of elliptic bridges, i.e. the image of the natural gluing map $\M_{1,2} \times \M_{1,2} \rightarrow \M_{3} \rightarrow \M_{3}[A_2]$. Thus, $\M_{3}[A_2] \rightarrow \M_{3}(\frac{7}{10})$ contracts the locus of the elliptic bridges to a point, and is an isomorphism elsewhere.

Now consider the map $\phi^-:\M_{3}[A_2] \rightarrow \M_{3}[A_3^*]$. By Proposition \ref{P:birational-maps}, $\phi^-$ also contracts precisely the locus of elliptic bridges, so we obtain an identification
$$\M_{3}(7/10) \simeq \M_{3}[A_3^*].$$

Next, we must compute the Mori chamber of $\M_{3}[A_3]$, i.e. we must show
$$
\M_{3}(\alpha)=\M_{3}[A_3] \text{  for $\alpha \in (17/28, 7/10)$}.
$$
Note that knowing the ample cone of $\M_3$ can no longer help us here, since the map $\M_{3} \dashrightarrow \M_{3}[A_3]$ is only rational. Instead, we use a slighly ad-hoc argument, applicable only in genus three. We claim that it is sufficient to prove $K_{\Mg{3}[A_3]}+\frac{7}{10}\delta$ is nef, with degree zero precisely on the locus of tacnodal curves, and that $K_{\Mg{3}[A_3]}+\frac{17}{28}\delta$ is nef, with degree zero precisely on the locus of hyperelliptic curves. Indeed, one can check that the intersection of the locus of tacnodal curves and the closure of the hyperelliptic curves on $\M_{3}[A_3]$ consists of a single point, so that a positive linear combination of $K_{\Mg{3}[A_3]}+\frac{7}{10}\delta$ and $K_{\Mg{3}[A_3]}+\frac{17}{28}\delta$ has positive degree on all curves, hence is ample.\footnote{To check that a klt divisor is ample, it suffices (by the Kawamata basepoint freeness theorem \ref{T:kawamata}) to prove positivity on curves. } Furthermore, for any $\alpha < \frac{9}{11}$, we have 
$$
R(\Mg{3}[A_3], K_{\Mg{3}[A_3]}+\alpha\delta)=R(\Mg{3}[A_2], K_{\Mg{3}[A_2]}+\alpha\delta)=R(\Mg{3}, K_{\Mg{3}}+\alpha\delta),
$$
the first equality holding because $\M_{3}[A_2] \dashrightarrow \M_{3}[A_3]$ is an isomorphism in codimension one, and the second holding by the discrepancy computation in Equation \eqref{E:discrepancy}. Together these observations immediatley imply $\M_{3}(\alpha)=\M_{3}[A_3]$ for $\alpha \in (17/28, 7/10)$.

Now the fact that $K_{\Mg{3}[A_3]}+\frac{7}{10}\delta$ is nef, with degree zero precisely on the locus of tacnodal curves is immediate from the fact that, since $\phi^+$ is small, $(\phi^+)^*(K_{\Mg{3}[A_3^*]}+\frac{7}{10}\delta)=K_{\Mg{3}[A_3]}+\frac{7}{10}\delta$. To obtain the nefness of $K_{\Mg{3}[A_3]}+\frac{17}{28}\delta$, we write
$$
K_{\Mg{3}[A_3]}+\frac{17}{28}\delta=\frac{5}{42}\left(K_{\Mg{3}[A_3]}+\frac{7}{10}\delta\right)+\frac{13}{52}H,
$$
where $H\equiv 9\lambda-\delta_{0}$ is the hyperelliptic divisor. It follows that if $K_{\Mg{3}[A_3]}+\frac{17}{28}\delta$ has negative degree on any curve, that curve must lie in the hyperelliptic locus. On the other hand, using the natural map $\M_{0,8} \rightarrow H \subset \M_{3}[A_3]$ one can verify by a direct calculation that 
$K_{\Mg{3}[A_3]}+\frac{17}{28}\delta$ restricts to a trivial divisor on $H$ \cite[Proposition 19]{HL}. 

Finally, since $K_{\Mg{3}[A_3]}+\frac{17}{28}\delta$ is trivial on the hyperelliptic divisor, applying the Kawamata basepoint freeness theorem \ref{T:kawamata} to $K_{\Mg{3}[A_3]}+\frac{17}{28}\delta$ yields a map $\M_{3}[A_3] \rightarrow \M_{3}(\frac{17}{28})$, contracting the hyperelliptic divisor to a point. Since this is precisely the locus contracted by $\M_{3}[A_3] \rightarrow \M_{3}[Q]$, we conclude that $\M_{3}[Q] \simeq \M_{3}(\frac{17}{28})$. To complete the proof of the proposition, we make the following two observations: Since $\M_{3}[Q]$ has Picard number one by its GIT construction, 
$\N^{1}(\M_{3}[Q])$ is generated by $\delta$. Second, since $\psi$ contracts the hyperelliptic divisor, we must have $9\lambda - \delta= 0 \in \N^{1}(\M_{3}[Q])$. It follows that
$$
K_{\Mg{3}[Q]}+\alpha\delta \sim (\alpha-5/9)\delta_{0},
$$
so that $K_{\Mg{3}[Q]}+\alpha\delta$ is ample iff $\alpha>\frac{5}{9}$ and becomes trivial at 
$\alpha=\frac{5}{9}$.
\end{proof}

\begin{figure}[t]\label{F:FullM3}
\includegraphics{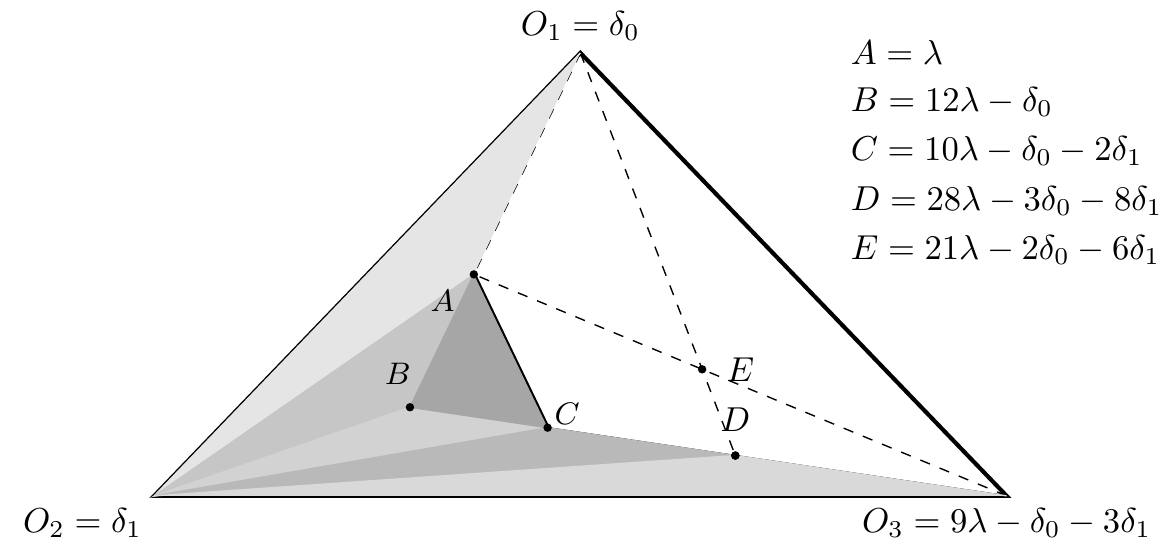}
\caption{Effective cone of $\M_3$}
\end{figure}

The preceding proposition shows that the modular birational models of $\M_{3}$ give
a complete Mori chamber decomposition of the restricted effective cone. Since we actually know the full 
effective cone of $\M_{3}$, it is interesting to see what parts of the effective cone are not accounted for by 
these birational models. In Figure \ref{F:FullM3}, we have drawn a cross section of the effective cone, 
spanned by $\delta_{0}$, $\delta_1$, and $9\lambda-\delta_{0}-3\delta_1$. 
By \cite[Section 2.4]{rulla}, the nef cone is ABC and the moving cone is $ABDE$.
The chambers $O_1AO_2$, $O_2BC$, $O_2CD$ and $O_2DO_3$ correspond to $\tau(\M_3)$, $\M_3[A_2]$, $\M_3[A_3]$ and $\M_3[Q]$ respectively. However, we do not know much  about the birational models accounting for the remaining portion of the effective cone: what are the 
small modifications of $\M_3$ coming from the divisors in $ACDE$ and do they have a
modular meaning? In particular, the question of whether $\M_{3}$ is a Mori dream space remains open.

\subsubsection{Mori chamber decomposition for $\M_{g}$.}\label{S:MMPMg}
Hassett and Hyeon have proved that the first two steps of the log MMP for $\M_{g}$ proceed exactly as in the case $g=3$, i.e. we have
\begin{theorem}[\text{\cite{HH1} and \cite{HH2}}]
 $$
 \M_{g}(\alpha)=
\begin{cases}
\M_{g} 		&\text{ iff } \alpha \in (9/11, 1],\\
\M_{g}[A_2] 	& \text{ iff } \alpha \in (7/10,9/11],\\
\M_{g}[A_3^*] 	&\text{ iff } \alpha=7/10,\\
\M_{g}[A_3] 	&\text{ iff } \alpha=(7/10-\epsilon, 7/10].\\
\end{cases}
$$
\end{theorem}
\begin{proof}
The first steps of the proof proceed exactly as in the genus three case. By Proposition \ref{P:first-F-curves},
 $K_{\Mg{g}}+\frac{9}{11}\delta=11\lambda-\delta$ is nef and has degree zero precisely on the F-curve of elliptic tails. But this is precisely the curve class contracted by $\eta\co \M_{g} \rightarrow \M_{g}[A_2]$, so we conclude $\M_{g}(\frac{9}{11})=\M_{g}[A_2]$. Exactly as in the proof of Proposition \ref{P:M3}, one shows that $K_{\SM_{g}}+\alpha\delta$ remains ample for $\alpha \in (\frac{7}{10}, \frac{9}{11})$, and that at $\alpha=\frac{7}{10}$, we have 
$$
\eta^*(K_{\SM_{g}[A_2]}+\frac{7}{10}\delta)=10\lambda-\delta-\delta_1.
$$
Applying the Kawamata basepoint freeness theorem \ref{T:kawamata} to $K_{\M_{g}[A_2]}+\frac{7}{10}\delta$ gives a birational morphism $\M_{g}[A_2] \rightarrow \M_{g}(\frac{7}{10})$, and Proposition \ref{P:first-F-curves} implies that the curves contracted by this morphism are precisely the curves in which every moving component is an elliptic bridge.

If we knew that the map $\M_{g}[A_2] \rightarrow \M_{g}[A_3^*]$ contracted the same curves, we would immediately conclude that  $\M_{g}(\frac{7}{10})=\M_{g}[A_3^*]$. But here, the case $g\gg 0$ is more complicated than the case $g=3$. While it is clear from Proposition \ref{P:birational-maps} that $\Exc(\phi^-)$ is the locus of the elliptic bridges and that any curve in which every moving component is an elliptic bridge is contracted by $\phi^-$, it is not clear \emph{a priori} that every curve contracted by $\phi^-$ is simply a curve in which elliptic bridges are varying in moduli. The problem is that the geometry of the exceptional locus $\Exc(\phi^-)$ is extremely complicated and does not admit any simple description as in the genus three case. While it should, in principle, be possible to precisely characterize the contracted curves of $\phi^-$ by refining the analysis in Proposition \ref{P:birational-maps}, we can avoid this altogether by taking advantage of the natural polarization coming from the GIT construction of $\M_{g}[A_3^*]$.

Since $\M_{g}[A_3^*]$ was constructed as the Chow quotient of bicanonically embedded curves, 
the Chow analogue of Corollary \ref{C:GIT-ample} implies that 
$\M_{g}[A_3^*]$ possesses an ample line bundle (cf. Equation \eqref{E:git-chow-polarizations})
numerically equivalent to
$$
\lim_{m \rightarrow \infty} \Lambda_{2,m} = 10\lambda-\delta=K_{\Mg{g}[A_3^*]}+\frac{7}{10}\delta.
$$
Furthermore, 
$$
R(\Mg{g}[A_3^*], K_{\SM_{g}[A_3^*]}+\frac{7}{10}\delta)=R(\Mg{g}[A_2], K_{\SM_{g}[A_2]}+\frac{7}{10}\delta)=R(\Mg{g}, K_{\SM_{g}}+\frac{7}{10}\delta)
$$
The first equality holds because $\Mg{g}[A_3^*]$ is isomorphic to $\Mg{g}[A_2]$ in codimension one (Proposition \ref{P:birational-maps}), and the second by a simple discrepancy calculation (see Equation\eqref{E:discrepancy}). It follows immediately $\M_{g}(\frac{7}{10})=\M_{g}[A_3^*]$.

Finally, to show that $\M_{g}(\alpha)=\M_{g}[A_3]$ for $\alpha \in (\frac{7}{10}-\epsilon, \frac{7}{10})$, it suffices to prove that $K_{\Mg{g}[A_3]}+(\frac{7}{10}-\epsilon)\delta$ is ample on $\M_{g}[A_3]$ for some $\epsilon > 0$. Indeed, since $\phi^+$ is an isomorphism in codimension one, we have 
$(\phi^{+})^*(K_{\Mg{g}[A_3^*]}+\frac{7}{10}\delta)=K_{\Mg{g}[A_3]}+\frac{7}{10}\delta$ is nef. 
Furthermore, for any $\alpha \leq 7/10$
$$
R(\Mg{g}[A_3^*], K_{\SM_{g}[A_3^*]}+\alpha\delta)=R(\SM_{g}[A_2], K_{\SM_{g}[A_2]}+\alpha\delta)=R(\Mg{g}, K_{\SM_{g}}+\alpha\delta)
$$
for precisely the same reasons as above. To obtain the requisite ample divisor, we simply use the construction of $\M_{g}[A_3]$ as the GIT quotient of the asymptotically linearized Hilbert scheme 
of bicanonically embedded curves. Corollary \ref{C:GIT-ample} implies that the line bundle $\Lambda_{2,m}$ is ample for sufficiently large $m$, and Equation \eqref{E:git-hilbert-polarizations} shows that for large $m$, 
$\Lambda_{2,m}=(10-\epsilon)\lambda-\delta \sim K_{\Mg{g}[A_3]}+(\frac{7}{10}-\epsilon)\delta$, as desired.
\end{proof}

\subsection{Log minimal model program for $\M_{0,n}$}\label{S:MMPM0n}

In this section, we turn our attention to $\M_{0,n}$, which has a much different flavor that $\M_{g}$. As noted in Section \ref{S:nef-cone}, the restricted effective cone is spanned by $-\frac{n-1}{\lfloor n/2\rfloor (n-\lfloor n/2\rfloor)}\psi+\Delta$ and 
$\frac{n-1}{2(n-2)}\psi-\Delta$. As in the case of $\M_{g}$, the log canonical divisors
$$K_{\M_{0,n}}+\alpha \Delta=\psi-(2-\alpha)\Delta$$ fill out the first quadrant of the restricted effective cone (see Figure
\ref{F:M0n-chamber}). But whereas in the case $\M_{g}$, the Mori chamber decomposition of the second quadrant was easily disposed of using the Torelli morphism while the Mori chamber decomposition of the first quadrant is still largely unknown, the case of $\M_{0,n}$ is exactly the opposite. While we have no understanding of the chamber decomposition of the second quadrant, we can give a complete description of the log MMP for $\M_{0,n}$. As usual, we set
\begin{align*}
\M_{0,n}(\alpha)= \Proj \bigoplus_{m \geq 0} \HH^0(\M_{0,n}, m (K_{\M_{0,n}}+\alpha\Delta)).
\end{align*}
The following result, pictured in Figure \ref{F:M0n-chamber} and proved in \cite{AS, FedSmyth}, gives a complete description of the models $\M_{0,n}(\alpha)$.
\begin{theorem}\label{T:logMMP-M0n}
\hfill
\begin{enumerate}
\item If $\alpha \in \Q \cap (\frac{2}{k+2},\frac{2}{k+1}]$ for some $k=1, \ldots, \lfloor \frac{n-1}{2} \rfloor$, then $\M_{0,n}(\alpha) \simeq \M_{0,\A}$,
the moduli space of $\A$-stable curves, with $\A=\bigl(\underbrace{1/k, \ldots, 1/k}_{n}\bigr)$.
\item  If $\alpha \in \Q \cap(\frac{2}{n-1},\frac{2}{ \lfloor n/2 \rfloor+1}]$, then $\M_{0,n}(\alpha)=(\PP^{1})^n \gitq \SL(2)$.
\end{enumerate}
\end{theorem}
Here, $(\PP^{1})^n \gitq \SL(2)$ is the classical GIT quotient of $n$ points on $\PP^1$, taken with the symmetric linearization \cite[Chapter 3]{GIT}; in the case of odd $n$, it is still of the form $\M_{0,\A}$. 
 \begin{figure}[thb]
\scalebox{0.9}{\includegraphics{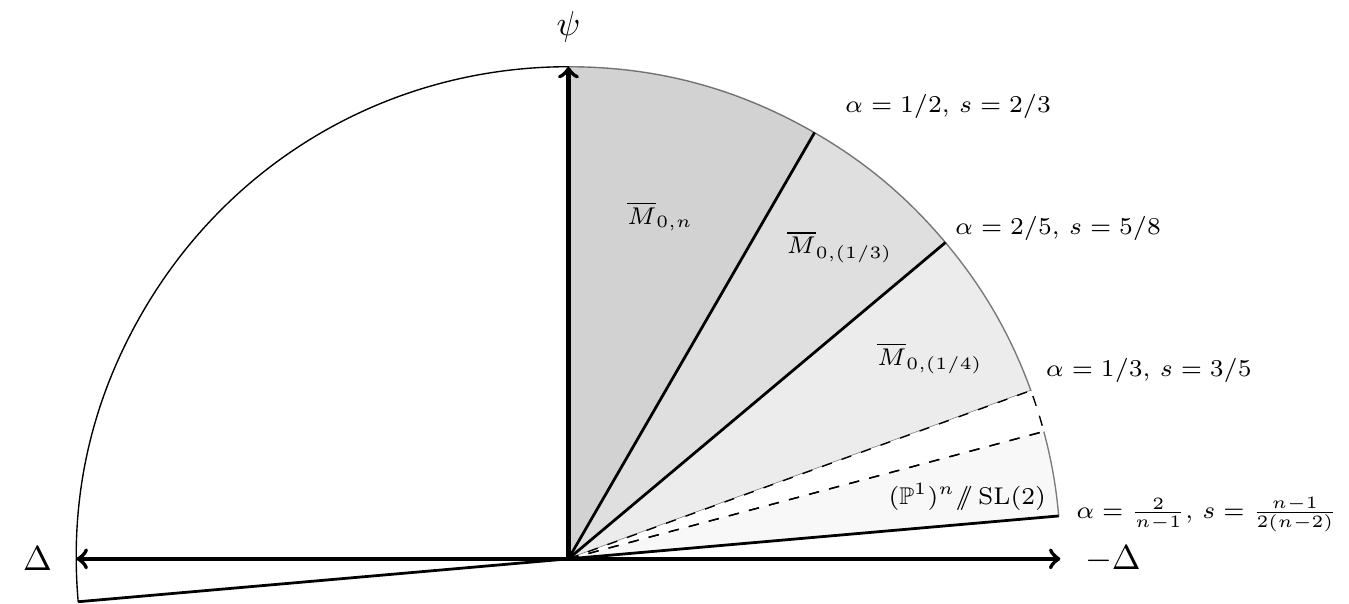}}
\caption{Mori chamber decomposition for $\M_{0,n}$.}
\label{F:M0n-chamber}
\end{figure}

\begin{proof}[Proof of Theorem \ref{T:logMMP-M0n}:]
The proof of (2) is elementary, and given the results of Section \ref{S:kollar}, the proof of (1) is fairly quick as well. Fixing $\A=\left( 1/k, \dots, 1/k\right)$ and $\alpha \in \Q \cap (\frac{2}{k+2},\frac{2}{k+1}]$, we must show that:
\begin{enumerate}
\item $(K_{\M_{0,n}}+\alpha\Delta)-\phi^*\phi_*(K_{\M_{0,n}}+\alpha\Delta)\geq 0$,
\item $\phi_*(K_{\M_{0,n}}+\alpha\Delta)$ is ample on $\M_{0,\A}$,
\end{enumerate}
where $\phi\co \M_{0,n}\ra \M_{0,\A}$ is the natural reduction morphism.
For (1), since the exceptional divisors of $\phi$ are the boundary divisors 
$\Delta_{m}$ with $3\leq m\leq k$ (see discussion preceding Corollary \ref{C:nef-divisors-weighted}), we have 
\[
(K_{\M_{0,n}}+\alpha\Delta)-\phi^*\phi_*(K_{\M_{0,n}}+\alpha\Delta)=\sum_{3\leq m\leq k} a_i\Delta_i,
\]
\noindent
for some undetermined coefficients $a_i$. In Exercise \ref{E:test-curve-m} below, we construct a curve $T_m \subset \M_{0,n}$ which is contracted by $\phi$, has the intersection numbers $T_m\cdot \psi=m(m-2)$, $T_m\cdot \Delta_2=
\binom{m}{2}$, $T_m\cdot \Delta_m=-1$, and avoids all other boundary divisors. 
Intersecting $T_m$ with the above equation gives
\[
a_m=-(K_{\M_{0,n}}+\alpha\delta)\cdot T_m=(m-2)(1-\frac{m+1}{2}\alpha)\geq 0,
\]
since $m \leq k$ and $\alpha \in (\frac{2}{k+2},\frac{2}{k+1}]$. It remains to show that $\phi_*(K_{\M_{0,n}}+\alpha\Delta)$ is ample on 
$\M_{0,\A}$. 

Here, the key point is that the divisor $\phi_*(K_{\M_{0,n}}+\alpha\Delta)$ lies in the convex hull of the nef divisors constructed in  Corollary \ref{C:nef-divisors-weighted}. Defining $\Delta_{i,j}$ and $\Delta_{nodal}$ as in the discussion preceding Corollary \ref{C:nef-divisors-weighted}, we may rewrite 
$$\phi_*(K_{\M_{0,n}}+\alpha\Delta)\sim c\psi+(2c-1)\sum_{i,j}\Delta_{i,j}-\Delta_{nodal},$$ 
where $c=1/(2-\alpha)\in \left(\frac{k+2}{2(k+1)}, \frac{k+1}{2k}\right].$ 
Setting $D(c):=c\psi+(2c-1)\sum_{i,j}\Delta_{i,j}-\Delta_{nodal}$, one easily checks that, in the notation of Corollary \ref{C:nef-divisors-weighted}, we have
\begin{align*}
D\left(\frac{k+1}{2k}\right) &=B \left(\frac{1}{k},\dots, \frac{1}{k}\right)+\frac{k-2}{2k}C \left(\frac{1}{k},\dots, \frac{1}{k}\right), \\
D\left(\frac{k+2}{2(k+1)}\right) &=B \left(\frac{1}{k+1},\dots, \frac{1}{k+1}\right)+\frac{k-1}{2(k+1)}C \left(\frac{1}{k+1},\dots, \frac{1}{k+1}\right).
\end{align*}
Now Corollary \ref{C:nef-divisors-weighted} implies that $D(c)\sim K_{\M_{0,\A}}+\alpha \Delta$ is nef and has positive degree on every complete curve. To see that $K_{\M_{0,\A}}+\alpha \Delta$ is actually ample (not merely nef), simply observe that $(\M_{0,\A}, \alpha\Delta)$
is a big and nef klt pair, so the Kawamata basepoint freeness theorem \ref{T:kawamata} says that $K_{\M_{0,\A}}+\alpha \Delta$ is semiample. Since $K_{\M_{0,\A}}+\alpha \Delta$ has
positive degree on every curve, it is ample.
\end{proof}

\begin{exercise}\label{E:test-curve-m} Let $m\geq 3$.
Construct a curve $T_m \subset \M_{0,n}$ contracted by the morphism $\M_{0,n} \ra
\M_{0,(1/k,\dots, 1/k)}$ for all $k\geq m$ as follows: Consider $\PP^2$ blown up at a point and marked by $m$ lines of self-intersection $1$. A curve
 $T_m\subset \M_{0,n}$ is obtained by attaching a constant family of $(n-m)$\nb-pointed curves along the 
$(-1)$\nb-curve. Compute that 
\begin{align}
T_m\cdot \psi=m(m-2), & & T_m\cdot \Delta_2=
\binom{m}{2}, & & T_m\cdot \Delta_m=-1.
\end{align}
\end{exercise}

\subsubsection{The contraction associated to $\psi$}
In this subsection, we make a small digression to discuss the geometry of the contraction associated to $\psi$. As seen in Figure \ref{F:M0n-chamber}, $\psi$ lies on the left-most side of the restricted nef cone of $\M_{0,n}$. Here, we will show that $\psi$ is semiample and defines a morphism $\M_{0,n} \rightarrow \M_{0,n}[\psi]$, where $\M_{0,n}[\psi]$ is the moduli space of $\psi$-stable curves constructed in Corollary \ref{C:psi-stack}, obtaining the projectivity of $\M_{0,n}[\psi]$ as a byproduct.

By Remark \ref{R:psi-contraction}, there is a morphism $\psi\co \M_{0,n}\ra \M_{0,n}[\psi]$ contracting the locus of curves possessing an unmarked
component with at least $4$ nodes. In the following lemma, we will show that this morphism is actually projective.
\begin{lemma}
$\psi$ is semiample and descends to an ample line bundle on $\M_{0,n}[\psi]$.
\end{lemma}
\begin{proof}
First, note that the morphism $\psi\co \M_{0,n} \ra \M_{0,n}[\psi]$ contracts precisely those families curves on which 
the line bundle 
$\psi$ has degree $0$, namely those families in which every moving component is unmarked. Moreover, $\psi=\psi^*\L$, where $\L$ is the analogue of the $\psi$ class for $\M_{0,n}[\psi]$, i.e. the line bundle $\otimes_{i=1}^{n}\sigma_i^*\O_{\C}(-\sigma_i)$, where $\sigman$ are the sections of the universal curve $\C \rightarrow \M_{0,n}[\psi]$. Finally, $\psi$ is semiample on $\M_{0,n}$ by Exercise \ref{E:psi-exercise} below. Since the semiample line bundle $\psi$ has degree zero on the same curves contracted by 
the morphism $\psi\co \M_{0,n}\ra
\M_{0,n}[\psi]$ of Remark \ref{R:psi-contraction}, we conclude that $\M_{0,n}[\psi]$ is projective
with normalization equal to $\proj R(\M_{0,n}, \psi)$.
\end{proof}

\begin{exercise}\label{E:psi-exercise}
Let $\A=(1,1/(n-2), \dots, 1/(n-2))$. 
Show that 
\begin{enumerate}
\item $\M_{0,\A}\simeq \PP^{n-3}$ and $\psi_1\simeq \O_{\PP^{n-3}}(1)$.
\item Show that if $\xi\co \M_{0,n} \ra \M_{0,\A}$ is the natural morphism, then 
$\psi_1=\xi^*(\psi_1)$. 
\end{enumerate}
Conclude that $\psi_i$ is semiample for each $i=1,\dots, n$, hence that $\psi$ is semiample.
\end{exercise}
\begin{remark} The morphism $\xi\co \M_{0,n} \ra \M_{0,\A}\simeq \PP^{n-3}$ of Exercise \ref{E:psi-exercise}
factors into a sequence of blow-ups with smooth centers and constitutes the Kapranov's construction of $\M_{0,n}$; see \cite{kap1, kap2} and \cite[Section 6.1]{Hweights}.
\end{remark}
Note that since $\M_{0,n}[\psi]$ is not $\Q$-factorial, the chamber corresponding to the contraction defined by $\psi$ is not full dimensional; within the restricted effective cone of $\M_{0,n}$, it is simply the single extremal ray spanned by $\psi$. However, if $\M_{0,n}$ is a Mori dream space, it should be possible to flip this contraction to obtain a $\Q$-factorial birational model whose chamber lies on the opposite side of this ray. This motivates the question whether $\psi$ can be flipped.
\begin{question}  What are the normal, projective, $\QQ$\nb-factorial varieties 
$X$ admitting a morphism $\psi^{+}\co X \ra \M_{0,n}[\psi]$ such that 
$\M_{0,n} \dashrightarrow X$ is an isomorphism in codimension one and the line bundle $\psi$ is nef on 
$X$? 
\end{question}

\subsection{Log minimal model program for $\M_{1,n}$}\label{S:MMPM1n}
In this section, we present a complete Mori chamber decomposition for a restricted effective cone of $\M_{1,n}$. In contrast to the case of $\M_{g}$ and $\M_{0,n}$, $\{K_{\SM_{1,n}}+\alpha\delta\,: \,\alpha \geq 0\}$ is not the most natural ray of the restricted effective cone to consider, so we will adopt slightly different notation for the log canonical models of $\M_{1,n}$. Set
\begin{equation}
\begin{aligned}
&D(s,t):=s\lambda+t\psi-\delta, \\
&R(s,t):=\oplus_{m \geq 0}\HH^0(\M_{1,n}, mD(s,t)),\\
&\M_{1,n}^{s,t}:=\Proj R(s,t).
\end{aligned}
\end{equation}

Evidently, any divisor in the half-space $\{s\lambda+t\psi+u\delta \,:\, u \leq 0 \}$ is numerically proportional to a divisor of the form $D(s,t)$. Since $K_{\SM_{1,n}}+\alpha\delta=13\lambda+\psi - (2-\alpha)\delta$, we have $\M_{1,n}(\alpha)=\M_{1,n}\left(\frac{13}{2-\alpha}, \frac{1}{2-\alpha}\right),$ and the birational models of the Hassett-Keel log minimal model program are a special class of the models that we study. As we shall see, however, $s$ and $t$ are the most convenient parameters for describing the thresholds in the chamber decomposition.

The half-space $\Eff_\rst(\M_{1,n}) \cap \{s\lambda+t\psi+u\delta \,:\, u \leq 0 \}$ is determined in Part (1) of the proposition, while the Mori chamber decomposition of this half-space is determined in Part (2).
\begin{proposition}\label{P:elliptic}
\begin{enumerate}
\item[]
\item $D(s,t)$ is effective iff $s+nt -12 \geq 0$ and $t \geq 1/2$.
\item $\M_{1,n}^{s,t}=\M_{1,\A_{n}^{k}}(m)$ for all $(s,t) \in B_{k,m}$,
where the polytope $B_{k,m} \subset \Q\{s,t\}$ is defined by:
\end{enumerate}
\end{proposition}
\begin{tabular}{| l | c | c | c |} \hline
				&$k=1$ & $k=1, \ldots, n-1$ 				& $k=n$			\\ \hline
				&&&\\
$m=0$ 			& $(11, \infty) \times (\frac{3}{4}, \infty)$ & $(11, \infty) \times (\frac{k+2}{2k+2}, \frac{k+1}{2k}]$ &$(11, \infty) \times (\frac{1}{2}, \frac{n+1}{2n})$\\
				&&&\\ \hline
				&&&\\
$m=1$			&$(10, 11] \times (\frac{3}{4}, \infty)$ & $(10, 11] \times (\frac{k+2}{2k+2}, \frac{k+1}{2k}]$ & $(10, 11] \times (\frac{1}{2}, \frac{n+1}{2n})$\\
				&&&\\ \hline
				&&&\\
$m \geq 2$	& \parbox{1in}{\centering $(11-m, 12-m)$ \par $\times$ \par $(\frac{3}{4}, \infty)$} & 
\parbox{1in}{\centering $(11-m, 12-m)$ \par $\times$ \par
 $\bigl(\frac{k+2}{2k+2}, \frac{k+1}{2k}\bigr]$}   &   \parbox{1in}{\centering $(11-m, 12-m)$ \par $\times$ \par $(\frac{1}{2}, \frac{n+1}{2n})$}  \\
				&&&\\ \hline
\end{tabular}\\

We will not discuss the proof of this result, except to note that the strategy is, as usual, to prove that $s\lambda+t\psi-\delta$ becomes ample on the appropriate model $\M_{1,\A}[m]$ by showing that it has positive intersection on every curve. Methods for proving the positivity of such divisor classes on $\M_{1,n}[m]$ are developed in \cite{SmythEII} -- in which this Proposition is proved in the special case $t=1$. These can be combined with the methods of Section \ref{S:kollar} for proving positivity of certain divisor classes on $\M_{1,\A}$ to obtain the result.

Our main purpose in stating this result is to indicate how the results of the log MMP for $\M_{g}$ (in which successive singularities are introduced into the moduli functors $\SM_{g}(\alpha)$) and the log MMP for $\M_{0,n}$ (in which successively more marked points are allowed to collide) may be expected to blend when considering the Mori chamber decomposition for $\M_{g,n}$. In the case of $\M_{1,n}$, we see that scaling $\lambda$ relative to $\psi-\delta$ has the effect of introducing singularities (at the same slope thresholds as in the log MMP for $\M_{g}$), while scaling $\psi$ relative to $\psi-\delta$ has the effect of allowing marked points to collide (at the same slope thresholds as in the log MMP for $\M_{0,n}$.) More generally, we expect to the same pattern to hold for the Mori chamber decomposition of $\M_{g,n}$.

\subsection{Heuristics and predictions}\label{S:heuristic}
In this section, we will discuss a heuristic method for predicting which singularities should arise in future stages of the log MMP for $\M_{g}$, as well as the critical $\alpha$-values at which these various singularities appear. As we shall see, these predictions are connected with a number of fascinating problems around the geometry and deformation theory of curves, many of which are worth of exploration in their own right. 

The heuristic operates as follows: Let $X$ be an irreducible proper 
curve of arithmetic genus $g \geq 2$ with a single isolated singularity $p\in X$. Since $X$ has a 
finite automorphism group (cf. proof of Theorem \ref{T:Construction}), there is 
a universal deformation space $\Def(X)$. Under very mild assumptions on the singularity, $\Def(X)$ 
is irreducible of dimension $3g-3$ and the generic point of the versal family corresponds to a smooth curve of genus $g$. Thus, we obtain a rational map $\Def(X) \dashrightarrow \Mg{g}$. 
Now we define $\Tl{X,p}$, the \emph{space of stable limits of $X$}, as follows: Consider a birational resolution
\[
\xymatrix{
&W \ar[rd]^{q} \ar[ld]_{p}&\\
\Def(X) \ar@{-->}[rr]&&\Mg{g}
}
\]
and let $\Tl{X,p}:=q(p^{-1}(0))$. Intuitively, $\Tl{X,p}$ is simply 
the locus of stable curves appearing as stable limits of smoothings of $X$. Note that if 
$m(p)$ is the number of branches of $p\in X$ and $\delta(p)$ is the $\delta$\nb-invariant,
then the curves lying in $\Tl{X,p}$ have form $\tilde{X} \cup T$, 
where $(\tilde{X}, q_1, \dots, q_{m(p)})$ is the pointed normalization of $X$ and 
$(T,p_1,\dots, p_{m(p)})$ is an $m(p)$-pointed curve attached to $\tilde{X}$ by identifying $p_i$ with $q_i$. Clearly,
$p_a(T)=\delta(p)-m(p)+1$.
(We call $(T, p_1,\dots, p_{m(p)})$ the tail of the stable limit.)
We set
$$
\alpha(X):=\sup_{\, \alpha \in \Q}\{\Tl{X,p} \text{ is covered by $(K_{\Mg{g}}+\alpha\delta)$-negative curves}\, \}.
$$
\begin{remark} Even though the set of tails of stable limits of a given curve $(X,p)$ depends only on $\hat{\O}_{X,p}$, the variety $\Tl{X,p}$ may depend on the global geometry of $X$. See Remark \ref{R:dangling} below.
\end{remark}
We can now state our main heuristic principle:
\begin{principle}\label{principle}
If $\M_{g}(\alpha)$ is weakly modular, then the curve $X$ appears in the corresponding moduli functor $\SM_{g}(\alpha)$ at $\alpha = \alpha(X)$.
\end{principle}
The intuition behind this heuristic is that once $X$ appears in the moduli functor, the unique limit property dictates that all stable limits associated to $X$ should \emph{not} appear. 
Hence $\Tl{X,p}$ should be in the exceptional locus of the map
$$|m(K_{\Mg{g}}+\alpha\delta)|\co \M_{g} \dashrightarrow \M_{g}(\alpha), \quad m\gg 0.$$
Under mild assumptions, this implies that $\Tl{X,p}$ is covered by curves which $K_{\Mg{g}}+\alpha\delta$ intersects non-positively.

Assuming this heuristic is reliable, we may predict when a given singularity will appear, provided that we know:
\begin{enumerate}
\item The variety $\Tl{X,p}$ of stable limits.
\item The extremal curve classes on $\Tl{X,p}$, i.e. covering families of maximal slope, where the slope of a curve class is defined by $s(C)=\delta\cdot C/\lambda\cdot C$.
\end{enumerate}
Note the usual conversion between slope and $\alpha$. If a family of curves covers $\Tl{X,p}$ with slope $s$, then $\Tl{X,p}$ must lie in the stable base locus of $K_{\Mg{g}}+\alpha\delta$ for $\alpha \leq 2-13/s$.

In what follows, we will apply Principle \ref{principle} to make three kinds of predictions. First, we will describe which rational $m$-fold points (Exercise \ref{E:rationalmfold}) and elliptic $m$-fold points (Definition \ref{D:ellipticmfold}) should appear in the log MMP. This calculation raises several interesting questions regarding the qualitative nature of the stacks $\SM_{g}(\alpha)$ which should appear in the course of this program. Second, we will give a formula for the expected $\alpha$-value of an arbitrary quasitoric singularity $x^{p}=y^{q}$. In particular, we will see that we expect only two 
infinite families of singularities, namely $A$ and $D$ singularities, to appear before the hyperelliptic threshold 
$\alpha=\frac{3g+8}{8g+4}$, as well as a finite number of others, including $E_6, E_7$
and $E_8$.
Finally, we will explain why certain non-reduced schemes are expected to arise in 
$\Mg{g}(\frac{3g+8}{8g+4})$.

\subsubsection{Rational $m$-fold points and elliptic $m$-fold points}
As a warm-up application of our heuristic, let us show that, for $m \geq 3$, the rational $m$-fold points do not appear in the moduli problems $\SM_{g}(\alpha)$ for any $\alpha$. 

\begin{lemma}\label{L:rationalmfold}
Let $X$ be an irreducible curve with a single rational $m$-fold point $p \in X$. Then the variety of stable limits of $X$ is:
$$
i\co \Tl{X,p} \simeq \M_{0,m} \hookrightarrow \M_{g},
$$
where the map $i$ is defined by mapping a curve $(C, \pm)$ to the point $\tilde{X} \cup (C, \pm)$, 
where the points $\pm$ are attached to $\tilde{X}$ at the $m$ points lying above $p \in X$.
\end{lemma}
\begin{proof}
It is clear, by genus considerations, that any stable limit of $X$ must be of the form $\tilde{X} \cup (C, \pm)$ with $(C, \pm) \in \M_{0,m}$. The fact that \emph{any} $(C, \pm) \in \M_{0,m}$ does arise as a stable limit is an easy consequence of Exercise \ref{E:rationalmfold}.
 Indeed, given any curve of the form  $\tilde{X}
 \cup (C, \pm)$, we may consider a smoothing $\C \rightarrow \Delta$, and the Contraction 
 Lemma \ref{L:Contraction} in conjunction with the exercise 
 shows that contracting $C$ produces a special fiber isomorphic to $X$. 
 Thus, $\tilde{X} \cup (C, \pm) \in \Tl{X,p}$. 
\end{proof}
To analyze the associated invariant $\alpha(X)$, we observe that
\begin{align*}
i^*\lambda& =0,\\
i^*(-\delta) &=\psi-\delta,
\end{align*}
and that $\psi-\delta$ is ample on $\M_{0,m}$. It follows that 
$K_{\Mg{g}}+\alpha\delta \sim \frac{13}{2-\alpha}\lambda-\delta$ has positive degree on every curve in $\Tl{X,p}$ for all values of $\alpha$! Hence, $X$ should never appear in the moduli functors $\SM_{g}(\alpha)$.

For our next example, let us compute the $\alpha$-invariants associated to elliptic $m$-fold points. To begin, we have the following analogue to Lemma \ref{L:rationalmfold}.

\begin{lemma}\label{L:ellipticmfold}
Let $X$ be an irreducible curve with a single elliptic $m$-fold point $p \in X$. Then the variety of stable limits of $X$ is:
$$
i\co \Tl{X,p} \simeq \M_{1,m} \hookrightarrow \M_{g},
$$
where the map $i$ is defined by mapping a curve $(C, \pm)$ to the point $\tilde{X} \cup (C, \pm)$, where the points $\pm$ are attached to $\tilde{X}$ at the $m$ points lying above $p \in X$.
\end{lemma}
\begin{proof}
It is clear, by genus considerations, that any stable limit of $X$ must be of the form $\tilde{X} \cup (C, \pm)$ with $(C, \pm) \in \M_{1,m}$, and we only need to check that every $(C, \pm)$ arises as a tail. For tacnodes, we can check this using Proposition \ref{P:BridgeContraction}. Indeed, Proposition \ref{P:BridgeContraction} shows that if $\C \rightarrow \Delta$ is any smoothing of $\tilde{X} \cup (C, \pm)$ with smooth total space, then contracting $C$ produces a special fiber with a tacnode. Essentially the same argument works for the general elliptic $m$-fold point (see \cite[Lemma 2.12]{SmythEI}).
\end{proof}

Now let us analyze the associated invariant $\alpha(X)$. We have
\begin{align*}
i^*\lambda &=\lambda,\\
i^*(-\delta) &=\psi-\delta,\\
i^*(s\lambda-\delta) &=s\lambda+\psi-\delta.
\end{align*}
Using Proposition \ref{P:elliptic}, we see that $s\lambda+\psi-\delta$ is big until $s=12-m$, so that any covering family of $\M_{1,m}$ must have slope less than $12-m$. On the other hand, it is easy to see that $\M_{1,m}$ is covered by curves of slope $12-m$. (For $m \leq 9$, one can construct these families explicitly using pencils of cubics.) Since slope $12-m$ corresponds to $\alpha=\frac{11-2m}{12-m}$, we conclude that $\M_{1,m}$ is covered by $(K_{\Mg{g}}+\alpha\delta)$-non-positive curves for the first time at $\alpha=\frac{11-2m}{12-m}$, i.e. \emph{elliptic $m$-fold points should arise in the moduli functor $\SM_{g}(\alpha)$ at  $\alpha=\frac{11-2m}{12-m}$}. Note that this is consistent with the result for cusps and tacnodes. Furthermore, we see that $\alpha>0$ iff $m \leq 5$, i.e. we only expect elliptic $m$-fold points to appear for $m \leq 5$.

It is natural to wonder whether there is some intrinsic property of these singularities which determines whether or not they appear in the moduli functors $\SM_g(\alpha)$. An interesting observation in this regard is that rational $m$-fold points are not Gorenstein (for $m \geq 3$), while elliptic $m$-fold points are. Furthermore, an elliptic $m$-fold point has unobstructed deformations iff $m \leq 5$. This raises the following questions:
\begin{question}
\begin{enumerate}
\item[]
\item Is there any reason the moduli stacks $\SM_{g}(\alpha)$ should involve only Gorenstein singularities?
\item Is there any reason that the moduli stacks $\SM_{g}(\alpha)$ should be smooth for $\alpha>0$. 
\end{enumerate}
\end{question}
\subsubsection{$A_k$ and $D_k$ singularities}
\label{S:test-families}
In this section, we apply our heuristic to compute the threshold $\alpha$-values at which $A_k$ ($y^2=x^{k+1}$) and $D_k$ ($xy^2=x^{k-1}$) singularities should appear. The results are displayed in Table 4. 
Note, in particular, that all $A_k$ and $D_k$ singularities are expected to appear before the hyperelliptic threshold $\alpha=\frac{3g+8}{8g+4}$.

The variety of stable limits associated to these singularities is described in \cite{hassett-stable}
and we recall this description below.
\begin{proposition}[Varieties $\Tl{A_k}$ and $\Tl{D_k}$]\label{P:stable-reduction-AD}
\begin{enumerate}
\item[]
\item Let $X$ be an irreducible curve of genus $g$ with a single $A_k$ singularity. 
Then the stable limits of $X$ are as follows: 
\begin{enumerate}
\item If $k$ is even, 
a curve $\tilde{X} \cup (T,p)$ with $T$ hyperelliptic of arithmetic genus $\lfloor k/2\rfloor$, $p$ 
a Weierstrass point.
\item If $k$ is odd, 
a curve $\tilde{X} \cup (T,p,q)$ with $T$ hyperelliptic of arithmetic genus $\lfloor k/2\rfloor$, $p$ and $q$ 
conjugate points.
\end{enumerate}
\item Let $X$ be an irreducible curve of genus $g$ with a single $D_k$ singularity. Then the stable limits of $X$ are as follows: 
\begin{enumerate}

\item If $k$ is odd, 
a curve $\tilde{X} \cup (T,p,q)$ with $T$ hyperelliptic of arithmetic  genus $\lfloor (k-1)/2\rfloor$, $p$ a 
Weierstrass point, $r$ a free point.

\item If $k$ is even, 
a curve $\tilde{X} \cup (T,p,q,r)$ with $T$ hyperelliptic of arithmetic  genus $\lfloor (k-1)/2\rfloor$, $p$ and $q$
conjugate points, $r$ a free point.
\end{enumerate}
\end{enumerate}
\end{proposition}

In order to understand when the varieties $\T_{A_k}$ and $\T_{D_k}$ fall into the base locus of $K_{\Mg{g}}+\alpha\delta$, we must construct covering families for these varieties of maximal slope. The key step is the following construction of hyperelliptic families.

\begin{proposition}[Families of hyperelliptic curves]
\label{P:hyperelliptic-families}
Let $k\geq 1$ be an integer. 
\begin{enumerate}
\item There exists a complete one-parameter family $T_k$ of $1$\nb-pointed curves of genus $k$ such that the generic fiber is a smooth hyperelliptic curve with a 
marked Weierstrass point. Furthermore, 
\begin{equation}
\begin{aligned}
\lambda\cdot T_k &=k^2, \\
\delta_{0}\cdot T_k &=8k^2+4k, \\
\psi\cdot T_k &=1, \\
\delta_1\cdot T_k &=\cdots =\delta_{\lfloor k/2\rfloor} \cdot T_k=0.
\end{aligned} 
\end{equation}
\item There exists a complete one-parameter family $B_k$ of $2$\nb-pointed curves of genus $k$ such that the generic fiber is a smooth hyperelliptic curve with marked points conjugate under the hyperelliptic involution. Furthermore,
\begin{equation}
\begin{aligned}
\lambda\cdot B_k &=(k^2+k)/2, \\
\delta_{0}\cdot B_k &=4k^2+6k+2, \\
\psi_1\cdot B_k&= \psi_2\cdot B_k =1, \\
\delta_1\cdot B_k &=\cdots =\delta_{\lfloor k/2\rfloor} \cdot B_k=0.
\end{aligned} 
\end{equation}

\end{enumerate}
\end{proposition}
\begin{proof}
There is an elementary way to write down a family of hyperelliptic curves marked with a Weierstrass section. 
One begins
with the Hirzebruch surface $\mathbb{F}_2$ realized as a $\PP^1$\nobreakdash-bundle
 $\mathbb{F}_2 \ra B$ over $B\simeq \PP^1$. Denote by 
$E$ the unique $(-2)$\nb-section. Next, 
choose $2k+1$ general divisors $S_1,\dots, S_{2k+1}$ in the linear system 
$\vert E+2F\vert$ (these are sections of $\mathbb{F}_2\ra B$ of self-intersection $2$).
The divisor $E+\sum_{i=1}^{2k+1} S_i$ is divisible by $2$ in $\Pic(\mathbb{F}_2)$ 
and so there is a cyclic degree $2$
branched cover $X\ra \mathbb{F}_2$ branched over $E+\sum_{i=1}^{2k+1} S_i$. Denote by $\Sigma$ the 
preimage of $E$. Then $(X,\Sigma) \ra B$ is a family of at worst nodal hyperelliptic curves of genus $k$ with a marked Weierstrass point.
It is easy to verify that the constructed family has 
the required intersection numbers.

The proof of the second part proceeds in an analogous manner.
The only modification being is that one needs to consider
the double cover of $\mathbb{F}_1$ branched
over $2k+2$ sections of self-intersection $1$. 
\end{proof}

If $g\geq k+1$, we abuse notation by using $\mathrm{T}_k$ to denote the family of 
stable curves obtained by gluing the family constructed in Part (1) 
of Proposition \ref{P:hyperelliptic-families} to a constant
family of $1$\nb-pointed genus $g-k$ curves. 
We call
the resulting family of stable genus $g$ curves the 
{\em hyperelliptic Weierstrass tails} of genus $k$. 
Using the intersection numbers in Proposition \ref{P:hyperelliptic-families}, one easily computes the slope of $T_k$ to be
 $\dfrac{\delta\cdot \mathrm{T}_k}{\lambda\cdot \mathrm{T}_k}=\dfrac{8k^2+4k-1}{k^2}.$
 
If $g\geq k+2$,
we abuse notation by using $B_k$ to denote the family of stable curves obtained by 
gluing the family of 2-pointed hyperelliptic genus $k$ curves of Proposition \ref{P:hyperelliptic-families} (2) to a constant $2$\nb-pointed curve of genus $g-k-1$. We call the resulting family of stable genus $g$ curves the
{\em hyperelliptic conjugate bridges} of genus $k$. Using Proposition \ref{P:hyperelliptic-families}, one easily computes
$\dfrac{\delta\cdot \mathrm{B}_k}{\lambda\cdot \mathrm{B}_k}=\dfrac{8k+12}{k+1}.$
\begin{remark}
By exploiting the existence of a regular birational morphism from $\Tl{A_k}$ to a variety of Picard number one, constructed in \cite[Main Theorem 1]{FedHyp}, it is easy to deduce  that the family $T_k$ (resp. $B_k$) of Proposition \ref{P:hyperelliptic-families} is a covering family of 
$\Tl{A_{2k}}$ (resp. $\Tl{A_{2k+1}}$) of the maximal slope. 
\end{remark}
This elementary construction of hyperelliptic tails and bridges can be modified 
to produce the following families (see \cite[Section 6.1]{afs}):
\begin{enumerate}
\item The {\em hyperelliptic bridges} attached at a Weierstrass and a non-Weierstrass point. These arise from stable reduction of a $D_{2k+1}$\nb-singularity. We denote this family by $\mathrm{BW}_{k}$.
\item The {\em hyperelliptic triboroughs} attached at a free point and two conjugate points. These arise from stable reduction 
of a $D_{2k+2}$\nb-singularity. We denote this family by $\mathrm{Tri}_{k}$.
\item The {\em hyperelliptic tails} attached at generically non-Weierstrass points. 
These arise from stable reduction of a curve with a dangling $A_{2k+1}$-singularity (also dubbed $A_{2k+1}^{\dag}$; see Definition \ref{D:danglingA}). We denote this family by $\textrm{H}_{k}$.
\end{enumerate}
\begin{remark}\label{R:dangling} The distinction between a curve having an $A_{2k+1}$ or an $A_{2k+1}^\dag$ singularity is seen only at the level of a complete curve. The reason that $A_{2k+1}^\dag$ singularities must be treated separately is that they give rise to a different variety of stable limits. The key point is that the normalization 
of a curve with $A_{2k+1}^\dag$ singularity has a $1$\nb-pointed $\PP^1$ which disappears after stabilization. Consequently, the stable limit is simply a
 nodal union of hyperelliptic genus $k$ curve and a genus $g-k$ curve, in contrast with Proposition \ref{P:stable-reduction-AD} (1.b).
\end{remark}
The intersection numbers of these curves can also be computed, and are displayed in Table 4. 

 \begin{table}[htb]\label{T:Families}
 \renewcommand{\arraystretch}{1.5}
\begin{tabular}{| l | c | c | c | c | c |} 
\hline

  Curve in $\N_1(\Mg{g})$   & $\lambda$ &  $\delta$ 	 & Singularity   & $\alpha$  \\  
\hline
$\textrm{T}_{k}$               &  $k^2$                         & $8k^2+4k-1$ & $A_{2k}$    & $\frac{3k^2+8k-2}{8k^2+4k-1}$  \\
\hline
$\textrm{B}_{k}$                 &   $(k+1)k/2$                &  $k(4k+6)$ & $A_{2k+1}$ &  $ \frac{3k+11}{8k+12}$\\
\hline
 $\textrm{H}_{k}\ (k\geq 2)$                   & $(k+1)k/2$      & $2(k+1)(2k+1)-1$ & $A_{2k+1}^\dag$ & $ \frac{3k^2+11k + 4}{8k^2+12k+2}$ \\
\hline
$\textrm{BW}_{k}$                           & $k^2$         & $8k^2+2k$ & $D_{2k+1}$     & $\frac{3k+4}{8k+2}$  \\
\hline
$\textrm{BH}_{k} \ (k\geq 2)$                  & $(k+1)k/2$         & $(k+1)(4k+1)$ & $D^\dag_{2k+2}$ & $\frac{3k+4}{8k+2}$ \\
\hline
$\textrm{Tri}_{k}$                         & $(k+1)k/2$         & $k(4k+5)$ & $D_{2k+2}$ &  $\frac{3k+7}{8k+10}$\\
\hline
\end{tabular}
\smallskip
\caption{Families of hyperelliptic tails}
\end{table}


\subsubsection{Toric singularities}
In this section, we observe that similar methods can be used to compute expected $\alpha$-values for toric singularities. Let $1<q<p$ be the coprime integers. 
Consider a genus $g$ curve $X$ with an isolated toric singularity $x^{pb}-y^{qb}=0$. The space of stable limits of $X$ has been described by Hassett \cite{hassett-stable}.
\begin{proposition}
\label{P:tails-description}
The variety of stable limits of $X$ are 
of the form $\tilde{X} \cup T$, 
where the tail $(T,p_1,\dots, p_b)$ is a $b$-pointed curve of genus 
$g=(pqb^2-pb-qb-b+2)/2$. Moreover, 
\begin{enumerate}
\item $K_T=(pqb-p-q-1)(p_1+\dots+p_b)$.
\item $T$ is $qb$-gonal with $g^1_{qb}$ given by $|q(p_1+\dots+p_b)|$.
\end{enumerate}
\end{proposition}

The construction of covering families 
for this locus is much more delicate than in the case of hyperelliptic tails and hyperelliptic bridges
and appears in \cite[Proposition 6.6]{afs}. The final result is
\begin{proposition}\label{P:toric-family} There is a family of tails $(T, p_1, \dots, p_b)$ inside $\Tl{X}\subset \Mg{g,b}$ whose intersection numbers are
\begin{multline*}
\begin{aligned}
\lambda &=\frac{b}{12}\left( (pqb-p-q)^2+pq(pqb^2-pb-qb+1)-1\right),\\
\delta_{0} &=pqb(pqb^2-pb-qb+1), \\
\psi &=b.
\end{aligned}
\end{multline*}
\end{proposition}
\begin{corollary}\label{C:toric-family} Suppose $p$ and $q$ are coprime. Then there is a family of tails of stable limits of $x^p=y^q$ of slope
\[
s=12\frac{pq(p-1)(q-1)-1}{(p-1)(q-1)(2pq-p-q-1)}.
\]
\end{corollary}

We expect these families to provide maximal slope covering families, in which case the expected $\alpha$-values at which the singularity $x^p=y^q$ should arise is:
\begin{align}\label{E:toric-prediction}
\alpha=\frac{(p-1)(q-1)(-2pq+13p+13q+13)-24}{12(pq(p-1)(q-1)-1)}.
\end{align}
This formula suggests that there are only finitely many infinite families of toric singularities that should
make an appearance in the intermediate log canonical models $\M_g(\alpha)$ for 
$\alpha\in [0,1]$. This raises the question: is there an intrinsic characterization of those singularities for which the above $\alpha$\nb-invariant 
is non-negative?

\subsubsection{The hyperelliptic locus and ribbons}
A fascinating question is what happens to $\SM_{g}(\alpha)$ at $\alpha=\frac{3g+8}{8g+4}$, at which point the closure of the hyperelliptic locus falls into the base locus of $K_{\SM_{g}}+\alpha\delta$. (This is an immediate consequence of the fact that the maximal slope of a covering family of the hyperelliptic locus is $8+\frac{4}{g}$, as proved in \cite{CH}.) If smooth hyperelliptic curves must be removed from the moduli functor $\SM_{g}(\alpha)$, what should replace them? One likely candidate, suggested to us by Joe Harris, is {\em ribbons}, i.e. non-reduced curves of multiplicity two whose underlying reduced curve is the rational normal curve in $\PP^{g-1}$ \cite{bayer-eisenbud}. This prediction is certainly consonant with the final step of the log MMP in genus three, where we saw hyperellitic curves replaced by a double plane conic. More generally, ribbons arise as the flat limit of canonically embedded smooth curves specializing (abstractly) to smooth hyperelliptic curves \cite{fong}. 
Using this it is easy to see that the variety of stable limits of a ribbon in 
$\PP^{g-1}$ is precisely the locus of hyperelliptic curves in $\M_{g}$, which makes the suggestion plausible. 

However, Example \ref{E:NCriterion} shows that in genus four ribbons alone 
cannot be the whole story: the canonical genus $4$ ribbon has nonsemistable $m^{th}$ Hilbert point for every 
$m\geq 3$. We can get insight into other candidates by considering the GIT quotient 
$\vert \O_{\PP^1\times\PP^1}(3,3)\vert^{\ss} \gitq (\SL(2)\times\SL(2))\rtimes \ZZ_2$, which is in fact equal 
to $\Mg{4}(29/60)$ \cite{genus4-final-model}.
A simple application of the Hilbert-Mumford 
numerical criterion (Proposition \ref{P:numerical}) shows that 
the union of a $(2,1)$-curve and a $(1,2)$-curve on 
$\PP^1\times \PP^1$ meeting at an $A_{9}$-singularity ($y^2=x^{10}$) 
and an irreducible $(3,3)$-curve with a unique $A_{8}$-singularity ($y^2=x^{9}$) are 
{\em stable}. These examples suggest that for higher
genera the singular curves which replace 
hyperelliptic curves in $\Mg{g}(\alpha)$ with $\alpha<\frac{3g+8}{8g+4}$
should be, in the order of generality: rational curves with an $A_{2g}$ singularity,
two rational normal curves meeting in an $A_{2g+1}$ singularity, 
and certain semistable ribbons. 
All these curves arise as flat limits of canonically embedded smooth curves specializing abstractly to smooth hyperelliptic curves.
We should also remark that there is no reason to preclude the possibility of non-reduced curves 
showing up prior to $\alpha=\frac{3g+8}{8g+4}$.
Needless to say, the prospect of constructing and classifying modular birational models involving non-reduced curves raises a host of questions regarding degenerations 
and deformations of non-reduced schemes, pushing far beyond the methods of this survey.

\bibliography{Bib}

\def\cprime{$'$} \def\cprime{$'$}
\begin{thebibliography}{{Che}10b}

\bibitem[AC87]{AC-picard}
Enrico Arbarello and Maurizio Cornalba.
\newblock The {P}icard groups of the moduli spaces of curves.
\newblock {\em Topology}, 26(2):153--171, 1987.

\bibitem[AC98]{AC}
Enrico Arbarello and Maurizio Cornalba.
\newblock Calculating cohomology groups of moduli spaces of curves via
  algebraic geometry.
\newblock {\em Inst. Hautes \'Etudes Sci. Publ. Math.}, (88):97--127 (1999),
  1998.

\bibitem[AFS10]{afs}
Jarod Alper, Maksym Fedorchuk, and David Smyth.
\newblock Singularities with {$\GG_m$}-action and the log minimal model program
  for {$\overline{M}_g$}, 2010.
\newblock {\tt arXiv:1010.3751v2 [math.AG]}.

\bibitem[Ale02]{alexeev-annals}
Valery Alexeev.
\newblock Complete moduli in the presence of semiabelian group action.
\newblock {\em Ann. of Math. (2)}, 155(3):611--708, 2002.

\bibitem[{Alp}08]{alper}
Jarod {Alper}.
\newblock Good moduli spaces for artin stacks.
\newblock {\tt arXiv:0810.1677 [math.AG]}, 2008.

\bibitem[Art70]{artin-formal-moduli-II}
M.~Artin.
\newblock Algebraization of formal moduli. {II}. {E}xistence of modifications.
\newblock {\em Ann. of Math. (2)}, 91:88--135, 1970.

\bibitem[Art09]{artebani}
Michela Artebani.
\newblock A compactification of {$\mathcal{M}_3$} via {$K3$} surfaces.
\newblock {\em Nagoya Math. J.}, 196:1--26, 2009.

\bibitem[AS79]{arbarello-sernesi}
Enrico Arbarello and Edoardo Sernesi.
\newblock The equation of a plane curve.
\newblock {\em Duke Math. J.}, 46(2):469--485, 1979.

\bibitem[AS08]{AS}
Valery Alexeev and David Swinarski.
\newblock Nef divisors on {$\overline{M}_{0,n}$} from {GIT}.
\newblock {\tt arXiv:0810.1677 [math.AG]}, 2008.

\bibitem[BCHM06]{BCHM}
C.~{Birkar}, P.~{Cascini}, C.~D. {Hacon}, and J.~{McKernan}.
\newblock Existence of minimal models for varieties of log general type.
\newblock {\tt arXiv:0610203 [math.AG]}, 2006.

\bibitem[BDPP04]{bdpp}
S.~{Boucksom}, {J.-P.} {Demailly}, M.~{Paun}, and T.~{Peternell}.
\newblock {The pseudo-effective cone of a compact {K\"{a}hler} manifold and
  varieties of negative Kodaira dimension}.
\newblock 2004.
\newblock {\tt arXiv:0405285 [math.AG]}.

\bibitem[BE95]{bayer-eisenbud}
Dave Bayer and David Eisenbud.
\newblock Ribbons and their canonical embeddings.
\newblock {\em Trans. Amer. Math. Soc.}, 347(3):719--756, 1995.

\bibitem[BV05]{bruno-verra}
Andrea Bruno and Alessandro Verra.
\newblock {{$\mathcal{M}_{15}$} is rationally connected}.
\newblock In {\em Projective varieties with unexpected properties}, pages
  51--65. Walter de Gruyter GmbH \& Co. KG, Berlin, 2005.

\bibitem[Cas09]{castravetM06}
Ana-Maria Castravet.
\newblock The {C}ox ring of {$\overline M_{0,6}$}.
\newblock {\em Trans. Amer. Math. Soc.}, 361(7):3851--3878, 2009.

\bibitem[Cat83]{catanese-prym}
F.~Catanese.
\newblock On the rationality of certain moduli spaces related to curves of
  genus {$4$}.
\newblock In {\em Algebraic geometry ({A}nn {A}rbor, {M}ich., 1981)}, volume
  1008 of {\em Lecture Notes in Math.}, pages 30--50. Springer, Berlin, 1983.

\bibitem[CH88]{CH}
Maurizio Cornalba and Joe Harris.
\newblock Divisor classes associated to families of stable varieties, with
  applications to the moduli space of curves.
\newblock {\em Ann. Sci. \'Ecole Norm. Sup. (4)}, 21(3):455--475, 1988.

\bibitem[{Che}10a]{dawei-quadratic}
D.~{Chen}.
\newblock {Covers of the projective line and the moduli space of quadratic
  differentials}.
\newblock {\tt arXiv:1005.3120 [math.AG]}, 2010.

\bibitem[{Che}10b]{dawei-elliptic}
D.~{Chen}.
\newblock {Square-tiled surfaces and rigid curves on moduli spaces}.
\newblock {\tt arXiv:1003.0731 [math.AG]}, 2010.

\bibitem[Che10c]{dawei-thesis2}
Dawei Chen.
\newblock Covers of elliptic curves and the moduli space of stable curves.
\newblock {\em J. Reine Angew. Math.}, 649:167--205, 2010.

\bibitem[Cle72]{clebsch-book}
A.~Clebsch.
\newblock {\em Theorie der bin\"{a}ren algebraischen {F}ormen}.
\newblock Verlag von {B}.{G}. Teubner, Leipzig, 1872.

\bibitem[CR84]{ChangRan13}
M.~C. Chang and Z.~Ran.
\newblock Unirationality of the moduli spaces of curves of genus {$11,$} {$13$}
  (and {$12$}).
\newblock {\em Invent. Math.}, 76(1):41--54, 1984.

\bibitem[CR86]{ChangRan15}
Mei-Chu Chang and Ziv Ran.
\newblock The {K}odaira dimension of the moduli space of curves of genus
  {$15$}.
\newblock {\em J. Differential Geom.}, 24(2):205--220, 1986.

\bibitem[CR91]{ChangRan16}
Mei-Chu Chang and Ziv Ran.
\newblock On the slope and {K}odaira dimension of {$\overline M_g$} for small
  {$g$}.
\newblock {\em J. Differential Geom.}, 34(1):267--274, 1991.

\bibitem[CT10]{castravet-tevelev}
Ana-Maria Castravet and Jenia Tevelev.
\newblock Hypertrees, projections, and moduli of stable rational curves.
\newblock {\tt arXiv:1004.2553 [math.AG]}, 2010.

\bibitem[CU93]{CU}
Fernando Cukierman and Douglas Ulmer.
\newblock Curves of genus ten on {$K3$} surfaces.
\newblock {\em Compositio Math.}, 89(1):81--90, 1993.

\bibitem[Dia84]{diaz}
Steven Diaz.
\newblock A bound on the dimensions of complete subvarieties of {${\mathcal
  M}_{g}$}.
\newblock {\em Duke Math. J.}, 51(2):405--408, 1984.

\bibitem[DM69]{DM}
P.~Deligne and D.~Mumford.
\newblock The irreducibility of the space of curves of given genus.
\newblock {\em Inst. Hautes \'Etudes Sci. Publ. Math.}, (36):75--109, 1969.

\bibitem[Dol08]{dolgachevR23}
Igor~V. Dolgachev.
\newblock Rationality of {$\mathcal{R}_2$} and {$\mathcal{R}_3$}.
\newblock {\em Pure Appl. Math. Q.}, 4(2, part 1):501--508, 2008.

\bibitem[Don84]{donagi-prym}
Ron Donagi.
\newblock The unirationality of {${\mathcal A}_{5}$}.
\newblock {\em Ann. of Math. (2)}, 119(2):269--307, 1984.

\bibitem[EH84]{EHLimit}
David Eisenbud and Joe Harris.
\newblock Limit linear series, the irrationality of {$M_{g}$}, and other
  applications.
\newblock {\em Bull. Amer. Math. Soc. (N.S.)}, 10(2):277--280, 1984.

\bibitem[EH87]{EHKodaira}
David Eisenbud and Joe Harris.
\newblock The {K}odaira dimension of the moduli space of curves of genus {$\geq
  23$}.
\newblock {\em Invent. Math.}, 90(2):359--387, 1987.

\bibitem[Far06]{farkas-syzygy}
Gavril Farkas.
\newblock Syzygies of curves and the effective cone of
  {$\overline{\mathcal{M}}_g$}.
\newblock {\em Duke Math. J.}, 135(1):53--98, 2006.

\bibitem[Far09a]{farkas-seattle}
Gavril Farkas.
\newblock The global geometry of the moduli space of curves.
\newblock In {\em Algebraic geometry---{S}eattle 2005. {P}art 1}, volume~80 of
  {\em Proc. Sympos. Pure Math.}, pages 125--147. Amer. Math. Soc., Providence,
  RI, 2009.

\bibitem[Far09b]{farkas-koszul}
Gavril Farkas.
\newblock Koszul divisors on moduli spaces of curves.
\newblock {\em Amer. J. Math.}, 131(3):819--867, 2009.

\bibitem[Far10]{farkas-aspects}
Gavril Farkas.
\newblock Aspects of the birational geometry of ${M}_g$.
\newblock In {\em Geometry of Riemann surfaces and their moduli spaces},
  volume~14 of {\em Surveys in Differential Geometry}, pages 57--111. 2010.

\bibitem[Fed]{genus4-final-model}
Maksym Fedorchuk.
\newblock The final log canonical model of {$\overline{M}_4,$} unpublished.

\bibitem[Fed08]{fedorchuk-thesis}
Maksym Fedorchuk.
\newblock Severi varieties and the moduli space of curves, 2008.
\newblock Ph.D. thesis, Harvard University.

\bibitem[Fed10]{FedHyp}
Maksym Fedorchuk.
\newblock Moduli spaces of hyperelliptic curves with {A} and {D} singularities.
\newblock {\tt arXiv:1007.4828 [math.AG]}, 2010.

\bibitem[Fed11]{FedAmple}
Maksym Fedorchuk.
\newblock Moduli of weighted stable curves and log canonical models of
  {$\Mg{g,n}$}.
\newblock {\em Math.Res.Lett. {\em to appear}}, 2011.

\bibitem[FG03]{FG}
Gavril Farkas and Angela Gibney.
\newblock The {M}ori cones of moduli spaces of pointed curves of small genus.
\newblock {\em Trans. Amer. Math. Soc.}, 355(3):1183--1199 (electronic), 2003.

\bibitem[FL10]{farkas-prym}
Gavril Farkas and Katharina Ludwig.
\newblock The {K}odaira dimension of the moduli space of {P}rym varieties.
\newblock {\em J. Eur. Math. Soc. (JEMS)}, 12(3):755--795, 2010.

\bibitem[Fon93]{fong}
Lung-Ying Fong.
\newblock Rational ribbons and deformation of hyperelliptic curves.
\newblock {\em J. Algebraic Geom.}, 2(2):295--307, 1993.

\bibitem[Fon09]{fontanari}
Claudio Fontanari.
\newblock Positive divisors on quotients of {$\overline M_{0,n}$} and the
  {M}ori cone of {$\overline M_{g,n}$}.
\newblock {\em J. Pure Appl. Algebra}, 213(4):454--457, 2009.

\bibitem[FP05]{farkas-popa}
Gavril Farkas and Mihnea Popa.
\newblock Effective divisors on {$\overline{\mathcal{M}}_g$}, curves on {$K3$}
  surfaces, and the slope conjecture.
\newblock {\em J. Algebraic Geom.}, 14(2):241--267, 2005.

\bibitem[FS11]{FedSmyth}
Maksym Fedorchuk and David Smyth.
\newblock Ample divisors on moduli spaces of pointed rational curves.
\newblock {\em J. Algebraic Geom.}, 2011.
\newblock PII: S 1056-3911(2011)00547-X (to appear in print).

\bibitem[FV10]{farkas-verra-jacobian}
G.~{Farkas} and A.~{Verra}.
\newblock {The classification of universal Jacobians over the moduli space of
  curves}.
\newblock {\tt arXiv:1005.5354 [math.AG]}, 2010.

\bibitem[FvdG04]{faber-complete}
Carel Faber and Gerard van~der Geer.
\newblock Complete subvarieties of moduli spaces and the {P}rym map.
\newblock {\em J. Reine Angew. Math.}, 573:117--137, 2004.

\bibitem[Gib09]{Gib}
Angela Gibney.
\newblock Numerical criteria for divisors on {$\overline M_g$} to be ample.
\newblock {\em Compos. Math.}, 145(5):1227--1248, 2009.

\bibitem[Gie82]{gieseker}
D.~Gieseker.
\newblock {\em Lectures on moduli of curves}, volume~69 of {\em Tata Institute
  of Fundamental Research Lectures on Mathematics and Physics}.
\newblock Published for the Tata Institute of Fundamental Research, Bombay,
  1982.

\bibitem[GK08]{krichever}
S.~{Grushevsky} and I.~{Krichever}.
\newblock {The universal Whitham hierarchy and the geometry of the moduli space
  of pointed Riemann surfaces}.
\newblock {\tt arXiv:0810.2139 [math.AG]}, 2008.

\bibitem[GKM02]{GKM}
Angela Gibney, Sean Keel, and Ian Morrison.
\newblock Towards the ample cone of {$\overline M_{g,n}$}.
\newblock {\em J. Amer. Math. Soc.}, 15(2):273--294 (electronic), 2002.

\bibitem[Har77]{Hartshorne}
Robin Hartshorne.
\newblock {\em Algebraic geometry}.
\newblock Springer-Verlag, New York, 1977.
\newblock Graduate Texts in Mathematics, No. 52.

\bibitem[Has00]{hassett-stable}
Brendan Hassett.
\newblock Local stable reduction of plane curve singularities.
\newblock {\em J. Reine Angew. Math.}, 520:169--194, 2000.

\bibitem[Has03]{Hweights}
Brendan Hassett.
\newblock Moduli spaces of weighted pointed stable curves.
\newblock {\em Adv. Math.}, 173(2):316--352, 2003.

\bibitem[Has05]{Hgenus2}
Brendan Hassett.
\newblock Classical and minimal models of the moduli space of curves of genus
  two.
\newblock In {\em Geometric methods in algebra and number theory}, volume 235
  of {\em Progr. Math.}, pages 169--192. Birkh\"auser Boston, Boston, MA, 2005.

\bibitem[HH08]{HH2}
Brendan Hassett and Donghoon Hyeon.
\newblock Log minimal model program for the moduli space of curves: the first
  flip.
\newblock Submitted. Available at {\tt arXiv:0806.3444 [math.AG]}, 2008.

\bibitem[HH09]{HH1}
Brendan Hassett and Donghoon Hyeon.
\newblock Log canonical models for the moduli space of curves: the first
  divisorial contraction.
\newblock {\em Trans. Amer. Math. Soc.}, 361(8):4471--4489, 2009.

\bibitem[HHL10]{HHL}
Brendan Hassett, Donghoon Hyeon, and Yongnam Lee.
\newblock Stability computation via {G}r\"obner basis.
\newblock {\em J. Korean Math. Soc.}, 47(1):41--62, 2010.

\bibitem[HK00]{hu-keel}
Yi~Hu and Se{\'a}n Keel.
\newblock Mori dream spaces and {GIT}.
\newblock {\em Michigan Math. J.}, 48:331--348, 2000.
\newblock Dedicated to William Fulton on the occasion of his 60th birthday.

\bibitem[HL10]{HL}
Donghoon Hyeon and Yongnman Lee.
\newblock Log minimal model program for the moduli space of stable curves of
  genus three.
\newblock {\em Math. Res. Lett.}, 17(4):625--636, 2010.

\bibitem[HM82]{HMKodaira}
Joe Harris and David Mumford.
\newblock On the {K}odaira dimension of the moduli space of curves.
\newblock {\em Invent. Math.}, 67(1):23--88, 1982.
\newblock With an appendix by William Fulton.

\bibitem[HM90]{harris-morrison}
J.~Harris and I.~Morrison.
\newblock Slopes of effective divisors on the moduli space of stable curves.
\newblock {\em Invent. Math.}, 99(2):321--355, 1990.

\bibitem[HM98a]{HM}
Joe Harris and Ian Morrison.
\newblock {\em Moduli of curves}, volume 187 of {\em Graduate Texts in
  Mathematics}.
\newblock Springer-Verlag, New York, 1998.

\bibitem[HM98b]{HarMor}
Joe Harris and Ian Morrison.
\newblock {\em Moduli of curves}, volume 187 of {\em Graduate Texts in
  Mathematics}.
\newblock Springer-Verlag, New York, 1998.

\bibitem[HM10]{hyeon-morrison}
Donghoon Hyeon and Ian Morrison.
\newblock Stability of tails and 4-canonical models.
\newblock {\em Math.Res.Lett.}, 17(04):721--729, 2010.

\bibitem[HT02]{brendan-yuri}
Brendan Hassett and Yuri Tschinkel.
\newblock On the effective cone of the moduli space of pointed rational curves.
\newblock In {\em Topology and geometry: commemorating {SISTAG}}, volume 314 of
  {\em Contemp. Math.}, pages 83--96. Amer. Math. Soc., Providence, RI, 2002.

\bibitem[Igu60]{Igusa}
Jun-ichi Igusa.
\newblock Arithmetic variety of moduli for genus two.
\newblock {\em Ann. of Math. (2)}, 72:612--649, 1960.

\bibitem[ILGS09]{izadi-etal}
Elham Izadi, Marco Lo~Giudice, and Gregory~Kumar Sankaran.
\newblock The moduli space of \'etale double covers of genus 5 curves is
  unirational.
\newblock {\em Pacific J. Math.}, 239(1):39--52, 2009.

\bibitem[Kap93a]{kap1}
M.~M. Kapranov.
\newblock Chow quotients of {G}rassmannians. {I}.
\newblock In {\em I. {M}. {G}el\cprime fand {S}eminar}, volume~16 of {\em Adv.
  Soviet Math.}, pages 29--110. Amer. Math. Soc., Providence, RI, 1993.

\bibitem[Kap93b]{kap2}
M.~M. Kapranov.
\newblock Veronese curves and {G}rothendieck-{K}nudsen moduli space {$\overline
  M_{0,n}$}.
\newblock {\em J. Algebraic Geom.}, 2(2):239--262, 1993.

\bibitem[Kat91]{katsylo-5}
P.~I. Katsylo.
\newblock Rationality of the variety of moduli of curves of genus {$5$}.
\newblock {\em Mat. Sb.}, 182(3):457--464, 1991.

\bibitem[Kat94]{katsylo-prym}
P.~I. Katsylo.
\newblock On the unramified {$2$}-covers of the curves of genus {$3$}.
\newblock In {\em Algebraic geometry and its applications ({Y}aroslavl\cprime,
  1992)}, Aspects Math., E25, pages 61--65. Vieweg, Braunschweig, 1994.

\bibitem[Kat96]{katsylo-3}
P.~Katsylo.
\newblock Rationality of the moduli variety of curves of genus {$3$}.
\newblock {\em Comment. Math. Helv.}, 71(4):507--524, 1996.

\bibitem[Kee99]{keel-annals}
Se{\'a}n Keel.
\newblock Basepoint freeness for nef and big line bundles in positive
  characteristic.
\newblock {\em Ann. of Math. (2)}, 149(1):253--286, 1999.

\bibitem[Kem78]{kempf}
George~R. Kempf.
\newblock Instability in invariant theory.
\newblock {\em Ann. of Math. (2)}, 108(2):299--316, 1978.

\bibitem[Kho05]{khosla}
Deepak Khosla.
\newblock Moduli spaces of curves with linear series and the slope conjecture,
  2005.
\newblock Ph.D. thesis, Harvard University, Cambridge, MA.

\bibitem[Kho07]{khosla2}
Deepak Khosla.
\newblock Tautological classes on moduli spaces of curves with linear series
  and a push-forward formula when $\rho=0$.
\newblock {\tt arXiv:0704.1340 [math.AG]}, 2007.

\bibitem[KM96]{KMc}
Se{\'a}n Keel and James McKernan.
\newblock Contractible extremal rays on {$\M_{0,n}$}, 1996.
\newblock arXiv: 9607.009.

\bibitem[KM97]{keel-mori}
Se{\'a}n Keel and Shigefumi Mori.
\newblock Quotients by groupoids.
\newblock {\em Ann. of Math. (2)}, 145(1):193--213, 1997.

\bibitem[KM98]{kollar-mori}
J{\'a}nos Koll{\'a}r and Shigefumi Mori.
\newblock {\em Birational geometry of algebraic varieties}, volume 134 of {\em
  Cambridge Tracts in Mathematics}.
\newblock Cambridge University Press, Cambridge, 1998.
\newblock With the collaboration of C. H. Clemens and A. Corti, Translated from
  the 1998 Japanese original.

\bibitem[Kol90]{kollar-projectivity}
J{\'a}nos Koll{\'a}r.
\newblock Projectivity of complete moduli.
\newblock {\em J. Differential Geom.}, 32(1):235--268, 1990.

\bibitem[Kon00]{kondo-3}
Shigeyuki Kond{\=o}.
\newblock A complex hyperbolic structure for the moduli space of curves of
  genus three.
\newblock {\em J. Reine Angew. Math.}, 525:219--232, 2000.

\bibitem[Kon02]{kondo-4}
Shigeyuki Kond{\=o}.
\newblock The moduli space of curves of genus 4 and {D}eligne-{M}ostow's
  complex reflection groups.
\newblock In {\em Algebraic geometry 2000, {A}zumino ({H}otaka)}, volume~36 of
  {\em Adv. Stud. Pure Math.}, pages 383--400. Math. Soc. Japan, Tokyo, 2002.

\bibitem[{Lar}09]{larsen}
P.~{Larsen}.
\newblock Fulton's conjecture for $\overline{M}_{0,7}$.
\newblock {\tt arXiv:0912.3104 [math.AG]}, 2009.

\bibitem[Laz04]{Laz1}
Robert Lazarsfeld.
\newblock {\em Positivity in algebraic geometry. {I}}, volume~48 of {\em
  Ergebnisse der Mathematik und ihrer Grenzgebiete. 3. Folge. A Series of
  Modern Surveys in Mathematics}.
\newblock Springer-Verlag, Berlin, 2004.
\newblock Classical setting: line bundles and linear series.

\bibitem[{Laz}09]{lazic}
V.~{Lazic}.
\newblock {Adjoint rings are finitely generated}.
\newblock {\tt arXiv:0905.2707 [math.AG]}, 2009.

\bibitem[LMB00]{LM}
G{\'e}rard Laumon and Laurent Moret-Bailly.
\newblock {\em Champs alg\'ebriques}, volume~39 of {\em Ergebnisse der
  Mathematik und ihrer Grenzgebiete. 3. Folge. A Series of Modern Surveys in
  Mathematics}.
\newblock Springer-Verlag, Berlin, 2000.

\bibitem[Log03]{logan-kodaira}
Adam Logan.
\newblock The {K}odaira dimension of moduli spaces of curves with marked
  points.
\newblock {\em Amer. J. Math.}, 125(1):105--138, 2003.

\bibitem[MFK94]{GIT}
D.~Mumford, J.~Fogarty, and F.~Kirwan.
\newblock {\em Geometric invariant theory}, volume~34 of {\em Ergebnisse der
  Mathematik und ihrer Grenzgebiete (2)}.
\newblock Springer-Verlag, Berlin, third edition, 1994.

\bibitem[MM83]{mori-mukai-11}
Shigefumi Mori and Shigeru Mukai.
\newblock The uniruledness of the moduli space of curves of genus {$11$}.
\newblock In {\em Algebraic geometry ({T}okyo/{K}yoto, 1982)}, volume 1016 of
  {\em Lecture Notes in Math.}, pages 334--353. Springer, Berlin, 1983.

\bibitem[Mor98]{moriwaki1}
Atsushi Moriwaki.
\newblock Relative {B}ogomolov's inequality and the cone of positive divisors
  on the moduli space of stable curves.
\newblock {\em J. Amer. Math. Soc.}, 11(3):569--600, 1998.

\bibitem[{Mor}07]{morrison-mori}
I.~{Morrison}.
\newblock Mori theory of moduli spaces of stable curves.
\newblock 2007.
\newblock Projective Press, New York.

\bibitem[Mor09]{morrison-git}
Ian Morrison.
\newblock G{IT} constructions of moduli spaces of stable curves and maps.
\newblock In {\em Surveys in differential geometry. {V}ol. {XIV}. {G}eometry of
  {R}iemann surfaces and their moduli spaces}, volume~14 of {\em Surv. Differ.
  Geom.}, pages 315--369. Int. Press, Somerville, MA, 2009.

\bibitem[MS09]{morrison-swinarski}
I.~{Morrison} and D.~{Swinarski}.
\newblock {Groebner techniques for low degree Hilbert stability}.
\newblock J. Exp. Math. , to appear. {\tt arXiv:0910.2047 [math.AG]}, 2009.

\bibitem[Mum66]{mumford-curves}
David Mumford.
\newblock {\em Lectures on curves on an algebraic surface}.
\newblock With a section by G. M. Bergman. Annals of Mathematics Studies, No.
  59. Princeton University Press, Princeton, N.J., 1966.

\bibitem[Mum77]{mumford-stability}
David Mumford.
\newblock Stability of projective varieties.
\newblock {\em Enseignement Math. (2)}, 23(1-2):39--110, 1977.

\bibitem[Nam73]{namikawa}
Yukihiko Namikawa.
\newblock On the canonical holomorphic map from the moduli space of stable
  curves to the {I}gusa monoidal transform.
\newblock {\em Nagoya Math. J.}, 52:197--259, 1973.

\bibitem[{Pan}]{pand-slopes}
Rahul {Pandharipande}.
\newblock {Descendent bounds for effective divisors on $\overline{M}_{g}$}.
\newblock {\em J. Algebraic Geom.}
\newblock PII: S 1056-3911(2010)00554-1 (to appear in print).

\bibitem[Rei80]{reid3folds}
Miles Reid.
\newblock Canonical {$3$}-folds.
\newblock In {\em Journ\'ees de {G}\'eometrie {A}lg\'ebrique d'{A}ngers,
  {J}uillet 1979/{A}lgebraic {G}eometry, {A}ngers, 1979}, pages 273--310.
  Sijthoff \& Noordhoff, Alphen aan den Rijn, 1980.

\bibitem[Rul01]{rulla}
William Rulla.
\newblock The birational geometry of {$\overline{M}_{2,1}$} and
  {$\overline{M}_3$}, 2001.
\newblock Ph.D. thesis, University of Texas at Austin.

\bibitem[SB87]{barron-4}
N.~I. Shepherd-Barron.
\newblock The rationality of certain spaces associated to trigonal curves.
\newblock In {\em Algebraic geometry, {B}owdoin, 1985 ({B}runswick, {M}aine,
  1985)}, volume~46 of {\em Proc. Sympos. Pure Math.}, pages 165--171. Amer.
  Math. Soc., Providence, RI, 1987.

\bibitem[SB89]{barron-6}
N.~I. Shepherd-Barron.
\newblock Invariant theory for {$S_5$} and the rationality of {$M_6$}.
\newblock {\em Compositio Math.}, 70(1):13--25, 1989.

\bibitem[Sch91]{Schubert}
David Schubert.
\newblock A new compactification of the moduli space of curves.
\newblock {\em Compositio Math.}, 78(3):297--313, 1991.

\bibitem[Ser81]{sernesi12}
Edoardo Sernesi.
\newblock Unirationality of the variety of moduli of curves of genus twelve.
\newblock {\em Ann. Scuola Norm. Sup. Pisa Cl. Sci. (4)}, 8(3):405--439, 1981.

\bibitem[Sev21]{severi-plane}
Francesco Severi.
\newblock {\em Vorlesungen \"{u}ber {A}lgebraische {G}eometrie}.
\newblock Teubner, Leipzig, 1921.

\bibitem[Siu08]{siu}
Yum-Tong Siu.
\newblock Finite generation of canonical ring by analytic method.
\newblock {\em Sci. China Ser. A}, 51(4):481--502, 2008.

\bibitem[Smy09]{Z-stability}
David Smyth.
\newblock Towards a classification of modular compactifications of the moduli
  space of curves.
\newblock With an appendix by Jack Hall. Submitted. {\tt arXiv:0902.3690v2
  [math.AG]}, 2009.

\bibitem[Smy10]{SmythEII}
David Smyth.
\newblock Modular compactifications of the space of pointed elliptic curves
  \text{II}.
\newblock Compositio Math., to appear. {\tt arXiv:1005.1083v1 [math.AG]}, 2010.

\bibitem[Smy11]{SmythEI}
David~Ishii Smyth.
\newblock Modular compactifications of the space of pointed elliptic curves
  {I}.
\newblock {\em Compositio Math.}, 147(03):877--913, 2011.

\bibitem[Tai82]{tai-kodaira}
Yung-Sheng Tai.
\newblock On the {K}odaira dimension of the moduli space of abelian varieties.
\newblock {\em Invent. Math.}, 68(3):425--439, 1982.

\bibitem[Ver84]{verra-prym-6}
Alessandro Verra.
\newblock A short proof of the unirationality of {${\mathcal A}_5$}.
\newblock {\em Nederl. Akad. Wetensch. Indag. Math.}, 46(3):339--355, 1984.

\bibitem[Ver02]{vermeire}
Peter Vermeire.
\newblock A counterexample to {F}ulton's conjecture on {$\overline M_{0,n}$}.
\newblock {\em J. Algebra}, 248(2):780--784, 2002.

\bibitem[Ver05]{verra14}
Alessandro Verra.
\newblock The unirationality of the moduli spaces of curves of genus 14 or
  lower.
\newblock {\em Compos. Math.}, 141(6):1425--1444, 2005.

\bibitem[{Ver}08]{verra-prym-5}
A.~{Verra}.
\newblock On the universal abelian variety of dimension 4.
\newblock {\em Contemporary Math.}, 465, 2008.
\newblock {\tt Arxiv:0711.3890 [math.AG]}.

\end{thebibliography}
\bibliographystyle{alpha}
\end{document}